\documentclass{amsart} % AMS math article

\usepackage{amsmath,amssymb,amsthm} % formulas, mathematical symbols, theorems, etc.
\usepackage{graphicx} % for graphics
\usepackage{subfig}
\usepackage[arrow, matrix, curve]{xy} % for commutative diagrams
\usepackage{xcolor} % colors
\usepackage{makecell} % for linebreak in tables
\usepackage{multirow} % for multiple rows in tables
\usepackage{enumerate}
\usepackage{booktabs} % for pretty tables
\usepackage{siunitx} % for decimal alignment
\usepackage{hyperref} % for links
\DeclareMathOperator{\tb}{tb}
\DeclareMathOperator{\rot}{rot}
\DeclareMathOperator{\lk}{lk}

\DeclareMathOperator{\de}{d}
\DeclareMathOperator{\e}{e}

\DeclareMathOperator{\cs}{cs}

\newcommand{\R}{\mathbb{R}}
\newcommand{\Z}{\mathbb{Z}}
\newcommand{\Q}{\mathbb{Q}}

\newcommand{\RP}{\mathbb{RP}^3}

\newcommand{\xist}{\xi_{\mathrm{st}}}

\newcommand{\Pushoff}{{\def\svgwidth{1,6ex}\,\,%% Creator: Inkscape 0.91_64bit, www.inkscape.org
\begingroup%
  \makeatletter%
  \providecommand\color[2][]{%
    \errmessage{(Inkscape) Color is used for the text in Inkscape, but the package 'color.sty' is not loaded}%
    \renewcommand\color[2][]{}%
  }%
  \providecommand\transparent[1]{%
    \errmessage{(Inkscape) Transparency is used (non-zero) for the text in Inkscape, but the package 'transparent.sty' is not loaded}%
    \renewcommand\transparent[1]{}%
  }%
  \providecommand\rotatebox[2]{#2}%
  \ifx\svgwidth\undefined%
    \setlength{\unitlength}{49.54732089bp}%
    \ifx\svgscale\undefined%
      \relax%
    \else%
      \setlength{\unitlength}{\unitlength * \real{\svgscale}}%
    \fi%
  \else%
    \setlength{\unitlength}{\svgwidth}%
  \fi%
  \global\let\svgwidth\undefined%
  \global\let\svgscale\undefined%
  \makeatother%
  \begin{picture}(1,1.00059637)%
    \put(0,0){\includegraphics[width=\unitlength,page=1]{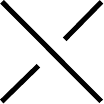}}%
  \end{picture}%
\endgroup%
\,\,}} % Definition Legendrian Push-Off; VORSICHT: Funktioniert nicht in Bildunterschriften

\usepackage{amsthm}
\usepackage{thmtools}
\declaretheoremstyle[notefont=\bfseries,notebraces={}{},%
    headpunct={},postheadspace=1em]{mystyle}
\declaretheorem[style=mystyle,numbered=no,name=Theorem]{thm-hand}

\newtheoremstyle{theorem}{}{}{\itshape}{}{\bfseries}{}{ }{} %Thereom style
\newtheoremstyle{definition}{}{}{}{}{\bfseries}{}{ }{} %Definition style

\theoremstyle{theorem}
\newtheorem{Theorem}{Theorem}[section]
\newtheorem{theorem}[Theorem]{Theorem}
\newtheorem{lem}[Theorem]{Lemma}

\newtheorem{cor}[Theorem]{Corollary}

\newtheorem{question}[Theorem]{Question}

\theoremstyle{definition}

\newtheorem{remark}[Theorem]{Remark}

\begingroup\expandafter\expandafter\expandafter\endgroup
\expandafter\ifx\csname pdfsuppresswarningpagegroup\endcsname\relax
\else
\fi

\definecolor{amaranth}{rgb}{0.9, 0.17, 0.31} %dark red
\definecolor{carrotorange}{rgb}{0.93, 0.57, 0.13} %orange
\definecolor{citrine}{rgb}{0.89, 0.82, 0.04} %dark yellow
\definecolor{dartmouthgreen}{rgb}{0.05, 0.5, 0.06} %green
\definecolor{ballblue}{rgb}{0.13, 0.67, 0.8} %blue
\definecolor{ceruleanblue}{rgb}{0.16, 0.32, 0.75} %deeper blue
\definecolor{amethyst}{rgb}{0.6, 0.4, 0.8} %purple
\definecolor{amber}{rgb}{1.0, 0.75, 0.0} %amber
\definecolor{burlywood}{rgb}{0.87, 0.72, 0.53} %beigebrown
\numberwithin{equation}{section}
\begin{document}

%%%%%%%%%%%%%%%%%%%%%%%%%%%%% Title and authors %%%%%%%%%%%%%%%%%%%%%%%%%%%%%%%%%%%%
\title[Contact surgery numbers of projective spaces]{Contact surgery numbers of projective spaces}

\author{Marc Kegel}
\address{Humboldt Universit\"at zu Berlin, Rudower Chaussee 25, 12489 Berlin, Germany}
\email{kegelmarc87@gmail.com}

\author{Monika Yadav}
\address{Statistics and Mathematics Unit, Indian Statistical Institute, Kolkata,
	India}
 \email{monika.yadav9413@gmail.com}

%%%%%%%%%%%%%%%%%%%%%%%%%%%%% Abstract %%%%%%%%%%%%%%%%%%%%%%%%%%%%%%%%%%%%

\date{\today} % date on first page

\begin{abstract}

\end{abstract}

\keywords{Contact surgery numbers, Legendrian knots, $\de_3$-invariant, $\Gamma$-invariant}

\begin{abstract}
We classify all contact projective spaces with contact surgery number one. In particular, this implies that there exist infinitely many non-isotopic contact structures on the real projective $3$-space which cannot be obtained by a single rational contact surgery from the standard tight contact $3$-sphere.

Large parts of our proofs deal with a detailed analysis of Gompf's $\Gamma$-invariant of tangential $2$-plane fields on $3$-manifolds. From our main result we also deduce that the $\Gamma$-invariant of a tangential $2$-plane field on the real projective $3$-space only depends on its $\de_3$-invariant.
\end{abstract}

\makeatletter
\@namedef{subjclassname@2020}{%
  \textup{2020} Mathematics Subject Classification}
\makeatother%For 2020

\subjclass[2020]{53D35; 53D10, 57K10, 57R65, 57K33} % Mathematical subject classification

\maketitle

%\tableofcontents
%%%%%%%%%%%%%%%%%%%%%%%%%%%%%%%%%%%%%%%%%% INTRODUCTION %%%%%%%%%%%%%%%%%%%%%%%%%%%%%

\section{Introduction}

A central result in $3$-dimensional contact topology, due to Ding and Geiges, states that any connected, oriented, closed contact $3$-manifold with a co-orientable, positive contact structure can be obtained by contact surgery on a Legendrian link in the standard tight contact $3$-sphere $(\mathbb S^3,\xist)$~\cite{Ding_Geiges_Surgery}. Moreover, all contact surgery coefficients can be assumed to be $\pm1$. This leads to a natural complexity measure for a given contact $3$-manifold, the minimal number of components required for the surgery link.
The \textit{contact surgery number} $\cs(M, \xi)$ of a contact $3$-manifold $(M,\xi)$ is defined as the minimal number of components of a Legendrian link $L$ in $(\mathbb{S}^3, \xist)$ needed to obtain $(M,\xi)$ via rational contact surgery along $L$ (with nonzero contact surgery coefficients). Variants of this notion include $\cs_\Z(M, \xi)$, $\cs_{1/\Z}(M, \xi)$, and $\cs_{\pm1}(M, \xi)$, where the surgery coefficients are required to be integers, reciprocal integers, or $\pm1$, respectively.

The study of contact surgery numbers was initiated in~\cite{EKS_contact_surgery_numbers,CK_contact_surgery_numbers}, where explicit calculations were performed for simple manifolds, and general upper bounds were established. Notably, it was shown that the contact surgery number of a contact manifold $(M,\xi)$ is at most three more than the topological surgery number of the underlying smooth manifold $M$. 

In this paper, we extend the study of contact surgery numbers by classifying all contact structures on the real projective $3$-space $\RP$ with contact surgery number one. To state the result we recall the classification of contact structures on $\RP$, for further details we refer to Section~\ref{sec:invariants} and~\ref{sec:computing}. By~\cite{Etnyre_lens,Honda_lens,Giroux_lens} there exists a unique tight contact structure $\xist$ on $\RP$, that is obtained as the quotient of $(\mathbb{S}^3,\xist)$. The overtwisted contact structures are determined by the underlying tangential $2$-plane fields~\cite{Eliashberg_OT}, which are completely classified by the two \textit{homotopical invariants}, Gompf's $\Gamma$-invariant and the $\de_3$-invariant~\cite{Gompf_Stein}.\footnote{Note, that in this article we use a normalization of the $\de_3$-invariant that differs from the one used for example in~\cite{Gompf_Stein}. We refer to the paragraph on our conventions at the end of this introduction for more details on our normalization.} Roughly speaking the $\de_3$-invariant $\de_3(\xi)$ is a rational number that determines a tangential $2$-plane field $\xi$ on the $3$-cells, while the $\Gamma$-invariant $\Gamma(\xi,\mathfrak{s})\in H_1(\RP)\cong\Z_2$ depends also on the choice of a spin structure $\mathfrak{s}$ and determines $\xi$ on the $2$-skeleton. 
To effectively compare the $\Gamma$ invariants, we consider the \textit{standard} surgery diagram of $\RP$ consisting of a single unknot with topological surgery coefficient $-2$. In that surgery diagram, the empty link is a characteristic sublink defining a spin structure $\mathfrak{s}_0$. Then we define $\Gamma(\xi)\in\Z_2$ to be $\Gamma(\xi,\mathfrak{s}_0)\in H_1(\RP)\cong\Z_2$. 

\begin{theorem}\label{thm:main}\hfill
\begin{enumerate}
    \item Any contact structure on $\RP$ has $\cs_{\pm1}\leq3$.
    \item The tight contact structure $\xist$ on $\RP$ has $\cs_{\pm1}=\cs_{1/\Z}=\cs_{\Z}=\cs=1$.
    \item An overtwisted contact structure on $\RP$ has $\cs_{\pm1}=1$ if and only if it has $\cs_{1/\Z}=1$ if and only if its pair $(\Gamma,\de_3)$ of $\Gamma$- and $\de_3$-invariants is equal to $(0,1+\frac{1}{4})$ or $(1,\frac{3}{4})$.
    \item An overtwisted contact structure on $\RP$ has $\cs_{\Z}=1$ if and only if its pair $(\Gamma,\de_3)$ of $\Gamma$- and $\de_3$-invariants is $\left(0, 1+\frac{1}{4}\right)$ or $\left(1,\frac{3}{4}\right)$ or there exists an integer $m\leq-1$ such that it is of the form
    \begin{align*}
        &\left(0,-2m^2-4m-1+\frac{1}{4}\right),\, \left(0,2m^2-2m+\frac{1}{4}\right),\\
        &\left(1,-2m^2-6m-4+\frac{3}{4}\right),\,\textrm{ or }
         \left(1,2m^2-1+\frac{3}{4}\right).
    \end{align*}
    \item An overtwisted contact structure on $\RP$ has $\cs=1$ if and only if its pair $(\Gamma,\de_3)$ of $\Gamma$- and $\de_3$-invariants takes one of the values given in Cases $(2)$--$(11)$ of Table~\ref{tab:tab2}.
\end{enumerate}
\end{theorem}

Here we state some of the corollaries from our main theorem. First, we observe that there exist infinitely many overtwisted contact structures on $\RP$ that cannot be obtained by a single rational contact Dehn surgery along a single Legendrian knot in $(\mathbb{S}^3,\xist)$.

\begin{cor}\label{cor:cs2}
For both values of $\Gamma\in\Z_2$, there exist infinitely many non-contacto\-mor\-phic, overtwisted contact structures on $\RP$ with $\Gamma$-invariant $\Gamma$ that have $\cs=2$. 
\end{cor}

We also observe that there exist certain contact structures with unique single-component contact surgery diagrams.

\begin{cor}\label{cor:unique_diag1}
    If $K$ is a Legendrian knot in $(\mathbb{S}^3,\xist)$ such that contact $(-1)$- or contact $(+1)$-surgery along $K$ yields a contact structure $\xi$ on $\RP$, then we are exactly in one of the following three cases.
     \begin{itemize}
        \item $K$ is isotopic to the unique Legendrian realization of the unknot with $\tb=-1$ and $\rot=0$, the contact surgery coefficient is $-1$, and $\xi$ is contactomorphic to the standard tight contact structure $\xist$. 
        \item $K$ is isotopic to the unique Legendrian realization of the unknot with $\tb=-3$ and $\rot=0$, the contact surgery coefficient is $+1$, and $\xi$ is the overtwisted contact structure with $\Gamma(\xi)=0$ and $\de_3(\xi)=1+\frac{1}{4}$. 
        \item $K$ is isotopic to the unique Legendrian realization of the unknot with $\tb=-3$ and $|\rot|=2$, the contact surgery coefficient is $+1$, and $\xi$ is the overtwisted contact structure with $\Gamma(\xi)=1$ and $\de_3(\xi)=\frac{3}{4}$. 
    \end{itemize}
\end{cor}

\begin{cor}\label{cor:unique_diag2}
    If $K$ is a Legendrian knot in $(\mathbb{S}^3,\xist)$ such that for some integer $k\in\Z-\{-1,0,1\}$ contact $(1/k)$-surgery along $K$ yields a contact structure $\xi$ on $\RP$, then $K$ is isotopic to the unique Legendrian realization of the unknot with $\tb=-1$ and $\rot=0$, the contact surgery coefficient is $1/3$, and $\xi$ is the overtwisted contact structure with $\Gamma(\xi)=1$ and $\de_3(\xi)=\frac{3}{4}$. 
\end{cor}

From our main result and the values from Table~\ref{tab:tab2} we will also deduce the surprising result that on $\RP$ the $\Gamma$-invariant of a tangential $2$-plane field is in fact determined by its $\de_3$-invariant. More precisely, we will show the following.

\begin{cor}\label{cor:Gamma}
    Let $\xi$ be a tangential $2$-plane field on $\RP$. Then we have
    \begin{enumerate}
        \item $\de_3(\xi)\in\Z+\frac{1}{4}$ or $\de_3(\xi)\in\Z+\frac{3}{4}$.
        \item $\de_3(\xi)\in\Z+\frac{1}{4}$ if and only if $\Gamma(\xi)=0$. 
        \item $\de_3(\xi)\in\Z+\frac{3}{4}$ if and only if $\Gamma(\xi)=1$.
        \item Conversely, for every pair $(\Gamma,d)$ with either $\Gamma=0$ and $d\in\Z+\frac{1}{4}$ or $\Gamma=1$ and $d\in\Z+\frac{3}{4}$ there exists a tangential $2$-plane field $\xi$ on $\RP$, unique up to homotopy of tangential $2$-plane fields, such that $(\Gamma(\xi),\de_3(\xi))=(\Gamma,d)$. 
    \end{enumerate}
\end{cor}

We wonder the following question.

\begin{question}
    On which other rational homology $3$-spheres is the $\Gamma$-invariant (or Euler class) of a tangential $2$-plane field determined by its $\de_3$-invariant?\footnote{Add in proof: As part of forthcoming joint work with Isacco Nonino we will answer this question in detail. A similar approach was suggested to us by the referee.}
\end{question}

\subsection{Outline of the arguments}

To prove Theorem \ref{thm:main}, we first deduce from the main result of~\cite{KMOS} that if $K$ is a knot in $\mathbb S^3$ such that topological $r$-surgery on $K$ yields a manifold diffeomorphic to $\RP$, then $K$ is the unknot and $r=\frac{2}{2n+1}$ for some $n\in \mathbb{Z}$. This implies that a contact structure on $\RP$ has contact surgery number $1$ if and only if it admits a contact surgery diagram consisting of a single Legendrian unknot with an appropriate contact surgery coefficient. Thus, by varying $n$ over $\mathbb{Z}$ and considering all Legendrian unknots~\cite{Eliashberg_Fraser}, we generate a complete list of contact surgery diagrams along single Legendrian knots that yield contact structures on $\RP$. Once we obtain these contact surgery diagrams, we compute their homotopical invariants, which proves Theorem~\ref{thm:main}. By carefully comparing the values of these invariants, we derive the above corollaries.
   
\subsection*{Conventions}
Throughout this paper, all contact structures are assumed to be positive and coorientable. Legendrian knots in $(\mathbb S^3,\xist)$ are always presented in their front projection. We write $t$ and $r$ for the Thurston--Bennequin invariant and the rotation number of an oriented Legendrian knot in $(\mathbb S^3,\xist)$. Since a contact surgery diagram determines a contact manifold only up to contactomorphism, we consider contact manifolds up to contactomorphism rather than isotopy. We write $\cong$ to denote a contactomorphism between two contact manifolds. Moreover, the contactomorphism type of a contact surgery does not depend on the orientation of the Legendrian surgery link. Nevertheless, we primarily work with oriented Legendrian links in $(\mathbb{S}^3, \xist)$, since then certain invariants are easier to compute. We normalize the $\de_3$-invariant such that $\de_3(\mathbb{S}^3, \xist) = 0$, following conventions in~\cite{casals2021stein, EKS_contact_surgery_numbers, surgery_graph,CK_contact_surgery_numbers}. With this normalization, the $\de_3$-invariant is additive under connected sum and takes integer values on homology spheres. In our normalization, the $\de_3$-invariant is defined via
$$d_3(Y,\xi):= \frac{1}{4}\big(c_1(X,J)^2 -3\sigma(X)-2\chi(X)\big)+ \frac{1}{2},$$
where $(X, J )$ is a compact almost-complex $4$-manifold whose boundary is $Y$ and such that $\xi$ is induced by the almost complex structure $J$. This differs from the original sources, such as~\cite{Gompf_Stein,Ding_Geiges_Stipsicz}.

\subsection*{Acknowledgments} 

This work began during a visit of MY at the Ruhr-Uni\-ver\-si\-ty Bochum. We thank Kai Zehmisch for the financial support that made the stay possible. Large parts of this project were carried out when MY visited the Humboldt University Berlin via a \textit{WINS postdoctoral fellowship} and a \textit{Math+ Hanna Neumann Fellowship}. MK is funded by the DFG, German Research Foundation, (Project: 561898308).

\section{Contact Dehn surgery}

In this section, we provide background on contact Dehn surgery along Legendrian knots. For further details, we refer to~\cite{Geiges_book,Gompf_Stein,Ding_Geiges_Surgery,Ding_Geiges_Stipsicz,Ozbagci_Stipsicz_book,Durst_Kegel_rot_surgery,phdthesis,Legendrian_knot_complement,casals2021stein,EKS_contact_surgery_numbers}.

 Let $K$ be a Legendrian knot in $(\mathbb{S}^3, \xist)$. A \textit{contact Dehn surgery} along $K$ with \textit{contact surgery coefficient} ${p}/{q} \in \mathbb{Q} - \{0\}$ is performed by removing a tubular neighborhood of $K$ and attaching a solid torus $\mathbb{S}^1 \times \mathbb{D}^2$ via a diffeomorphism that maps $\{pt\} \times \partial \mathbb{D}^2$ to $p\mu + q\lambda_c$. Here, $\mu$ denotes the meridian of $K$, and the \textit{contact longitude} $\lambda_c$ is the Legendrian knot obtained by pushing $K$ in the Reeb direction. The resulting $3$-manifold carries natural contact structures that coincide with $\xist$ outside the removed neighborhood and are tight on the newly glued-in solid torus. It was shown in~\cite{Honda_lens,Giroux_lens} that such a contact structure always exists, and is unique if $p=\pm1$. For more general contact surgery coefficients, there are only finitely many such contact structures, with the exact number depending on the continued fraction expansion of $p/q$ (see Lemma~\ref{lem:Kirby} below). We denote by $K({p}/{q})$ the surgered manifold with one such contact structure.

 The \textit{Seifert longitude} $\lambda_s$, obtained by pushing $K$ into a Seifert surface, satisfies the relation $\lambda_c=\lambda_s+\tb(K)\mu$
where $\tb(K)$ is the Thurston–Bennequin invariant of $K$. This implies that the \textit{topological surgery coefficient} $r_t$, which is measured with respect to the Seifert longitude $\lambda_s$, and the contact surgery coefficient $r_c$ are related by $r_c=r_t-\tb(K)$. 

Next, we introduce some useful notation which will be used throughout this article. For $m,n\in \mathbb{N}_0$, let $K_n$ denote a Legendrian knot which is obtained by adding $n$ stabilizations to $K$ and further $K_{n,m}$ denotes a Legendrian knot which is obtained by adding $m$ extra stabilizations to $K_n$. Let $K{\def\svgwidth{1,6ex}\,\,%% Creator: Inkscape 0.91_64bit, www.inkscape.org
%% PDF/EPS/PS + LaTeX output extension by Johan Engelen, 2010
%% Accompanies image file 'PushOff.pdf' (pdf, eps, ps)
%%
%% To include the image in your LaTeX document, write
%%   \input{<filename>.pdf_tex}
%%  instead of
%%   \includegraphics{<filename>.pdf}
%% To scale the image, write
%%   \def\svgwidth{<desired width>}
%%   \input{<filename>.pdf_tex}
%%  instead of
%%   \includegraphics[width=<desired width>]{<filename>.pdf}
%%
%% Images with a different path to the parent latex file can
%% be accessed with the `import' package (which may need to be
%% installed) using
%%   \usepackage{import}
%% in the preamble, and then including the image with
%%   \import{<path to file>}{<filename>.pdf_tex}
%% Alternatively, one can specify
%%   \graphicspath{{<path to file>/}}
%% 
%% For more information, please see info/svg-inkscape on CTAN:
%%   http://tug.ctan.org/tex-archive/info/svg-inkscape
%%
\begingroup%
  \makeatletter%
  \providecommand\color[2][]{%
    \errmessage{(Inkscape) Color is used for the text in Inkscape, but the package 'color.sty' is not loaded}%
    \renewcommand\color[2][]{}%
  }%
  \providecommand\transparent[1]{%
    \errmessage{(Inkscape) Transparency is used (non-zero) for the text in Inkscape, but the package 'transparent.sty' is not loaded}%
    \renewcommand\transparent[1]{}%
  }%
  \providecommand\rotatebox[2]{#2}%
  \ifx\svgwidth\undefined%
    \setlength{\unitlength}{49.54732089bp}%
    \ifx\svgscale\undefined%
      \relax%
    \else%
      \setlength{\unitlength}{\unitlength * \real{\svgscale}}%
    \fi%
  \else%
    \setlength{\unitlength}{\svgwidth}%
  \fi%
  \global\let\svgwidth\undefined%
  \global\let\svgscale\undefined%
  \makeatother%
  \begin{picture}(1,1.00059637)%
    \put(0,0){\includegraphics[width=\unitlength,page=1]{PushOff.pdf}}%
  \end{picture}%
\endgroup%
\,\,} K$ denote the Legendrian link consisting of $K$ and a Legendrian knot obtained by pushing $K$ in the Reeb direction (i.e.\ the contact longitude). 

\begin{lem} [Ding--Geiges~\cite{Ding_Geiges_torus_bundles,Ding_Geiges_Surgery}]\label{lem:Kirby}
Let $K$ be a Legendrian knot in $(\mathbb{S}^3,\xi_{st})$. 
\begin{enumerate}
    \item \textbf{Cancellation lemma:} For all $n\in \mathbb{Z}-\{0\}$, we have
    $$K\left(\frac{1}{n}\right){\def\svgwidth{1,6ex}\,\,%% Creator: Inkscape 0.91_64bit, www.inkscape.org
%% PDF/EPS/PS + LaTeX output extension by Johan Engelen, 2010
%% Accompanies image file 'PushOff.pdf' (pdf, eps, ps)
%%
%% To include the image in your LaTeX document, write
%%   \input{<filename>.pdf_tex}
%%  instead of
%%   \includegraphics{<filename>.pdf}
%% To scale the image, write
%%   \def\svgwidth{<desired width>}
%%   \input{<filename>.pdf_tex}
%%  instead of
%%   \includegraphics[width=<desired width>]{<filename>.pdf}
%%
%% Images with a different path to the parent latex file can
%% be accessed with the `import' package (which may need to be
%% installed) using
%%   \usepackage{import}
%% in the preamble, and then including the image with
%%   \import{<path to file>}{<filename>.pdf_tex}
%% Alternatively, one can specify
%%   \graphicspath{{<path to file>/}}
%% 
%% For more information, please see info/svg-inkscape on CTAN:
%%   http://tug.ctan.org/tex-archive/info/svg-inkscape
%%
\begingroup%
  \makeatletter%
  \providecommand\color[2][]{%
    \errmessage{(Inkscape) Color is used for the text in Inkscape, but the package 'color.sty' is not loaded}%
    \renewcommand\color[2][]{}%
  }%
  \providecommand\transparent[1]{%
    \errmessage{(Inkscape) Transparency is used (non-zero) for the text in Inkscape, but the package 'transparent.sty' is not loaded}%
    \renewcommand\transparent[1]{}%
  }%
  \providecommand\rotatebox[2]{#2}%
  \ifx\svgwidth\undefined%
    \setlength{\unitlength}{49.54732089bp}%
    \ifx\svgscale\undefined%
      \relax%
    \else%
      \setlength{\unitlength}{\unitlength * \real{\svgscale}}%
    \fi%
  \else%
    \setlength{\unitlength}{\svgwidth}%
  \fi%
  \global\let\svgwidth\undefined%
  \global\let\svgscale\undefined%
  \makeatother%
  \begin{picture}(1,1.00059637)%
    \put(0,0){\includegraphics[width=\unitlength,page=1]{PushOff.pdf}}%
  \end{picture}%
\endgroup%
\,\,} K\left(-\frac{1}{n}\right)\cong (\mathbb{S}^3,\xi_{st}).$$
    \item \textbf{Replacement lemma:} For all $n\in \mathbb{Z}-\{0\}$, we have
    $$K\left(\pm\frac{1}{n}\right)\cong K(\pm 1){\def\svgwidth{1,6ex}\,\,%% Creator: Inkscape 0.91_64bit, www.inkscape.org
%% PDF/EPS/PS + LaTeX output extension by Johan Engelen, 2010
%% Accompanies image file 'PushOff.pdf' (pdf, eps, ps)
%%
%% To include the image in your LaTeX document, write
%%   \input{<filename>.pdf_tex}
%%  instead of
%%   \includegraphics{<filename>.pdf}
%% To scale the image, write
%%   \def\svgwidth{<desired width>}
%%   \input{<filename>.pdf_tex}
%%  instead of
%%   \includegraphics[width=<desired width>]{<filename>.pdf}
%%
%% Images with a different path to the parent latex file can
%% be accessed with the `import' package (which may need to be
%% installed) using
%%   \usepackage{import}
%% in the preamble, and then including the image with
%%   \import{<path to file>}{<filename>.pdf_tex}
%% Alternatively, one can specify
%%   \graphicspath{{<path to file>/}}
%% 
%% For more information, please see info/svg-inkscape on CTAN:
%%   http://tug.ctan.org/tex-archive/info/svg-inkscape
%%
\begingroup%
  \makeatletter%
  \providecommand\color[2][]{%
    \errmessage{(Inkscape) Color is used for the text in Inkscape, but the package 'color.sty' is not loaded}%
    \renewcommand\color[2][]{}%
  }%
  \providecommand\transparent[1]{%
    \errmessage{(Inkscape) Transparency is used (non-zero) for the text in Inkscape, but the package 'transparent.sty' is not loaded}%
    \renewcommand\transparent[1]{}%
  }%
  \providecommand\rotatebox[2]{#2}%
  \ifx\svgwidth\undefined%
    \setlength{\unitlength}{49.54732089bp}%
    \ifx\svgscale\undefined%
      \relax%
    \else%
      \setlength{\unitlength}{\unitlength * \real{\svgscale}}%
    \fi%
  \else%
    \setlength{\unitlength}{\svgwidth}%
  \fi%
  \global\let\svgwidth\undefined%
  \global\let\svgscale\undefined%
  \makeatother%
  \begin{picture}(1,1.00059637)%
    \put(0,0){\includegraphics[width=\unitlength,page=1]{PushOff.pdf}}%
  \end{picture}%
\endgroup%
\,\,}\cdots{\def\svgwidth{1,6ex}\,\,%% Creator: Inkscape 0.91_64bit, www.inkscape.org
%% PDF/EPS/PS + LaTeX output extension by Johan Engelen, 2010
%% Accompanies image file 'PushOff.pdf' (pdf, eps, ps)
%%
%% To include the image in your LaTeX document, write
%%   \input{<filename>.pdf_tex}
%%  instead of
%%   \includegraphics{<filename>.pdf}
%% To scale the image, write
%%   \def\svgwidth{<desired width>}
%%   \input{<filename>.pdf_tex}
%%  instead of
%%   \includegraphics[width=<desired width>]{<filename>.pdf}
%%
%% Images with a different path to the parent latex file can
%% be accessed with the `import' package (which may need to be
%% installed) using
%%   \usepackage{import}
%% in the preamble, and then including the image with
%%   \import{<path to file>}{<filename>.pdf_tex}
%% Alternatively, one can specify
%%   \graphicspath{{<path to file>/}}
%% 
%% For more information, please see info/svg-inkscape on CTAN:
%%   http://tug.ctan.org/tex-archive/info/svg-inkscape
%%
\begingroup%
  \makeatletter%
  \providecommand\color[2][]{%
    \errmessage{(Inkscape) Color is used for the text in Inkscape, but the package 'color.sty' is not loaded}%
    \renewcommand\color[2][]{}%
  }%
  \providecommand\transparent[1]{%
    \errmessage{(Inkscape) Transparency is used (non-zero) for the text in Inkscape, but the package 'transparent.sty' is not loaded}%
    \renewcommand\transparent[1]{}%
  }%
  \providecommand\rotatebox[2]{#2}%
  \ifx\svgwidth\undefined%
    \setlength{\unitlength}{49.54732089bp}%
    \ifx\svgscale\undefined%
      \relax%
    \else%
      \setlength{\unitlength}{\unitlength * \real{\svgscale}}%
    \fi%
  \else%
    \setlength{\unitlength}{\svgwidth}%
  \fi%
  \global\let\svgwidth\undefined%
  \global\let\svgscale\undefined%
  \makeatother%
  \begin{picture}(1,1.00059637)%
    \put(0,0){\includegraphics[width=\unitlength,page=1]{PushOff.pdf}}%
  \end{picture}%
\endgroup%
\,\,} K(\pm 1).$$
    \item \textbf{Translation lemma:} For $r\in \mathbb{Q}-\{0\}$ and $k \in \mathbb{Z}$, we have
    $$K(r)\cong K\left(\frac{1}{k}\right){\def\svgwidth{1,6ex}\,\,%% Creator: Inkscape 0.91_64bit, www.inkscape.org
%% PDF/EPS/PS + LaTeX output extension by Johan Engelen, 2010
%% Accompanies image file 'PushOff.pdf' (pdf, eps, ps)
%%
%% To include the image in your LaTeX document, write
%%   \input{<filename>.pdf_tex}
%%  instead of
%%   \includegraphics{<filename>.pdf}
%% To scale the image, write
%%   \def\svgwidth{<desired width>}
%%   \input{<filename>.pdf_tex}
%%  instead of
%%   \includegraphics[width=<desired width>]{<filename>.pdf}
%%
%% Images with a different path to the parent latex file can
%% be accessed with the `import' package (which may need to be
%% installed) using
%%   \usepackage{import}
%% in the preamble, and then including the image with
%%   \import{<path to file>}{<filename>.pdf_tex}
%% Alternatively, one can specify
%%   \graphicspath{{<path to file>/}}
%% 
%% For more information, please see info/svg-inkscape on CTAN:
%%   http://tug.ctan.org/tex-archive/info/svg-inkscape
%%
\begingroup%
  \makeatletter%
  \providecommand\color[2][]{%
    \errmessage{(Inkscape) Color is used for the text in Inkscape, but the package 'color.sty' is not loaded}%
    \renewcommand\color[2][]{}%
  }%
  \providecommand\transparent[1]{%
    \errmessage{(Inkscape) Transparency is used (non-zero) for the text in Inkscape, but the package 'transparent.sty' is not loaded}%
    \renewcommand\transparent[1]{}%
  }%
  \providecommand\rotatebox[2]{#2}%
  \ifx\svgwidth\undefined%
    \setlength{\unitlength}{49.54732089bp}%
    \ifx\svgscale\undefined%
      \relax%
    \else%
      \setlength{\unitlength}{\unitlength * \real{\svgscale}}%
    \fi%
  \else%
    \setlength{\unitlength}{\svgwidth}%
  \fi%
  \global\let\svgwidth\undefined%
  \global\let\svgscale\undefined%
  \makeatother%
  \begin{picture}(1,1.00059637)%
    \put(0,0){\includegraphics[width=\unitlength,page=1]{PushOff.pdf}}%
  \end{picture}%
\endgroup%
\,\,} K\left(\frac{1}{
\frac{1}{r}-k}\right).$$
In the case when $r<0$, $r$ can be uniquely written as 
\begin{align*}
    r=[r_1+1,\ldots,r_n]:=r_1+1-\frac{1}{r_2-\frac{1}{\cdots-\frac{1}{r_n}}}
\end{align*} with integers $r_1,\ldots,r_n\leq -2$ and then we have 
\begin{align*}
    K(r)\cong K_{|2+r_1|}(-1){\def\svgwidth{1,6ex}\,\,%% Creator: Inkscape 0.91_64bit, www.inkscape.org
%% PDF/EPS/PS + LaTeX output extension by Johan Engelen, 2010
%% Accompanies image file 'PushOff.pdf' (pdf, eps, ps)
%%
%% To include the image in your LaTeX document, write
%%   \input{<filename>.pdf_tex}
%%  instead of
%%   \includegraphics{<filename>.pdf}
%% To scale the image, write
%%   \def\svgwidth{<desired width>}
%%   \input{<filename>.pdf_tex}
%%  instead of
%%   \includegraphics[width=<desired width>]{<filename>.pdf}
%%
%% Images with a different path to the parent latex file can
%% be accessed with the `import' package (which may need to be
%% installed) using
%%   \usepackage{import}
%% in the preamble, and then including the image with
%%   \import{<path to file>}{<filename>.pdf_tex}
%% Alternatively, one can specify
%%   \graphicspath{{<path to file>/}}
%% 
%% For more information, please see info/svg-inkscape on CTAN:
%%   http://tug.ctan.org/tex-archive/info/svg-inkscape
%%
\begingroup%
  \makeatletter%
  \providecommand\color[2][]{%
    \errmessage{(Inkscape) Color is used for the text in Inkscape, but the package 'color.sty' is not loaded}%
    \renewcommand\color[2][]{}%
  }%
  \providecommand\transparent[1]{%
    \errmessage{(Inkscape) Transparency is used (non-zero) for the text in Inkscape, but the package 'transparent.sty' is not loaded}%
    \renewcommand\transparent[1]{}%
  }%
  \providecommand\rotatebox[2]{#2}%
  \ifx\svgwidth\undefined%
    \setlength{\unitlength}{49.54732089bp}%
    \ifx\svgscale\undefined%
      \relax%
    \else%
      \setlength{\unitlength}{\unitlength * \real{\svgscale}}%
    \fi%
  \else%
    \setlength{\unitlength}{\svgwidth}%
  \fi%
  \global\let\svgwidth\undefined%
  \global\let\svgscale\undefined%
  \makeatother%
  \begin{picture}(1,1.00059637)%
    \put(0,0){\includegraphics[width=\unitlength,page=1]{PushOff.pdf}}%
  \end{picture}%
\endgroup%
\,\,} K_{|2+r_1|,|2+r_2|}(-1){\def\svgwidth{1,6ex}\,\,%% Creator: Inkscape 0.91_64bit, www.inkscape.org
%% PDF/EPS/PS + LaTeX output extension by Johan Engelen, 2010
%% Accompanies image file 'PushOff.pdf' (pdf, eps, ps)
%%
%% To include the image in your LaTeX document, write
%%   \input{<filename>.pdf_tex}
%%  instead of
%%   \includegraphics{<filename>.pdf}
%% To scale the image, write
%%   \def\svgwidth{<desired width>}
%%   \input{<filename>.pdf_tex}
%%  instead of
%%   \includegraphics[width=<desired width>]{<filename>.pdf}
%%
%% Images with a different path to the parent latex file can
%% be accessed with the `import' package (which may need to be
%% installed) using
%%   \usepackage{import}
%% in the preamble, and then including the image with
%%   \import{<path to file>}{<filename>.pdf_tex}
%% Alternatively, one can specify
%%   \graphicspath{{<path to file>/}}
%% 
%% For more information, please see info/svg-inkscape on CTAN:
%%   http://tug.ctan.org/tex-archive/info/svg-inkscape
%%
\begingroup%
  \makeatletter%
  \providecommand\color[2][]{%
    \errmessage{(Inkscape) Color is used for the text in Inkscape, but the package 'color.sty' is not loaded}%
    \renewcommand\color[2][]{}%
  }%
  \providecommand\transparent[1]{%
    \errmessage{(Inkscape) Transparency is used (non-zero) for the text in Inkscape, but the package 'transparent.sty' is not loaded}%
    \renewcommand\transparent[1]{}%
  }%
  \providecommand\rotatebox[2]{#2}%
  \ifx\svgwidth\undefined%
    \setlength{\unitlength}{49.54732089bp}%
    \ifx\svgscale\undefined%
      \relax%
    \else%
      \setlength{\unitlength}{\unitlength * \real{\svgscale}}%
    \fi%
  \else%
    \setlength{\unitlength}{\svgwidth}%
  \fi%
  \global\let\svgwidth\undefined%
  \global\let\svgscale\undefined%
  \makeatother%
  \begin{picture}(1,1.00059637)%
    \put(0,0){\includegraphics[width=\unitlength,page=1]{PushOff.pdf}}%
  \end{picture}%
\endgroup%
\,\,}
\cdots {\def\svgwidth{1,6ex}\,\,%% Creator: Inkscape 0.91_64bit, www.inkscape.org
%% PDF/EPS/PS + LaTeX output extension by Johan Engelen, 2010
%% Accompanies image file 'PushOff.pdf' (pdf, eps, ps)
%%
%% To include the image in your LaTeX document, write
%%   \input{<filename>.pdf_tex}
%%  instead of
%%   \includegraphics{<filename>.pdf}
%% To scale the image, write
%%   \def\svgwidth{<desired width>}
%%   \input{<filename>.pdf_tex}
%%  instead of
%%   \includegraphics[width=<desired width>]{<filename>.pdf}
%%
%% Images with a different path to the parent latex file can
%% be accessed with the `import' package (which may need to be
%% installed) using
%%   \usepackage{import}
%% in the preamble, and then including the image with
%%   \import{<path to file>}{<filename>.pdf_tex}
%% Alternatively, one can specify
%%   \graphicspath{{<path to file>/}}
%% 
%% For more information, please see info/svg-inkscape on CTAN:
%%   http://tug.ctan.org/tex-archive/info/svg-inkscape
%%
\begingroup%
  \makeatletter%
  \providecommand\color[2][]{%
    \errmessage{(Inkscape) Color is used for the text in Inkscape, but the package 'color.sty' is not loaded}%
    \renewcommand\color[2][]{}%
  }%
  \providecommand\transparent[1]{%
    \errmessage{(Inkscape) Transparency is used (non-zero) for the text in Inkscape, but the package 'transparent.sty' is not loaded}%
    \renewcommand\transparent[1]{}%
  }%
  \providecommand\rotatebox[2]{#2}%
  \ifx\svgwidth\undefined%
    \setlength{\unitlength}{49.54732089bp}%
    \ifx\svgscale\undefined%
      \relax%
    \else%
      \setlength{\unitlength}{\unitlength * \real{\svgscale}}%
    \fi%
  \else%
    \setlength{\unitlength}{\svgwidth}%
  \fi%
  \global\let\svgwidth\undefined%
  \global\let\svgscale\undefined%
  \makeatother%
  \begin{picture}(1,1.00059637)%
    \put(0,0){\includegraphics[width=\unitlength,page=1]{PushOff.pdf}}%
  \end{picture}%
\endgroup%
\,\,} K_{|2+r_1|,|2+r_2|,\ldots, |2+r_n| }(-1).
\end{align*}
\end{enumerate}
In addition, all these results hold true in a tubular neighborhood of $K$. In particular, they can be applied to knots in larger contact surgery diagrams.\qed
\end{lem}

\section{The homotopical invariants of a contact structure}\label{sec:invariants}

We will need to compute the algebraic invariants of the underlying tangential $2$-plane field of a contact structure. It is known that a tangential $2$-plane field $\xi$ on a rational homology sphere $M$ is (up to homotopy) completely determined by the $\de_3$-invariant and the Gompf's $\Gamma$-invariant~\cite{Gompf_Stein,Ding_Geiges_Stipsicz}. 
Roughly speaking, the $\Gamma$-invariant is a refinement of the Euler class and encodes $\xi$ on the $2$-skeleton of $M$, while the $\de_3$-invariant specifies $\xi$ on the $3$-cell. First, we recall the following lemma to compute the $\de_3$-invariant for $(\pm1/n)$-surgeries~\cite{Durst_Kegel_rot_surgery}.

 \begin{lem}\label{lem:d3}
     Let $L=L_1\cup\cdots\cup L_k$ be an oriented Legendrian link in $(\mathbb{S}^3,\xi_{st})$ and let $(M,\xi)$ be the contact manifold obtained by contact $\left(\pm {1}/{n_i}\right)$-surgeries along $L$, for $n_i>0$. Let $t_i$ and $r_i$ be the Thurston--Bennequin invariant and rotation number of $L_i $ for all $i=1,\ldots,k$. Let $l_{ij}$ be the linking number of $L_i$ with $L_j$ and let $\frac{p_i}{q_i}=t_i\pm \frac{1}{n_i}$ be the topological surgery coefficient of $L_i$. We define the generalized linking matrix $Q$ by
\begin{align*}
  Q=\begin{pmatrix}
     p_1&q_2l_{12}&\cdots& q_kl_{1k}\\
     q_1l_{12}&p_2&\cdots &q_kl_{2k}\\
     \vdots&&\ddots&\vdots\\
     q_1l_{1k}&&&p_k
  \end{pmatrix} . 
\end{align*}
\begin{enumerate}
    \item The first homology $H_1(M)$ is presented  by the abelian group $\langle\mu_i| Q\mu^T=0\rangle$, where $\mu=(\mu_1,\ldots,\mu_k)$ is the vector of meridians $\mu_i$ of $L_i$.
    \item If there exists a rational solution $b \in \Q^k$ of $Qb = r$, where $r=(r_1,\ldots,r_k)^T$, the $\de_3$-invariant is well defined and is computed as
    $$d_3=\frac{1}{4}\sum_{i=1}^k\big(n_ir_ib_i+(3-n_i)\operatorname{sign}_i\big)-\frac{3}{4}\sigma(Q),$$
    where $\operatorname{sign}_i$ denotes the sign of the contact surgery coefficient of $L_i$ and $\sigma(Q)$ is the signature of $Q$. Note that the eigenvalues of $Q$ are all real and thus the signature of $Q$ is well-defined $($see Theorem 5.1. in~\cite{Durst_Kegel_rot_surgery}$)$.\qed
\end{enumerate}   
 \end{lem}

Next, we discuss the $\Gamma$-invariant of a contact manifold $(M, \xi)$. Let $\mathfrak{s}$ be a spin structure on $M$, then Gompf~\cite{Gompf_Stein} defines an invariant $\Gamma(\xi,\mathfrak{s})\in H_1(M)$, that depends only on the tangential $2$-plane field $\xi$ and the spin structure $\mathfrak{s}$. Intuitively, $\Gamma$-invariant can be viewed as a \textit{half} Euler class of $\xi$ since $2\Gamma(\xi,\mathfrak{s})=\operatorname{PD}(\e(\xi))$, for all spin structures $\mathfrak{s}$. But if $M$ has $2$-torsion in its first homology (as for example $\RP$) then $\Gamma$-invariant contains more information than the Euler class. To explain how to compute $\Gamma$-invariant from a contact $(\pm1)$-surgery diagram, we first recall how spin structures are represented in surgery diagrams~\cite{GS99}.

Let $L = L_1 \cup \dots \cup L_k$ be an integer surgery diagram of a smooth $3$-manifold $M$ along an oriented link $L$, where the framings of $L_i$ are measured relative to the Seifert framing. We denote the framing of $L_i$ by $l_{ii}$. A sublink $(L_j)_{j \in J}$, for some subset $J \subset \{1, 2, \dots, k\}$, is called a \textit{characteristic sublink} if  
\[
l_{ii} \equiv \sum_{j \in J} l_{ij} \pmod{2}
\]  
for every component $L_i$ of $L$. The set of characteristic sublinks of $L$ is in natural bijection with the set of spin structures on $M$~\cite{GS99}. Thus, a spin structure on $M$ can be described via a characteristic sublink of $L$.
The following lemma from~\cite{EKS_contact_surgery_numbers} explains how to compute the $\Gamma$-invariant from a contact $(\pm1)$-surgery diagram. 

\begin{lem}  \label{lem:Gamma}
Let \( L = L_1 \cup \dots \cup L_k \) be a Legendrian link in \( (\mathbb{S}^3, \xi_{st}) \), and let \( (M, \xi) \) be the contact manifold obtained by performing contact $(\pm 1)$-surgery along \( L \). 
%For $r_i$, $\mu_i$, and \( Q \) as above.
Let \( (L_j)_{j \in J} \) be a characteristic sublink describing a spin structure \( \mathfrak{s} \) on \( M \). Then the \( \Gamma \)-invariant satisfies  
\[
\pushQED{\qed} 
\Gamma(\xi, \mathfrak{s}) = \frac{1}{2} \left( \sum_{i=1}^k r_i \mu_i + \sum_{j \in J} (Q\mu)_j \right) \in H_1(M).\qedhere
\popQED
\]  
\end{lem}  

To effectively use the above lemma for comparing $\Gamma$-invariants of contact structures described by contact surgery diagrams of the same underlying smooth $3$-manifold, we need to understand how the spin structures in these surgery diagrams are related. For that, we need to understand how a characteristic sublink changes under smooth Kirby moves. This is summarized in the following lemma which can be extracted from~\cite{GS99}.
 
\begin{lem}
    Let $L = L_1 \cup \dots \cup L_k$ be a smooth oriented integer surgery link with characteristic sublink $(L_j)_{j\in J}$ representing a spin structure $\mathfrak{s}$. Then the following modifications of the characteristic sublink under Kirby moves preserve the spin structure $\mathfrak{s}$.
    \begin{itemize}
      \item \textbf{Blow up/down:} Let $U$ be an unknot with surgery coefficient $\pm 1$ which is added to $L$ under a blow up move. Then $U$ gets added to the characteristic sublink if and only if $$\sum_{j\in J}\lk(U,L_j)=0\, \pmod{2}.$$ The other components of the characteristic sublink stay unchanged. Under the blow down move, we remove an unknot with $\pm1$ coefficients from $L$ and the other components of the characteristic sublink stay the same.
      \item \textbf{Handle slide:}
      If we slide the component $L_i$ over the component $L_k$ then $L_k$ changes its membership status in the characteristic sublink if and only if $L_i$ is in the characteristic sublink. All other components of the characteristic sublink stay the same.
      \item \textbf{Rolfsen twist:} If $L_k$ is a $0$-framed unknot and we perform an $n$-fold Rolfsen twist on $L_k$, then the resulting surgery diagram is again an integer surgery diagram. All components of the characteristic sublink remain unchanged, except for a possible change of the unknot $L_k$ which is being twisted. $L_k$ changes the membership status to the characteristic sublink if and only if 
      \[
      \pushQED{\qed} 
      n\left(1+\sum_{j\in J-\{k\}}\lk(L_k,L_j)\right)=1\, \pmod{2}.\qedhere
\popQED
\] 
  \end{itemize}
\end{lem}

\section{Smooth surgery diagrams of \texorpdfstring{$\RP$}{RP3}}

For the proof of Theorem~\ref{thm:main} we will first recall the classification of smooth surgery diagrams of $\RP$ along a single knot, which is a direct corollary of~\cite{KMOS}.

\begin{lem}\label{lem:smooth_diagrams_RP3}
    Let $K$ be a knot in $\mathbb{S}^3$ such that for some topological surgery coefficient $r\in\Q$ the $r$-surgery $K(r)$ along $K$ is diffeomorphic to $\RP$. Then $K$ is the unknot and $r=\frac{2}{2n+1}$ for some $n\in \mathbb{Z}$. Conversely, topological $(\frac{2}{2n+1})$-surgery on an unknot yields $\RP$. 
\end{lem}

\begin{proof}
We know that $\RP$ is diffeomorphic to $(-2)$-surgery on an unknot. Performing an $(n+1)$-fold Rolfsen twist along this unknot preserves the knot but changes the surgery coefficient to $$\frac{1}{n+1-\frac{1}{2}}=\frac{2}{2n+1}.$$ 

    Now let $K$ be a knot in $\mathbb{S}^3$ such that for some $r=p/q\in\Q$ the $r$-surgery $K(r)$ along $K$ is diffeomorphic to $\RP$. Since $H_1(K(r))$ is isomorphic to $\Z_p$ and the first homology of $\RP$ is isomorphic to $\Z_2$ it follows that $p=\pm2$. Since $p$ and $q$ are coprime it follows that $r$ is of the form $\frac{2}{2n+1}$. We conclude that $K$ and the unknot have orientation-preserving diffeomorphic $(\frac{2}{2n+1})$-surgeries. Since every slope of the unknot is characterizing~\cite{KMOS}, it follows that $K$ is isotopic to an unknot.
\end{proof}

\section{Contact surgery diagrams of contact structures on \texorpdfstring{$\RP$}{RP3}}

In this section, we will use Lemma~\ref{lem:smooth_diagrams_RP3} to describe contact $(\pm1/n)$-surgery diagrams of all contact structures on $\RP$ that have rational contact surgery number one.

\begin{lem}\label{lem:contact_diagrams_RP3}
    Let $K$ be a Legendrian knot in $(\mathbb{S}^3,\xist)$ such that some contact $r_c$-surgery on $K$ yields a contact structure on $\RP$. Then $K$ is isotopic to a Legendrian unknot $U$ with Thurston--Bennequin invariant $t\leq-1$ and there exists an integer $n\in\Z$ such that 
    $$r_c=\frac{2}{2n+1}-t.$$
    Moreover, depending on $n$ and $t$, the contact manifold $U(r_c)$ is contactomorphic to exactly one of the contact $(\pm1/k)$-surgery diagrams from the $12$ cases shown in Table~\ref{tab:tab1}. Conversely, all the contact surgery diagrams shown in Table~\ref{tab:tab1} present contact structures on $\RP$.

    In particular, it follows that the contact structures on $\RP$ with rational contact surgery number one are exactly the contact structures presented by the contact surgery diagrams in Table~\ref{tab:tab1}.
\end{lem}

\begin{table}[htbp]
\caption{Surgery descriptions for contact structures on $\RP$ with contact surgery number $1$.}
\begin{tabular}{ c|l|c|c} 
Case &$U\left(\frac{2}{2n+1}-t\right)$&$t$ & $n$\\
\toprule
  (0)&$U(-1)$&-1&-1\\
    \midrule
  (1)&$U(+1){\def\svgwidth{1,6ex}\,\,%% Creator: Inkscape 0.91_64bit, www.inkscape.org
%% PDF/EPS/PS + LaTeX output extension by Johan Engelen, 2010
%% Accompanies image file 'PushOff.pdf' (pdf, eps, ps)
%%
%% To include the image in your LaTeX document, write
%%   \input{<filename>.pdf_tex}
%%  instead of
%%   \includegraphics{<filename>.pdf}
%% To scale the image, write
%%   \def\svgwidth{<desired width>}
%%   \input{<filename>.pdf_tex}
%%  instead of
%%   \includegraphics[width=<desired width>]{<filename>.pdf}
%%
%% Images with a different path to the parent latex file can
%% be accessed with the `import' package (which may need to be
%% installed) using
%%   \usepackage{import}
%% in the preamble, and then including the image with
%%   \import{<path to file>}{<filename>.pdf_tex}
%% Alternatively, one can specify
%%   \graphicspath{{<path to file>/}}
%% 
%% For more information, please see info/svg-inkscape on CTAN:
%%   http://tug.ctan.org/tex-archive/info/svg-inkscape
%%
\begingroup%
  \makeatletter%
  \providecommand\color[2][]{%
    \errmessage{(Inkscape) Color is used for the text in Inkscape, but the package 'color.sty' is not loaded}%
    \renewcommand\color[2][]{}%
  }%
  \providecommand\transparent[1]{%
    \errmessage{(Inkscape) Transparency is used (non-zero) for the text in Inkscape, but the package 'transparent.sty' is not loaded}%
    \renewcommand\transparent[1]{}%
  }%
  \providecommand\rotatebox[2]{#2}%
  \ifx\svgwidth\undefined%
    \setlength{\unitlength}{49.54732089bp}%
    \ifx\svgscale\undefined%
      \relax%
    \else%
      \setlength{\unitlength}{\unitlength * \real{\svgscale}}%
    \fi%
  \else%
    \setlength{\unitlength}{\svgwidth}%
  \fi%
  \global\let\svgwidth\undefined%
  \global\let\svgscale\undefined%
  \makeatother%
  \begin{picture}(1,1.00059637)%
    \put(0,0){\includegraphics[width=\unitlength,page=1]{PushOff.pdf}}%
  \end{picture}%
\endgroup%
\,\,} U_{n+1}(-\frac{1}{2})$&-1  &$\geq 0$ \\
  \midrule
 (2)& $U(\frac{1}{3})$&-1  & -2\\
 \midrule
 (3)&$ U(\frac{1}{2}){\def\svgwidth{1,6ex}\,\,%% Creator: Inkscape 0.91_64bit, www.inkscape.org
%% PDF/EPS/PS + LaTeX output extension by Johan Engelen, 2010
%% Accompanies image file 'PushOff.pdf' (pdf, eps, ps)
%%
%% To include the image in your LaTeX document, write
%%   \input{<filename>.pdf_tex}
%%  instead of
%%   \includegraphics{<filename>.pdf}
%% To scale the image, write
%%   \def\svgwidth{<desired width>}
%%   \input{<filename>.pdf_tex}
%%  instead of
%%   \includegraphics[width=<desired width>]{<filename>.pdf}
%%
%% Images with a different path to the parent latex file can
%% be accessed with the `import' package (which may need to be
%% installed) using
%%   \usepackage{import}
%% in the preamble, and then including the image with
%%   \import{<path to file>}{<filename>.pdf_tex}
%% Alternatively, one can specify
%%   \graphicspath{{<path to file>/}}
%% 
%% For more information, please see info/svg-inkscape on CTAN:
%%   http://tug.ctan.org/tex-archive/info/svg-inkscape
%%
\begingroup%
  \makeatletter%
  \providecommand\color[2][]{%
    \errmessage{(Inkscape) Color is used for the text in Inkscape, but the package 'color.sty' is not loaded}%
    \renewcommand\color[2][]{}%
  }%
  \providecommand\transparent[1]{%
    \errmessage{(Inkscape) Transparency is used (non-zero) for the text in Inkscape, but the package 'transparent.sty' is not loaded}%
    \renewcommand\transparent[1]{}%
  }%
  \providecommand\rotatebox[2]{#2}%
  \ifx\svgwidth\undefined%
    \setlength{\unitlength}{49.54732089bp}%
    \ifx\svgscale\undefined%
      \relax%
    \else%
      \setlength{\unitlength}{\unitlength * \real{\svgscale}}%
    \fi%
  \else%
    \setlength{\unitlength}{\svgwidth}%
  \fi%
  \global\let\svgwidth\undefined%
  \global\let\svgscale\undefined%
  \makeatother%
  \begin{picture}(1,1.00059637)%
    \put(0,0){\includegraphics[width=\unitlength,page=1]{PushOff.pdf}}%
  \end{picture}%
\endgroup%
\,\,} U_2(-1)$&-1  & -3\\
\midrule
(4)&$U(\frac{1}{2}){\def\svgwidth{1,6ex}\,\,%% Creator: Inkscape 0.91_64bit, www.inkscape.org
%% PDF/EPS/PS + LaTeX output extension by Johan Engelen, 2010
%% Accompanies image file 'PushOff.pdf' (pdf, eps, ps)
%%
%% To include the image in your LaTeX document, write
%%   \input{<filename>.pdf_tex}
%%  instead of
%%   \includegraphics{<filename>.pdf}
%% To scale the image, write
%%   \def\svgwidth{<desired width>}
%%   \input{<filename>.pdf_tex}
%%  instead of
%%   \includegraphics[width=<desired width>]{<filename>.pdf}
%%
%% Images with a different path to the parent latex file can
%% be accessed with the `import' package (which may need to be
%% installed) using
%%   \usepackage{import}
%% in the preamble, and then including the image with
%%   \import{<path to file>}{<filename>.pdf_tex}
%% Alternatively, one can specify
%%   \graphicspath{{<path to file>/}}
%% 
%% For more information, please see info/svg-inkscape on CTAN:
%%   http://tug.ctan.org/tex-archive/info/svg-inkscape
%%
\begingroup%
  \makeatletter%
  \providecommand\color[2][]{%
    \errmessage{(Inkscape) Color is used for the text in Inkscape, but the package 'color.sty' is not loaded}%
    \renewcommand\color[2][]{}%
  }%
  \providecommand\transparent[1]{%
    \errmessage{(Inkscape) Transparency is used (non-zero) for the text in Inkscape, but the package 'transparent.sty' is not loaded}%
    \renewcommand\transparent[1]{}%
  }%
  \providecommand\rotatebox[2]{#2}%
  \ifx\svgwidth\undefined%
    \setlength{\unitlength}{49.54732089bp}%
    \ifx\svgscale\undefined%
      \relax%
    \else%
      \setlength{\unitlength}{\unitlength * \real{\svgscale}}%
    \fi%
  \else%
    \setlength{\unitlength}{\svgwidth}%
  \fi%
  \global\let\svgwidth\undefined%
  \global\let\svgscale\undefined%
  \makeatother%
  \begin{picture}(1,1.00059637)%
    \put(0,0){\includegraphics[width=\unitlength,page=1]{PushOff.pdf}}%
  \end{picture}%
\endgroup%
\,\,} U_1(\frac{-1}{|n|-3}){\def\svgwidth{1,6ex}\,\,%% Creator: Inkscape 0.91_64bit, www.inkscape.org
%% PDF/EPS/PS + LaTeX output extension by Johan Engelen, 2010
%% Accompanies image file 'PushOff.pdf' (pdf, eps, ps)
%%
%% To include the image in your LaTeX document, write
%%   \input{<filename>.pdf_tex}
%%  instead of
%%   \includegraphics{<filename>.pdf}
%% To scale the image, write
%%   \def\svgwidth{<desired width>}
%%   \input{<filename>.pdf_tex}
%%  instead of
%%   \includegraphics[width=<desired width>]{<filename>.pdf}
%%
%% Images with a different path to the parent latex file can
%% be accessed with the `import' package (which may need to be
%% installed) using
%%   \usepackage{import}
%% in the preamble, and then including the image with
%%   \import{<path to file>}{<filename>.pdf_tex}
%% Alternatively, one can specify
%%   \graphicspath{{<path to file>/}}
%% 
%% For more information, please see info/svg-inkscape on CTAN:
%%   http://tug.ctan.org/tex-archive/info/svg-inkscape
%%
\begingroup%
  \makeatletter%
  \providecommand\color[2][]{%
    \errmessage{(Inkscape) Color is used for the text in Inkscape, but the package 'color.sty' is not loaded}%
    \renewcommand\color[2][]{}%
  }%
  \providecommand\transparent[1]{%
    \errmessage{(Inkscape) Transparency is used (non-zero) for the text in Inkscape, but the package 'transparent.sty' is not loaded}%
    \renewcommand\transparent[1]{}%
  }%
  \providecommand\rotatebox[2]{#2}%
  \ifx\svgwidth\undefined%
    \setlength{\unitlength}{49.54732089bp}%
    \ifx\svgscale\undefined%
      \relax%
    \else%
      \setlength{\unitlength}{\unitlength * \real{\svgscale}}%
    \fi%
  \else%
    \setlength{\unitlength}{\svgwidth}%
  \fi%
  \global\let\svgwidth\undefined%
  \global\let\svgscale\undefined%
  \makeatother%
  \begin{picture}(1,1.00059637)%
    \put(0,0){\includegraphics[width=\unitlength,page=1]{PushOff.pdf}}%
  \end{picture}%
\endgroup%
\,\,} U_{1,1}(-1)$& -1 & $\leq -4$ \\ 
\midrule
 (5)& $U(+1){\def\svgwidth{1,6ex}\,\,%% Creator: Inkscape 0.91_64bit, www.inkscape.org
%% PDF/EPS/PS + LaTeX output extension by Johan Engelen, 2010
%% Accompanies image file 'PushOff.pdf' (pdf, eps, ps)
%%
%% To include the image in your LaTeX document, write
%%   \input{<filename>.pdf_tex}
%%  instead of
%%   \includegraphics{<filename>.pdf}
%% To scale the image, write
%%   \def\svgwidth{<desired width>}
%%   \input{<filename>.pdf_tex}
%%  instead of
%%   \includegraphics[width=<desired width>]{<filename>.pdf}
%%
%% Images with a different path to the parent latex file can
%% be accessed with the `import' package (which may need to be
%% installed) using
%%   \usepackage{import}
%% in the preamble, and then including the image with
%%   \import{<path to file>}{<filename>.pdf_tex}
%% Alternatively, one can specify
%%   \graphicspath{{<path to file>/}}
%% 
%% For more information, please see info/svg-inkscape on CTAN:
%%   http://tug.ctan.org/tex-archive/info/svg-inkscape
%%
\begingroup%
  \makeatletter%
  \providecommand\color[2][]{%
    \errmessage{(Inkscape) Color is used for the text in Inkscape, but the package 'color.sty' is not loaded}%
    \renewcommand\color[2][]{}%
  }%
  \providecommand\transparent[1]{%
    \errmessage{(Inkscape) Transparency is used (non-zero) for the text in Inkscape, but the package 'transparent.sty' is not loaded}%
    \renewcommand\transparent[1]{}%
  }%
  \providecommand\rotatebox[2]{#2}%
  \ifx\svgwidth\undefined%
    \setlength{\unitlength}{49.54732089bp}%
    \ifx\svgscale\undefined%
      \relax%
    \else%
      \setlength{\unitlength}{\unitlength * \real{\svgscale}}%
    \fi%
  \else%
    \setlength{\unitlength}{\svgwidth}%
  \fi%
  \global\let\svgwidth\undefined%
  \global\let\svgscale\undefined%
  \makeatother%
  \begin{picture}(1,1.00059637)%
    \put(0,0){\includegraphics[width=\unitlength,page=1]{PushOff.pdf}}%
  \end{picture}%
\endgroup%
\,\,} U_{1}(-\frac{1}{-t-1}){\def\svgwidth{1,6ex}\,\,%% Creator: Inkscape 0.91_64bit, www.inkscape.org
%% PDF/EPS/PS + LaTeX output extension by Johan Engelen, 2010
%% Accompanies image file 'PushOff.pdf' (pdf, eps, ps)
%%
%% To include the image in your LaTeX document, write
%%   \input{<filename>.pdf_tex}
%%  instead of
%%   \includegraphics{<filename>.pdf}
%% To scale the image, write
%%   \def\svgwidth{<desired width>}
%%   \input{<filename>.pdf_tex}
%%  instead of
%%   \includegraphics[width=<desired width>]{<filename>.pdf}
%%
%% Images with a different path to the parent latex file can
%% be accessed with the `import' package (which may need to be
%% installed) using
%%   \usepackage{import}
%% in the preamble, and then including the image with
%%   \import{<path to file>}{<filename>.pdf_tex}
%% Alternatively, one can specify
%%   \graphicspath{{<path to file>/}}
%% 
%% For more information, please see info/svg-inkscape on CTAN:
%%   http://tug.ctan.org/tex-archive/info/svg-inkscape
%%
\begingroup%
  \makeatletter%
  \providecommand\color[2][]{%
    \errmessage{(Inkscape) Color is used for the text in Inkscape, but the package 'color.sty' is not loaded}%
    \renewcommand\color[2][]{}%
  }%
  \providecommand\transparent[1]{%
    \errmessage{(Inkscape) Transparency is used (non-zero) for the text in Inkscape, but the package 'transparent.sty' is not loaded}%
    \renewcommand\transparent[1]{}%
  }%
  \providecommand\rotatebox[2]{#2}%
  \ifx\svgwidth\undefined%
    \setlength{\unitlength}{49.54732089bp}%
    \ifx\svgscale\undefined%
      \relax%
    \else%
      \setlength{\unitlength}{\unitlength * \real{\svgscale}}%
    \fi%
  \else%
    \setlength{\unitlength}{\svgwidth}%
  \fi%
  \global\let\svgwidth\undefined%
  \global\let\svgscale\undefined%
  \makeatother%
  \begin{picture}(1,1.00059637)%
    \put(0,0){\includegraphics[width=\unitlength,page=1]{PushOff.pdf}}%
  \end{picture}%
\endgroup%
\,\,} U_{1,n}(-\frac{1}{2})$&$<-1$  &$n\geq 0$ \\
\midrule
 (6)&$U(+1){\def\svgwidth{1,6ex}\,\,%% Creator: Inkscape 0.91_64bit, www.inkscape.org
%% PDF/EPS/PS + LaTeX output extension by Johan Engelen, 2010
%% Accompanies image file 'PushOff.pdf' (pdf, eps, ps)
%%
%% To include the image in your LaTeX document, write
%%   \input{<filename>.pdf_tex}
%%  instead of
%%   \includegraphics{<filename>.pdf}
%% To scale the image, write
%%   \def\svgwidth{<desired width>}
%%   \input{<filename>.pdf_tex}
%%  instead of
%%   \includegraphics[width=<desired width>]{<filename>.pdf}
%%
%% Images with a different path to the parent latex file can
%% be accessed with the `import' package (which may need to be
%% installed) using
%%   \usepackage{import}
%% in the preamble, and then including the image with
%%   \import{<path to file>}{<filename>.pdf_tex}
%% Alternatively, one can specify
%%   \graphicspath{{<path to file>/}}
%% 
%% For more information, please see info/svg-inkscape on CTAN:
%%   http://tug.ctan.org/tex-archive/info/svg-inkscape
%%
\begingroup%
  \makeatletter%
  \providecommand\color[2][]{%
    \errmessage{(Inkscape) Color is used for the text in Inkscape, but the package 'color.sty' is not loaded}%
    \renewcommand\color[2][]{}%
  }%
  \providecommand\transparent[1]{%
    \errmessage{(Inkscape) Transparency is used (non-zero) for the text in Inkscape, but the package 'transparent.sty' is not loaded}%
    \renewcommand\transparent[1]{}%
  }%
  \providecommand\rotatebox[2]{#2}%
  \ifx\svgwidth\undefined%
    \setlength{\unitlength}{49.54732089bp}%
    \ifx\svgscale\undefined%
      \relax%
    \else%
      \setlength{\unitlength}{\unitlength * \real{\svgscale}}%
    \fi%
  \else%
    \setlength{\unitlength}{\svgwidth}%
  \fi%
  \global\let\svgwidth\undefined%
  \global\let\svgscale\undefined%
  \makeatother%
  \begin{picture}(1,1.00059637)%
    \put(0,0){\includegraphics[width=\unitlength,page=1]{PushOff.pdf}}%
  \end{picture}%
\endgroup%
\,\,} U_3(-1)$&$-2$  &$-2$ \\
\midrule
(7)& $U(+1){\def\svgwidth{1,6ex}\,\,%% Creator: Inkscape 0.91_64bit, www.inkscape.org
%% PDF/EPS/PS + LaTeX output extension by Johan Engelen, 2010
%% Accompanies image file 'PushOff.pdf' (pdf, eps, ps)
%%
%% To include the image in your LaTeX document, write
%%   \input{<filename>.pdf_tex}
%%  instead of
%%   \includegraphics{<filename>.pdf}
%% To scale the image, write
%%   \def\svgwidth{<desired width>}
%%   \input{<filename>.pdf_tex}
%%  instead of
%%   \includegraphics[width=<desired width>]{<filename>.pdf}
%%
%% Images with a different path to the parent latex file can
%% be accessed with the `import' package (which may need to be
%% installed) using
%%   \usepackage{import}
%% in the preamble, and then including the image with
%%   \import{<path to file>}{<filename>.pdf_tex}
%% Alternatively, one can specify
%%   \graphicspath{{<path to file>/}}
%% 
%% For more information, please see info/svg-inkscape on CTAN:
%%   http://tug.ctan.org/tex-archive/info/svg-inkscape
%%
\begingroup%
  \makeatletter%
  \providecommand\color[2][]{%
    \errmessage{(Inkscape) Color is used for the text in Inkscape, but the package 'color.sty' is not loaded}%
    \renewcommand\color[2][]{}%
  }%
  \providecommand\transparent[1]{%
    \errmessage{(Inkscape) Transparency is used (non-zero) for the text in Inkscape, but the package 'transparent.sty' is not loaded}%
    \renewcommand\transparent[1]{}%
  }%
  \providecommand\rotatebox[2]{#2}%
  \ifx\svgwidth\undefined%
    \setlength{\unitlength}{49.54732089bp}%
    \ifx\svgscale\undefined%
      \relax%
    \else%
      \setlength{\unitlength}{\unitlength * \real{\svgscale}}%
    \fi%
  \else%
    \setlength{\unitlength}{\svgwidth}%
  \fi%
  \global\let\svgwidth\undefined%
  \global\let\svgscale\undefined%
  \makeatother%
  \begin{picture}(1,1.00059637)%
    \put(0,0){\includegraphics[width=\unitlength,page=1]{PushOff.pdf}}%
  \end{picture}%
\endgroup%
\,\,} U_1(-\frac{1}{-t-2}){\def\svgwidth{1,6ex}\,\,%% Creator: Inkscape 0.91_64bit, www.inkscape.org
%% PDF/EPS/PS + LaTeX output extension by Johan Engelen, 2010
%% Accompanies image file 'PushOff.pdf' (pdf, eps, ps)
%%
%% To include the image in your LaTeX document, write
%%   \input{<filename>.pdf_tex}
%%  instead of
%%   \includegraphics{<filename>.pdf}
%% To scale the image, write
%%   \def\svgwidth{<desired width>}
%%   \input{<filename>.pdf_tex}
%%  instead of
%%   \includegraphics[width=<desired width>]{<filename>.pdf}
%%
%% Images with a different path to the parent latex file can
%% be accessed with the `import' package (which may need to be
%% installed) using
%%   \usepackage{import}
%% in the preamble, and then including the image with
%%   \import{<path to file>}{<filename>.pdf_tex}
%% Alternatively, one can specify
%%   \graphicspath{{<path to file>/}}
%% 
%% For more information, please see info/svg-inkscape on CTAN:
%%   http://tug.ctan.org/tex-archive/info/svg-inkscape
%%
\begingroup%
  \makeatletter%
  \providecommand\color[2][]{%
    \errmessage{(Inkscape) Color is used for the text in Inkscape, but the package 'color.sty' is not loaded}%
    \renewcommand\color[2][]{}%
  }%
  \providecommand\transparent[1]{%
    \errmessage{(Inkscape) Transparency is used (non-zero) for the text in Inkscape, but the package 'transparent.sty' is not loaded}%
    \renewcommand\transparent[1]{}%
  }%
  \providecommand\rotatebox[2]{#2}%
  \ifx\svgwidth\undefined%
    \setlength{\unitlength}{49.54732089bp}%
    \ifx\svgscale\undefined%
      \relax%
    \else%
      \setlength{\unitlength}{\unitlength * \real{\svgscale}}%
    \fi%
  \else%
    \setlength{\unitlength}{\svgwidth}%
  \fi%
  \global\let\svgwidth\undefined%
  \global\let\svgscale\undefined%
  \makeatother%
  \begin{picture}(1,1.00059637)%
    \put(0,0){\includegraphics[width=\unitlength,page=1]{PushOff.pdf}}%
  \end{picture}%
\endgroup%
\,\,} U_{1,2}(-1)$&$<-2$  &$-2$\\
 \midrule
(8)&$U(+1){\def\svgwidth{1,6ex}\,\,%% Creator: Inkscape 0.91_64bit, www.inkscape.org
%% PDF/EPS/PS + LaTeX output extension by Johan Engelen, 2010
%% Accompanies image file 'PushOff.pdf' (pdf, eps, ps)
%%
%% To include the image in your LaTeX document, write
%%   \input{<filename>.pdf_tex}
%%  instead of
%%   \includegraphics{<filename>.pdf}
%% To scale the image, write
%%   \def\svgwidth{<desired width>}
%%   \input{<filename>.pdf_tex}
%%  instead of
%%   \includegraphics[width=<desired width>]{<filename>.pdf}
%%
%% Images with a different path to the parent latex file can
%% be accessed with the `import' package (which may need to be
%% installed) using
%%   \usepackage{import}
%% in the preamble, and then including the image with
%%   \import{<path to file>}{<filename>.pdf_tex}
%% Alternatively, one can specify
%%   \graphicspath{{<path to file>/}}
%% 
%% For more information, please see info/svg-inkscape on CTAN:
%%   http://tug.ctan.org/tex-archive/info/svg-inkscape
%%
\begingroup%
  \makeatletter%
  \providecommand\color[2][]{%
    \errmessage{(Inkscape) Color is used for the text in Inkscape, but the package 'color.sty' is not loaded}%
    \renewcommand\color[2][]{}%
  }%
  \providecommand\transparent[1]{%
    \errmessage{(Inkscape) Transparency is used (non-zero) for the text in Inkscape, but the package 'transparent.sty' is not loaded}%
    \renewcommand\transparent[1]{}%
  }%
  \providecommand\rotatebox[2]{#2}%
  \ifx\svgwidth\undefined%
    \setlength{\unitlength}{49.54732089bp}%
    \ifx\svgscale\undefined%
      \relax%
    \else%
      \setlength{\unitlength}{\unitlength * \real{\svgscale}}%
    \fi%
  \else%
    \setlength{\unitlength}{\svgwidth}%
  \fi%
  \global\let\svgwidth\undefined%
  \global\let\svgscale\undefined%
  \makeatother%
  \begin{picture}(1,1.00059637)%
    \put(0,0){\includegraphics[width=\unitlength,page=1]{PushOff.pdf}}%
  \end{picture}%
\endgroup%
\,\,} U_{2}(-\frac{-1}{-n-2}){\def\svgwidth{1,6ex}\,\,%% Creator: Inkscape 0.91_64bit, www.inkscape.org
%% PDF/EPS/PS + LaTeX output extension by Johan Engelen, 2010
%% Accompanies image file 'PushOff.pdf' (pdf, eps, ps)
%%
%% To include the image in your LaTeX document, write
%%   \input{<filename>.pdf_tex}
%%  instead of
%%   \includegraphics{<filename>.pdf}
%% To scale the image, write
%%   \def\svgwidth{<desired width>}
%%   \input{<filename>.pdf_tex}
%%  instead of
%%   \includegraphics[width=<desired width>]{<filename>.pdf}
%%
%% Images with a different path to the parent latex file can
%% be accessed with the `import' package (which may need to be
%% installed) using
%%   \usepackage{import}
%% in the preamble, and then including the image with
%%   \import{<path to file>}{<filename>.pdf_tex}
%% Alternatively, one can specify
%%   \graphicspath{{<path to file>/}}
%% 
%% For more information, please see info/svg-inkscape on CTAN:
%%   http://tug.ctan.org/tex-archive/info/svg-inkscape
%%
\begingroup%
  \makeatletter%
  \providecommand\color[2][]{%
    \errmessage{(Inkscape) Color is used for the text in Inkscape, but the package 'color.sty' is not loaded}%
    \renewcommand\color[2][]{}%
  }%
  \providecommand\transparent[1]{%
    \errmessage{(Inkscape) Transparency is used (non-zero) for the text in Inkscape, but the package 'transparent.sty' is not loaded}%
    \renewcommand\transparent[1]{}%
  }%
  \providecommand\rotatebox[2]{#2}%
  \ifx\svgwidth\undefined%
    \setlength{\unitlength}{49.54732089bp}%
    \ifx\svgscale\undefined%
      \relax%
    \else%
      \setlength{\unitlength}{\unitlength * \real{\svgscale}}%
    \fi%
  \else%
    \setlength{\unitlength}{\svgwidth}%
  \fi%
  \global\let\svgwidth\undefined%
  \global\let\svgscale\undefined%
  \makeatother%
  \begin{picture}(1,1.00059637)%
    \put(0,0){\includegraphics[width=\unitlength,page=1]{PushOff.pdf}}%
  \end{picture}%
\endgroup%
\,\,} U_{2,1}(-1)$&$-2$  &$<-3$  \\
 \midrule
 (9)& $U(+1){\def\svgwidth{1,6ex}\,\,%% Creator: Inkscape 0.91_64bit, www.inkscape.org
%% PDF/EPS/PS + LaTeX output extension by Johan Engelen, 2010
%% Accompanies image file 'PushOff.pdf' (pdf, eps, ps)
%%
%% To include the image in your LaTeX document, write
%%   \input{<filename>.pdf_tex}
%%  instead of
%%   \includegraphics{<filename>.pdf}
%% To scale the image, write
%%   \def\svgwidth{<desired width>}
%%   \input{<filename>.pdf_tex}
%%  instead of
%%   \includegraphics[width=<desired width>]{<filename>.pdf}
%%
%% Images with a different path to the parent latex file can
%% be accessed with the `import' package (which may need to be
%% installed) using
%%   \usepackage{import}
%% in the preamble, and then including the image with
%%   \import{<path to file>}{<filename>.pdf_tex}
%% Alternatively, one can specify
%%   \graphicspath{{<path to file>/}}
%% 
%% For more information, please see info/svg-inkscape on CTAN:
%%   http://tug.ctan.org/tex-archive/info/svg-inkscape
%%
\begingroup%
  \makeatletter%
  \providecommand\color[2][]{%
    \errmessage{(Inkscape) Color is used for the text in Inkscape, but the package 'color.sty' is not loaded}%
    \renewcommand\color[2][]{}%
  }%
  \providecommand\transparent[1]{%
    \errmessage{(Inkscape) Transparency is used (non-zero) for the text in Inkscape, but the package 'transparent.sty' is not loaded}%
    \renewcommand\transparent[1]{}%
  }%
  \providecommand\rotatebox[2]{#2}%
  \ifx\svgwidth\undefined%
    \setlength{\unitlength}{49.54732089bp}%
    \ifx\svgscale\undefined%
      \relax%
    \else%
      \setlength{\unitlength}{\unitlength * \real{\svgscale}}%
    \fi%
  \else%
    \setlength{\unitlength}{\svgwidth}%
  \fi%
  \global\let\svgwidth\undefined%
  \global\let\svgscale\undefined%
  \makeatother%
  \begin{picture}(1,1.00059637)%
    \put(0,0){\includegraphics[width=\unitlength,page=1]{PushOff.pdf}}%
  \end{picture}%
\endgroup%
\,\,} U_1(-\frac{1}{-t-2}){\def\svgwidth{1,6ex}\,\,%% Creator: Inkscape 0.91_64bit, www.inkscape.org
%% PDF/EPS/PS + LaTeX output extension by Johan Engelen, 2010
%% Accompanies image file 'PushOff.pdf' (pdf, eps, ps)
%%
%% To include the image in your LaTeX document, write
%%   \input{<filename>.pdf_tex}
%%  instead of
%%   \includegraphics{<filename>.pdf}
%% To scale the image, write
%%   \def\svgwidth{<desired width>}
%%   \input{<filename>.pdf_tex}
%%  instead of
%%   \includegraphics[width=<desired width>]{<filename>.pdf}
%%
%% Images with a different path to the parent latex file can
%% be accessed with the `import' package (which may need to be
%% installed) using
%%   \usepackage{import}
%% in the preamble, and then including the image with
%%   \import{<path to file>}{<filename>.pdf_tex}
%% Alternatively, one can specify
%%   \graphicspath{{<path to file>/}}
%% 
%% For more information, please see info/svg-inkscape on CTAN:
%%   http://tug.ctan.org/tex-archive/info/svg-inkscape
%%
\begingroup%
  \makeatletter%
  \providecommand\color[2][]{%
    \errmessage{(Inkscape) Color is used for the text in Inkscape, but the package 'color.sty' is not loaded}%
    \renewcommand\color[2][]{}%
  }%
  \providecommand\transparent[1]{%
    \errmessage{(Inkscape) Transparency is used (non-zero) for the text in Inkscape, but the package 'transparent.sty' is not loaded}%
    \renewcommand\transparent[1]{}%
  }%
  \providecommand\rotatebox[2]{#2}%
  \ifx\svgwidth\undefined%
    \setlength{\unitlength}{49.54732089bp}%
    \ifx\svgscale\undefined%
      \relax%
    \else%
      \setlength{\unitlength}{\unitlength * \real{\svgscale}}%
    \fi%
  \else%
    \setlength{\unitlength}{\svgwidth}%
  \fi%
  \global\let\svgwidth\undefined%
  \global\let\svgscale\undefined%
  \makeatother%
  \begin{picture}(1,1.00059637)%
    \put(0,0){\includegraphics[width=\unitlength,page=1]{PushOff.pdf}}%
  \end{picture}%
\endgroup%
\,\,} U_{1,1}(-\frac{1}{-n-2}){\def\svgwidth{1,6ex}\,\,%% Creator: Inkscape 0.91_64bit, www.inkscape.org
%% PDF/EPS/PS + LaTeX output extension by Johan Engelen, 2010
%% Accompanies image file 'PushOff.pdf' (pdf, eps, ps)
%%
%% To include the image in your LaTeX document, write
%%   \input{<filename>.pdf_tex}
%%  instead of
%%   \includegraphics{<filename>.pdf}
%% To scale the image, write
%%   \def\svgwidth{<desired width>}
%%   \input{<filename>.pdf_tex}
%%  instead of
%%   \includegraphics[width=<desired width>]{<filename>.pdf}
%%
%% Images with a different path to the parent latex file can
%% be accessed with the `import' package (which may need to be
%% installed) using
%%   \usepackage{import}
%% in the preamble, and then including the image with
%%   \import{<path to file>}{<filename>.pdf_tex}
%% Alternatively, one can specify
%%   \graphicspath{{<path to file>/}}
%% 
%% For more information, please see info/svg-inkscape on CTAN:
%%   http://tug.ctan.org/tex-archive/info/svg-inkscape
%%
\begingroup%
  \makeatletter%
  \providecommand\color[2][]{%
    \errmessage{(Inkscape) Color is used for the text in Inkscape, but the package 'color.sty' is not loaded}%
    \renewcommand\color[2][]{}%
  }%
  \providecommand\transparent[1]{%
    \errmessage{(Inkscape) Transparency is used (non-zero) for the text in Inkscape, but the package 'transparent.sty' is not loaded}%
    \renewcommand\transparent[1]{}%
  }%
  \providecommand\rotatebox[2]{#2}%
  \ifx\svgwidth\undefined%
    \setlength{\unitlength}{49.54732089bp}%
    \ifx\svgscale\undefined%
      \relax%
    \else%
      \setlength{\unitlength}{\unitlength * \real{\svgscale}}%
    \fi%
  \else%
    \setlength{\unitlength}{\svgwidth}%
  \fi%
  \global\let\svgwidth\undefined%
  \global\let\svgscale\undefined%
  \makeatother%
  \begin{picture}(1,1.00059637)%
    \put(0,0){\includegraphics[width=\unitlength,page=1]{PushOff.pdf}}%
  \end{picture}%
\endgroup%
\,\,} U_{1,1,1}(-1)$ &$<-2$ &$\leq -3$   \\
  \midrule
 (10)&$U(+1)$&$-3$  &$-1$ \\
 \midrule
 (11)&$U(+1){\def\svgwidth{1,6ex}\,\,%% Creator: Inkscape 0.91_64bit, www.inkscape.org
%% PDF/EPS/PS + LaTeX output extension by Johan Engelen, 2010
%% Accompanies image file 'PushOff.pdf' (pdf, eps, ps)
%%
%% To include the image in your LaTeX document, write
%%   \input{<filename>.pdf_tex}
%%  instead of
%%   \includegraphics{<filename>.pdf}
%% To scale the image, write
%%   \def\svgwidth{<desired width>}
%%   \input{<filename>.pdf_tex}
%%  instead of
%%   \includegraphics[width=<desired width>]{<filename>.pdf}
%%
%% Images with a different path to the parent latex file can
%% be accessed with the `import' package (which may need to be
%% installed) using
%%   \usepackage{import}
%% in the preamble, and then including the image with
%%   \import{<path to file>}{<filename>.pdf_tex}
%% Alternatively, one can specify
%%   \graphicspath{{<path to file>/}}
%% 
%% For more information, please see info/svg-inkscape on CTAN:
%%   http://tug.ctan.org/tex-archive/info/svg-inkscape
%%
\begingroup%
  \makeatletter%
  \providecommand\color[2][]{%
    \errmessage{(Inkscape) Color is used for the text in Inkscape, but the package 'color.sty' is not loaded}%
    \renewcommand\color[2][]{}%
  }%
  \providecommand\transparent[1]{%
    \errmessage{(Inkscape) Transparency is used (non-zero) for the text in Inkscape, but the package 'transparent.sty' is not loaded}%
    \renewcommand\transparent[1]{}%
  }%
  \providecommand\rotatebox[2]{#2}%
  \ifx\svgwidth\undefined%
    \setlength{\unitlength}{49.54732089bp}%
    \ifx\svgscale\undefined%
      \relax%
    \else%
      \setlength{\unitlength}{\unitlength * \real{\svgscale}}%
    \fi%
  \else%
    \setlength{\unitlength}{\svgwidth}%
  \fi%
  \global\let\svgwidth\undefined%
  \global\let\svgscale\undefined%
  \makeatother%
  \begin{picture}(1,1.00059637)%
    \put(0,0){\includegraphics[width=\unitlength,page=1]{PushOff.pdf}}%
  \end{picture}%
\endgroup%
\,\,} U_1(-\frac{1}{-t-3})$&$<-3$  &$-1$\\
 \bottomrule
\end{tabular}
\label{tab:tab1}
\end{table}

\begin{proof}
By Lemma~\ref{lem:smooth_diagrams_RP3}, we know that $K$ is a Legendrian unknot $U$. Legendrian unknots are classified by their Thurston--Bennequin invariant $t\leq-1$ and rotation number $r$~\cite{Eliashberg_Fraser}. Moreover, the topological surgery coefficient is $\frac{2}{2n+1}$ for some $n\in \mathbb{Z}$ and also for all $n\in \mathbb{Z}$, topological $(\frac{2}{2n+1})$-surgery along an unknot gives $\RP$.
Therefore, $(\RP,\xi)$ has contact surgery number $1$ if and only if it is obtained by a contact surgery along a Legendrian unknot $U$, with contact surgery coefficient $r_c=(\frac{2}{2n+1}-t)$, for some $n\in \mathbb{Z}$. We observe that $r_c=0$ if and only if $(t,n)=(-2,-1)$. Since contact surgery is only defined if $r_c\neq0$ we do not consider the case $(t,n)=(-2,-1)$. Then the transformation lemma implies
\begin{align*}
    U(r_c)\cong
    U\left(\frac{2-2nt-t}{2n+1}\right)
    \cong U(+1){\def\svgwidth{1,6ex}\,\,%% Creator: Inkscape 0.91_64bit, www.inkscape.org
%% PDF/EPS/PS + LaTeX output extension by Johan Engelen, 2010
%% Accompanies image file 'PushOff.pdf' (pdf, eps, ps)
%%
%% To include the image in your LaTeX document, write
%%   \input{<filename>.pdf_tex}
%%  instead of
%%   \includegraphics{<filename>.pdf}
%% To scale the image, write
%%   \def\svgwidth{<desired width>}
%%   \input{<filename>.pdf_tex}
%%  instead of
%%   \includegraphics[width=<desired width>]{<filename>.pdf}
%%
%% Images with a different path to the parent latex file can
%% be accessed with the `import' package (which may need to be
%% installed) using
%%   \usepackage{import}
%% in the preamble, and then including the image with
%%   \import{<path to file>}{<filename>.pdf_tex}
%% Alternatively, one can specify
%%   \graphicspath{{<path to file>/}}
%% 
%% For more information, please see info/svg-inkscape on CTAN:
%%   http://tug.ctan.org/tex-archive/info/svg-inkscape
%%
\begingroup%
  \makeatletter%
  \providecommand\color[2][]{%
    \errmessage{(Inkscape) Color is used for the text in Inkscape, but the package 'color.sty' is not loaded}%
    \renewcommand\color[2][]{}%
  }%
  \providecommand\transparent[1]{%
    \errmessage{(Inkscape) Transparency is used (non-zero) for the text in Inkscape, but the package 'transparent.sty' is not loaded}%
    \renewcommand\transparent[1]{}%
  }%
  \providecommand\rotatebox[2]{#2}%
  \ifx\svgwidth\undefined%
    \setlength{\unitlength}{49.54732089bp}%
    \ifx\svgscale\undefined%
      \relax%
    \else%
      \setlength{\unitlength}{\unitlength * \real{\svgscale}}%
    \fi%
  \else%
    \setlength{\unitlength}{\svgwidth}%
  \fi%
  \global\let\svgwidth\undefined%
  \global\let\svgscale\undefined%
  \makeatother%
  \begin{picture}(1,1.00059637)%
    \put(0,0){\includegraphics[width=\unitlength,page=1]{PushOff.pdf}}%
  \end{picture}%
\endgroup%
\,\,} U\left(-\frac{t(2n+1)-2}{t(2n+1)+2n-1}\right),
\end{align*}
with the latter contact surgery coefficient negative unless $t=-1$ and $n<-1$. Using induction, we prove the following negative continued fraction expansions for $P(t,n)=-\frac{t(2n+1)-2}{t(2n+1)+2n-1}$, where $n\in \mathbb{Z}$ and $t\leq-1$.
\begin{itemize}
\item $P(t, n)=[\underbrace{-2,\cdots ,-2}_\text{$-t-1$},-n-2,-2]$ for $t\leq -1, n\geq 0$,
\item $P(t, n)=[\underbrace{-2,\cdots, -2}_\text{$-t-2$},-3,\underbrace{-2,
\cdots ,-2}_\text{$-n-3$}, -3]$ for $t\leq -2, n\leq -3$,
\item $P(t, n)=[\underbrace{-2, \cdots ,-2}_\text{$-t-2$},-4]$ for $t\leq -2, n=-2$, and 
\item $P(t, n)=[\underbrace{-2,\cdots ,-2}_\text{$-t-3$}]$ for $t< -3, n=-1$.
%\item $P(t=-1, n\leq -1)$
\end{itemize}
Applying again Lemma~\ref{lem:Kirby} yields the surgery descriptions $(1)$ and $(5)$--$(11)$ from Table~\ref{tab:tab1}.

Finally, we consider the case when $t=-1$ and $n\leq -1$. Using the transformation lemma whenever required, we deduce that
\begin{itemize}
    \item $U(r_c)\cong U(-1)$ for $n=t=-1$, 
    \item  $U(r_c)\cong U(\frac{1}{3})$ for $t=-1, n=-2$,
    \item $U(r_c)\cong U(\frac{1}{2}){\def\svgwidth{1,6ex}\,\,%% Creator: Inkscape 0.91_64bit, www.inkscape.org
%% PDF/EPS/PS + LaTeX output extension by Johan Engelen, 2010
%% Accompanies image file 'PushOff.pdf' (pdf, eps, ps)
%%
%% To include the image in your LaTeX document, write
%%   \input{<filename>.pdf_tex}
%%  instead of
%%   \includegraphics{<filename>.pdf}
%% To scale the image, write
%%   \def\svgwidth{<desired width>}
%%   \input{<filename>.pdf_tex}
%%  instead of
%%   \includegraphics[width=<desired width>]{<filename>.pdf}
%%
%% Images with a different path to the parent latex file can
%% be accessed with the `import' package (which may need to be
%% installed) using
%%   \usepackage{import}
%% in the preamble, and then including the image with
%%   \import{<path to file>}{<filename>.pdf_tex}
%% Alternatively, one can specify
%%   \graphicspath{{<path to file>/}}
%% 
%% For more information, please see info/svg-inkscape on CTAN:
%%   http://tug.ctan.org/tex-archive/info/svg-inkscape
%%
\begingroup%
  \makeatletter%
  \providecommand\color[2][]{%
    \errmessage{(Inkscape) Color is used for the text in Inkscape, but the package 'color.sty' is not loaded}%
    \renewcommand\color[2][]{}%
  }%
  \providecommand\transparent[1]{%
    \errmessage{(Inkscape) Transparency is used (non-zero) for the text in Inkscape, but the package 'transparent.sty' is not loaded}%
    \renewcommand\transparent[1]{}%
  }%
  \providecommand\rotatebox[2]{#2}%
  \ifx\svgwidth\undefined%
    \setlength{\unitlength}{49.54732089bp}%
    \ifx\svgscale\undefined%
      \relax%
    \else%
      \setlength{\unitlength}{\unitlength * \real{\svgscale}}%
    \fi%
  \else%
    \setlength{\unitlength}{\svgwidth}%
  \fi%
  \global\let\svgwidth\undefined%
  \global\let\svgscale\undefined%
  \makeatother%
  \begin{picture}(1,1.00059637)%
    \put(0,0){\includegraphics[width=\unitlength,page=1]{PushOff.pdf}}%
  \end{picture}%
\endgroup%
\,\,} U(-3)\cong U(\frac{1}{2}){\def\svgwidth{1,6ex}\,\,%% Creator: Inkscape 0.91_64bit, www.inkscape.org
%% PDF/EPS/PS + LaTeX output extension by Johan Engelen, 2010
%% Accompanies image file 'PushOff.pdf' (pdf, eps, ps)
%%
%% To include the image in your LaTeX document, write
%%   \input{<filename>.pdf_tex}
%%  instead of
%%   \includegraphics{<filename>.pdf}
%% To scale the image, write
%%   \def\svgwidth{<desired width>}
%%   \input{<filename>.pdf_tex}
%%  instead of
%%   \includegraphics[width=<desired width>]{<filename>.pdf}
%%
%% Images with a different path to the parent latex file can
%% be accessed with the `import' package (which may need to be
%% installed) using
%%   \usepackage{import}
%% in the preamble, and then including the image with
%%   \import{<path to file>}{<filename>.pdf_tex}
%% Alternatively, one can specify
%%   \graphicspath{{<path to file>/}}
%% 
%% For more information, please see info/svg-inkscape on CTAN:
%%   http://tug.ctan.org/tex-archive/info/svg-inkscape
%%
\begingroup%
  \makeatletter%
  \providecommand\color[2][]{%
    \errmessage{(Inkscape) Color is used for the text in Inkscape, but the package 'color.sty' is not loaded}%
    \renewcommand\color[2][]{}%
  }%
  \providecommand\transparent[1]{%
    \errmessage{(Inkscape) Transparency is used (non-zero) for the text in Inkscape, but the package 'transparent.sty' is not loaded}%
    \renewcommand\transparent[1]{}%
  }%
  \providecommand\rotatebox[2]{#2}%
  \ifx\svgwidth\undefined%
    \setlength{\unitlength}{49.54732089bp}%
    \ifx\svgscale\undefined%
      \relax%
    \else%
      \setlength{\unitlength}{\unitlength * \real{\svgscale}}%
    \fi%
  \else%
    \setlength{\unitlength}{\svgwidth}%
  \fi%
  \global\let\svgwidth\undefined%
  \global\let\svgscale\undefined%
  \makeatother%
  \begin{picture}(1,1.00059637)%
    \put(0,0){\includegraphics[width=\unitlength,page=1]{PushOff.pdf}}%
  \end{picture}%
\endgroup%
\,\,} U_2(-1)$ for $t=-1, n=-3$,
    \item $U(r_c)\cong U(\frac{1}{2}){\def\svgwidth{1,6ex}\,\,%% Creator: Inkscape 0.91_64bit, www.inkscape.org
%% PDF/EPS/PS + LaTeX output extension by Johan Engelen, 2010
%% Accompanies image file 'PushOff.pdf' (pdf, eps, ps)
%%
%% To include the image in your LaTeX document, write
%%   \input{<filename>.pdf_tex}
%%  instead of
%%   \includegraphics{<filename>.pdf}
%% To scale the image, write
%%   \def\svgwidth{<desired width>}
%%   \input{<filename>.pdf_tex}
%%  instead of
%%   \includegraphics[width=<desired width>]{<filename>.pdf}
%%
%% Images with a different path to the parent latex file can
%% be accessed with the `import' package (which may need to be
%% installed) using
%%   \usepackage{import}
%% in the preamble, and then including the image with
%%   \import{<path to file>}{<filename>.pdf_tex}
%% Alternatively, one can specify
%%   \graphicspath{{<path to file>/}}
%% 
%% For more information, please see info/svg-inkscape on CTAN:
%%   http://tug.ctan.org/tex-archive/info/svg-inkscape
%%
\begingroup%
  \makeatletter%
  \providecommand\color[2][]{%
    \errmessage{(Inkscape) Color is used for the text in Inkscape, but the package 'color.sty' is not loaded}%
    \renewcommand\color[2][]{}%
  }%
  \providecommand\transparent[1]{%
    \errmessage{(Inkscape) Transparency is used (non-zero) for the text in Inkscape, but the package 'transparent.sty' is not loaded}%
    \renewcommand\transparent[1]{}%
  }%
  \providecommand\rotatebox[2]{#2}%
  \ifx\svgwidth\undefined%
    \setlength{\unitlength}{49.54732089bp}%
    \ifx\svgscale\undefined%
      \relax%
    \else%
      \setlength{\unitlength}{\unitlength * \real{\svgscale}}%
    \fi%
  \else%
    \setlength{\unitlength}{\svgwidth}%
  \fi%
  \global\let\svgwidth\undefined%
  \global\let\svgscale\undefined%
  \makeatother%
  \begin{picture}(1,1.00059637)%
    \put(0,0){\includegraphics[width=\unitlength,page=1]{PushOff.pdf}}%
  \end{picture}%
\endgroup%
\,\,} U(\frac{2n+3}{-2n-5})$ for $t=-1, n\leq -4$. Note that for $n\leq -4$, $\frac{2n+3}{-2n-5}<0$ and with induction we show that $$\frac{2n+3}{-2n-5}=[-2,\underbrace{-2\cdots-2}_{|n|-4},-3].$$ Thus the translation lemma yields $U(r_c) \cong U(\frac{1}{2}){\def\svgwidth{1,6ex}\,\,%% Creator: Inkscape 0.91_64bit, www.inkscape.org
%% PDF/EPS/PS + LaTeX output extension by Johan Engelen, 2010
%% Accompanies image file 'PushOff.pdf' (pdf, eps, ps)
%%
%% To include the image in your LaTeX document, write
%%   \input{<filename>.pdf_tex}
%%  instead of
%%   \includegraphics{<filename>.pdf}
%% To scale the image, write
%%   \def\svgwidth{<desired width>}
%%   \input{<filename>.pdf_tex}
%%  instead of
%%   \includegraphics[width=<desired width>]{<filename>.pdf}
%%
%% Images with a different path to the parent latex file can
%% be accessed with the `import' package (which may need to be
%% installed) using
%%   \usepackage{import}
%% in the preamble, and then including the image with
%%   \import{<path to file>}{<filename>.pdf_tex}
%% Alternatively, one can specify
%%   \graphicspath{{<path to file>/}}
%% 
%% For more information, please see info/svg-inkscape on CTAN:
%%   http://tug.ctan.org/tex-archive/info/svg-inkscape
%%
\begingroup%
  \makeatletter%
  \providecommand\color[2][]{%
    \errmessage{(Inkscape) Color is used for the text in Inkscape, but the package 'color.sty' is not loaded}%
    \renewcommand\color[2][]{}%
  }%
  \providecommand\transparent[1]{%
    \errmessage{(Inkscape) Transparency is used (non-zero) for the text in Inkscape, but the package 'transparent.sty' is not loaded}%
    \renewcommand\transparent[1]{}%
  }%
  \providecommand\rotatebox[2]{#2}%
  \ifx\svgwidth\undefined%
    \setlength{\unitlength}{49.54732089bp}%
    \ifx\svgscale\undefined%
      \relax%
    \else%
      \setlength{\unitlength}{\unitlength * \real{\svgscale}}%
    \fi%
  \else%
    \setlength{\unitlength}{\svgwidth}%
  \fi%
  \global\let\svgwidth\undefined%
  \global\let\svgscale\undefined%
  \makeatother%
  \begin{picture}(1,1.00059637)%
    \put(0,0){\includegraphics[width=\unitlength,page=1]{PushOff.pdf}}%
  \end{picture}%
\endgroup%
\,\,} U_1(\frac{-1}{|n|-3}){\def\svgwidth{1,6ex}\,\,%% Creator: Inkscape 0.91_64bit, www.inkscape.org
%% PDF/EPS/PS + LaTeX output extension by Johan Engelen, 2010
%% Accompanies image file 'PushOff.pdf' (pdf, eps, ps)
%%
%% To include the image in your LaTeX document, write
%%   \input{<filename>.pdf_tex}
%%  instead of
%%   \includegraphics{<filename>.pdf}
%% To scale the image, write
%%   \def\svgwidth{<desired width>}
%%   \input{<filename>.pdf_tex}
%%  instead of
%%   \includegraphics[width=<desired width>]{<filename>.pdf}
%%
%% Images with a different path to the parent latex file can
%% be accessed with the `import' package (which may need to be
%% installed) using
%%   \usepackage{import}
%% in the preamble, and then including the image with
%%   \import{<path to file>}{<filename>.pdf_tex}
%% Alternatively, one can specify
%%   \graphicspath{{<path to file>/}}
%% 
%% For more information, please see info/svg-inkscape on CTAN:
%%   http://tug.ctan.org/tex-archive/info/svg-inkscape
%%
\begingroup%
  \makeatletter%
  \providecommand\color[2][]{%
    \errmessage{(Inkscape) Color is used for the text in Inkscape, but the package 'color.sty' is not loaded}%
    \renewcommand\color[2][]{}%
  }%
  \providecommand\transparent[1]{%
    \errmessage{(Inkscape) Transparency is used (non-zero) for the text in Inkscape, but the package 'transparent.sty' is not loaded}%
    \renewcommand\transparent[1]{}%
  }%
  \providecommand\rotatebox[2]{#2}%
  \ifx\svgwidth\undefined%
    \setlength{\unitlength}{49.54732089bp}%
    \ifx\svgscale\undefined%
      \relax%
    \else%
      \setlength{\unitlength}{\unitlength * \real{\svgscale}}%
    \fi%
  \else%
    \setlength{\unitlength}{\svgwidth}%
  \fi%
  \global\let\svgwidth\undefined%
  \global\let\svgscale\undefined%
  \makeatother%
  \begin{picture}(1,1.00059637)%
    \put(0,0){\includegraphics[width=\unitlength,page=1]{PushOff.pdf}}%
  \end{picture}%
\endgroup%
\,\,} U_{1,1}(-1)$ for $t=-1, n\leq -4$.
\end{itemize}
This yields the surgery descriptions $(0)$ and $(2)$--$(4)$ from Table~\ref{tab:tab1} and finishes the proof.
\end{proof}

\section{Computing the homotopical invariants}\label{sec:computing}

In this section, we will use Lemma~\ref{lem:d3} and~\ref{lem:Gamma} to compute the $\Gamma$- and the $\de_3$-invariants for all contact structures given by the surgery diagrams from Table~\ref{tab:tab1}. 
To effectively compare the $\Gamma$-invariants, we fix a spin structure $\mathfrak{s}_0$ on $\RP$ and compute all $\Gamma$-invariants with respect to $\mathfrak{s}_0$. For that, we consider the \textit{standard} surgery diagram of $\RP$ consisting of a single unknot with topological surgery coefficient $-2$. In that surgery diagram, the empty link is a characteristic sublink defining a spin structure $\mathfrak{s}_0$. If $\xi$ is a contact structure on $\RP$ then we define $\Gamma(\xi)\in\Z_2$ to be $\Gamma(\xi,\mathfrak{s}_0)\in H_1(\RP)\cong\Z_2$. (The other characteristic sublink is given by the whole link which defines another spin structure, say, $\mathfrak{s}_1$. It would also be possible to perform all calculations with respect to $\mathfrak{s}_1$, then the concrete values of $\Gamma$ would all change.)

\begin{lem}
    The possible values of the pairs of $(\Gamma,\de_3)$ of the contact structures given by the surgery diagrams from Table~\ref{tab:tab1} are as shown in Table~\ref{tab:tab2}.
\end{lem}

\begin{table}[htbp]
    \caption{The $\Gamma$- and $d_3$-invariants for the cases from Table \ref{tab:tab1}}
    \label{tab:tab2}
\begin{tabular}{ c|l} 
Case & $(\Gamma,d_3$)\\
\toprule
(0)&$\left(0,\frac{1}{4}\right)$\\
\midrule
  (1)& $\left(0,\frac{1}{4}\right)$ \\
    \midrule
 (2)& $\left(1,\frac{3}{4}\right)$\\
 \midrule
(3)& $\left(0,1+\frac{1}{4}\right)$, $\left(1,\frac{3}{4}\right)$ \\
\midrule
(4) & $\left(0,1+\frac{1}{4}\right),\,\left(1,\frac{3}{4}\right)$ \\ 
\midrule
 \multirow{8}{1.4em}{(5)}& $\left(0,2m^2(2n+1)+4n+2m(4n+1)\pm2x( m+1)+\frac{1}{4}\right)$,\\
&for $n\geq0$, $m< -1$, and \\
& $x =2i$, for $i=0,1,\ldots, \frac{n}{2}$ if  $n$ is even, and\\
& $x=(2i+1)$, for $i=0,1,\ldots, \left\lfloor \frac{n}{2}\right\rfloor$ if $n$ is odd;\\
 & $\left(1,2m^2(2n+1)+n(4m+1)-1\pm x( 2m+1)+\frac{3}{4}\right)$,\\
 &for $n\geq0$, $m\leq -1$, and \\
& $x =2i$, for $i=0,1,\ldots, \frac{n}{2}$ if  $n$ is even, and\\
& $x=(2i+1)$, for $i=0,1,\ldots, \left\lfloor \frac{n}{2}\right\rfloor$ if $n$ is odd\\
\midrule
  (6)& $(0,1+\frac{1}{4}),(0,-1+\frac{1}{4}),(1,
 \frac{3}{4})$\\
\midrule
\multirow{4}{1.4em}{(7)}  &$\left(0,-6m^2-14m-7+\frac{1}{4}\right)$, $\left(0,-6m^2-6m-1+\frac{1}{4}\right)$\\
&$\left(0,-6m^2-10m-3+\frac{1}{4}\right), \left(1,-6m^2-16m-10+\frac{3}{4}\right)$,\\
&$\left(1, -6m^2-8m-2+\frac{3}{4}\right)$, $\left(1,-6m^2-12m-6+\frac{3}{4}\right)$, \\ &for $m\leq -1$\\
\midrule
\multirow{2}{1.4em}{(8)}  & $ \left(0,1+\frac{1}{4}\right),\, \left(0,2n+3+\frac{1}{4}\right),\, \left(1,\frac{3}{4}\right), \,\left(1,2n+4+\frac{3}{4}\right)$,\\
 &for $n<-3$\\
 \midrule
 \multirow{8}{1.4em}{(9)}&  $\left(0,2m^2(1+2n)+2m(n-1)-1+\frac{1}{4}\right),$ \\
 &  $\left(0,2m^2(1+2n)+2m(3n-1)+2n-3+\frac{1}{4}\right),$\\
 &  $\left(0, 2m^2(1+2n)+2m(3n+1)+2n+1+\frac{1}{4}\right),$\\
  &$\left(0, 2m^2(1+2n)+2m(5n+3)+6n+5+\frac{1}{4}\right)$,\\
 &$\left(1,2m^2(1+2n)+6mn+2n-2+\frac{3}{4}\right),$\\
  &$\left(1,2m^2(1+2n)+2m(n-2)-2+\frac{3}{4}\right)$,\\
  &$\left(1, 2m^2(1+2n)+2m(5n+2)+6n+2+\frac{3}{4}\right),$\\
 &  $\left(1, 2m^2(1+2n)+2m(3n+2)+2n+2+\frac{3}{4}\right),$\\
 &for $m\leq -1$ and $n\leq-3$\\
  \midrule
 (10)&$\left(0, 1+\frac{1}{4}),\,(1,\frac{3}{4}\right)$\\
 \midrule
 \multirow{2}{1.4em}{(11)} &  $\left(0,-2m^2-4m-1+\frac{1}{4}\right),\, \left(1,-2m^2-6m-4+\frac{3}{4}\right)$,\\
&for $m\leq -1\:$\\
 \bottomrule
\end{tabular}
\end{table}

\begin{proof}
In each of the $12$ cases from Table~\ref{tab:tab1}, we first use Lemma~\ref{lem:d3} to compute the possible values of the $\de_3$-invariant. For computing the $\Gamma$-invariant we then proceed by using the transformation lemma to convert the rational contact surgery diagram into a contact $(\pm1)$-surgery diagram $L$. Then we perform smooth Kirby calculus to transform $L$ into the standard surgery diagram of $\RP$. Following these Kirby moves backwards we can describe the characteristic sublink of $L$ that corresponds to $\mathfrak s_0$. In the surgery diagrams below we mark the components of the characteristic sublinks with a star. To distinguish between different cases, we use stars of different colors. For example, in Figure~\ref{fig:c1}, the characteristic sublink differs depending on whether $n$ is odd or even. 

In Figures~\ref{fig:c1}--\ref{fig:c12}, starting from the left we have the contact $(\pm1)$-surgery diagram, then the corresponding topological surgery diagram with the characteristic sublink marked followed by a sequence of Kirby moves connecting it to the standard surgery diagram. In these figures, $RT_n$ stands for an $n$-fold Rolfsen twist, and $BD$ for a blow down. Once we have the characteristic sublink describing $\mathfrak s_0$, we use Lemma~\ref{lem:Gamma} to compute the $\Gamma$-invariant.

Next, we consider the twelve cases separately. We always denote by $t$ and $r=r_1$ the Thurston--Bennequin invariant and rotation number of $U$, the first Legendrian knot in the surgery description, and by $r_i$ and $\mu_i$ the rotation number and meridian of the $i^{th}$ knot in the surgery description.

\noindent \textbf{Case (0):} This is contact $(-1)$-surgery on a single Legendrian knot with $t=-1$ and thus the computation of the $\de_3$-invariant is straightforward. Since topologically the surgery diagram is just an unknot with topological surgery coefficient $-2$, the characteristic sublink corresponding to $\mathfrak{s}_0$ is empty. Because the rotation number $r$ of $U$ is zero, it follows that $\Gamma=0$.

\noindent \textbf{Case (1):} In this case the generalized linking matrix is
 \begin{align*}
    Q=\begin{pmatrix}
        0&-2\\
        -1&-5-2n
    \end{pmatrix}.
    \end{align*}
It has vanishing signature and $Q^{-1}\mathbf{r}=(-r_2,0)$ for $r_2$ the rotation number of $U_{n+1}$, from which it is straightforward to compute $d_3=\frac{1}{4}$.

Using the transformation lemma we write this contact surgery diagram as 
$$U(+1){\def\svgwidth{1,6ex}\,\,%% Creator: Inkscape 0.91_64bit, www.inkscape.org
%% PDF/EPS/PS + LaTeX output extension by Johan Engelen, 2010
%% Accompanies image file 'PushOff.pdf' (pdf, eps, ps)
%%
%% To include the image in your LaTeX document, write
%%   \input{<filename>.pdf_tex}
%%  instead of
%%   \includegraphics{<filename>.pdf}
%% To scale the image, write
%%   \def\svgwidth{<desired width>}
%%   \input{<filename>.pdf_tex}
%%  instead of
%%   \includegraphics[width=<desired width>]{<filename>.pdf}
%%
%% Images with a different path to the parent latex file can
%% be accessed with the `import' package (which may need to be
%% installed) using
%%   \usepackage{import}
%% in the preamble, and then including the image with
%%   \import{<path to file>}{<filename>.pdf_tex}
%% Alternatively, one can specify
%%   \graphicspath{{<path to file>/}}
%% 
%% For more information, please see info/svg-inkscape on CTAN:
%%   http://tug.ctan.org/tex-archive/info/svg-inkscape
%%
\begingroup%
  \makeatletter%
  \providecommand\color[2][]{%
    \errmessage{(Inkscape) Color is used for the text in Inkscape, but the package 'color.sty' is not loaded}%
    \renewcommand\color[2][]{}%
  }%
  \providecommand\transparent[1]{%
    \errmessage{(Inkscape) Transparency is used (non-zero) for the text in Inkscape, but the package 'transparent.sty' is not loaded}%
    \renewcommand\transparent[1]{}%
  }%
  \providecommand\rotatebox[2]{#2}%
  \ifx\svgwidth\undefined%
    \setlength{\unitlength}{49.54732089bp}%
    \ifx\svgscale\undefined%
      \relax%
    \else%
      \setlength{\unitlength}{\unitlength * \real{\svgscale}}%
    \fi%
  \else%
    \setlength{\unitlength}{\svgwidth}%
  \fi%
  \global\let\svgwidth\undefined%
  \global\let\svgscale\undefined%
  \makeatother%
  \begin{picture}(1,1.00059637)%
    \put(0,0){\includegraphics[width=\unitlength,page=1]{PushOff.pdf}}%
  \end{picture}%
\endgroup%
\,\,} U_{n+1}(-1){\def\svgwidth{1,6ex}\,\,%% Creator: Inkscape 0.91_64bit, www.inkscape.org
%% PDF/EPS/PS + LaTeX output extension by Johan Engelen, 2010
%% Accompanies image file 'PushOff.pdf' (pdf, eps, ps)
%%
%% To include the image in your LaTeX document, write
%%   \input{<filename>.pdf_tex}
%%  instead of
%%   \includegraphics{<filename>.pdf}
%% To scale the image, write
%%   \def\svgwidth{<desired width>}
%%   \input{<filename>.pdf_tex}
%%  instead of
%%   \includegraphics[width=<desired width>]{<filename>.pdf}
%%
%% Images with a different path to the parent latex file can
%% be accessed with the `import' package (which may need to be
%% installed) using
%%   \usepackage{import}
%% in the preamble, and then including the image with
%%   \import{<path to file>}{<filename>.pdf_tex}
%% Alternatively, one can specify
%%   \graphicspath{{<path to file>/}}
%% 
%% For more information, please see info/svg-inkscape on CTAN:
%%   http://tug.ctan.org/tex-archive/info/svg-inkscape
%%
\begingroup%
  \makeatletter%
  \providecommand\color[2][]{%
    \errmessage{(Inkscape) Color is used for the text in Inkscape, but the package 'color.sty' is not loaded}%
    \renewcommand\color[2][]{}%
  }%
  \providecommand\transparent[1]{%
    \errmessage{(Inkscape) Transparency is used (non-zero) for the text in Inkscape, but the package 'transparent.sty' is not loaded}%
    \renewcommand\transparent[1]{}%
  }%
  \providecommand\rotatebox[2]{#2}%
  \ifx\svgwidth\undefined%
    \setlength{\unitlength}{49.54732089bp}%
    \ifx\svgscale\undefined%
      \relax%
    \else%
      \setlength{\unitlength}{\unitlength * \real{\svgscale}}%
    \fi%
  \else%
    \setlength{\unitlength}{\svgwidth}%
  \fi%
  \global\let\svgwidth\undefined%
  \global\let\svgscale\undefined%
  \makeatother%
  \begin{picture}(1,1.00059637)%
    \put(0,0){\includegraphics[width=\unitlength,page=1]{PushOff.pdf}}%
  \end{picture}%
\endgroup%
\,\,} U_{n+1}(-1),$$ 
see Figure~\ref{fig:c1}. 
Here, we see that $r=0$ and $r_2=r_3$ have opposite parity as $n$. From the linking matrix, we get a presentation for the first homology from which we deduce that $\mu_2=\mu_3$ is a generator and thus we compute
\begin{align*}
\Gamma=\begin{cases}
    \frac{1}{2}\big(r_2\mu_2+r_3\mu_3-\mu_2-\mu_3\big)=(r_2-1)\,\mu_2=0, & \textrm{ if } n  \textrm{ is even, }\\
    \frac{1}{2}\big(r_2\mu_2+r_3\mu_3\big)=r_2\,\mu_2=0, & \textrm{ if } n \textrm{ is odd. }
\end{cases}
\end{align*}

Alternatively, in the proof of Lemma~\ref{lem:tight}, we will also use contact Kirby moves to show that all contact surgery diagrams from Case $(1)$ yields the same contact structure as Case $(0)$, which explains why we get the same values for the homotopical invariants.
\begin{figure}[htbp]
    \centering
    \includegraphics[width=1\linewidth]{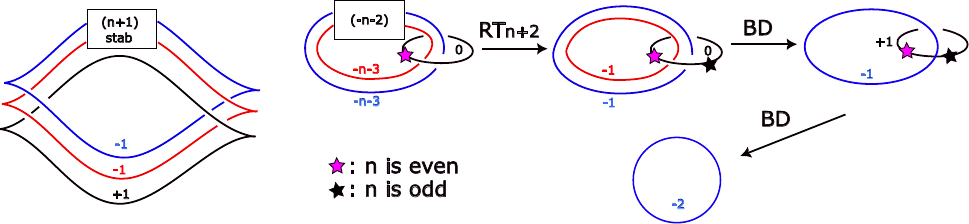}
    \caption{Case (1)}
     \label{fig:c1}
\end{figure}

 \noindent \textbf{Case (2):}  This is contact $(1/3)$-surgery on a single Legendrian knot and thus the computation of the $\de_3$-invariant is straightforward. Using the transformation lemma, we write this contact surgery diagram as 
$$U(+1){\def\svgwidth{1,6ex}\,\,%% Creator: Inkscape 0.91_64bit, www.inkscape.org
%% PDF/EPS/PS + LaTeX output extension by Johan Engelen, 2010
%% Accompanies image file 'PushOff.pdf' (pdf, eps, ps)
%%
%% To include the image in your LaTeX document, write
%%   \input{<filename>.pdf_tex}
%%  instead of
%%   \includegraphics{<filename>.pdf}
%% To scale the image, write
%%   \def\svgwidth{<desired width>}
%%   \input{<filename>.pdf_tex}
%%  instead of
%%   \includegraphics[width=<desired width>]{<filename>.pdf}
%%
%% Images with a different path to the parent latex file can
%% be accessed with the `import' package (which may need to be
%% installed) using
%%   \usepackage{import}
%% in the preamble, and then including the image with
%%   \import{<path to file>}{<filename>.pdf_tex}
%% Alternatively, one can specify
%%   \graphicspath{{<path to file>/}}
%% 
%% For more information, please see info/svg-inkscape on CTAN:
%%   http://tug.ctan.org/tex-archive/info/svg-inkscape
%%
\begingroup%
  \makeatletter%
  \providecommand\color[2][]{%
    \errmessage{(Inkscape) Color is used for the text in Inkscape, but the package 'color.sty' is not loaded}%
    \renewcommand\color[2][]{}%
  }%
  \providecommand\transparent[1]{%
    \errmessage{(Inkscape) Transparency is used (non-zero) for the text in Inkscape, but the package 'transparent.sty' is not loaded}%
    \renewcommand\transparent[1]{}%
  }%
  \providecommand\rotatebox[2]{#2}%
  \ifx\svgwidth\undefined%
    \setlength{\unitlength}{49.54732089bp}%
    \ifx\svgscale\undefined%
      \relax%
    \else%
      \setlength{\unitlength}{\unitlength * \real{\svgscale}}%
    \fi%
  \else%
    \setlength{\unitlength}{\svgwidth}%
  \fi%
  \global\let\svgwidth\undefined%
  \global\let\svgscale\undefined%
  \makeatother%
  \begin{picture}(1,1.00059637)%
    \put(0,0){\includegraphics[width=\unitlength,page=1]{PushOff.pdf}}%
  \end{picture}%
\endgroup%
\,\,} U(+1){\def\svgwidth{1,6ex}\,\,%% Creator: Inkscape 0.91_64bit, www.inkscape.org
%% PDF/EPS/PS + LaTeX output extension by Johan Engelen, 2010
%% Accompanies image file 'PushOff.pdf' (pdf, eps, ps)
%%
%% To include the image in your LaTeX document, write
%%   \input{<filename>.pdf_tex}
%%  instead of
%%   \includegraphics{<filename>.pdf}
%% To scale the image, write
%%   \def\svgwidth{<desired width>}
%%   \input{<filename>.pdf_tex}
%%  instead of
%%   \includegraphics[width=<desired width>]{<filename>.pdf}
%%
%% Images with a different path to the parent latex file can
%% be accessed with the `import' package (which may need to be
%% installed) using
%%   \usepackage{import}
%% in the preamble, and then including the image with
%%   \import{<path to file>}{<filename>.pdf_tex}
%% Alternatively, one can specify
%%   \graphicspath{{<path to file>/}}
%% 
%% For more information, please see info/svg-inkscape on CTAN:
%%   http://tug.ctan.org/tex-archive/info/svg-inkscape
%%
\begingroup%
  \makeatletter%
  \providecommand\color[2][]{%
    \errmessage{(Inkscape) Color is used for the text in Inkscape, but the package 'color.sty' is not loaded}%
    \renewcommand\color[2][]{}%
  }%
  \providecommand\transparent[1]{%
    \errmessage{(Inkscape) Transparency is used (non-zero) for the text in Inkscape, but the package 'transparent.sty' is not loaded}%
    \renewcommand\transparent[1]{}%
  }%
  \providecommand\rotatebox[2]{#2}%
  \ifx\svgwidth\undefined%
    \setlength{\unitlength}{49.54732089bp}%
    \ifx\svgscale\undefined%
      \relax%
    \else%
      \setlength{\unitlength}{\unitlength * \real{\svgscale}}%
    \fi%
  \else%
    \setlength{\unitlength}{\svgwidth}%
  \fi%
  \global\let\svgwidth\undefined%
  \global\let\svgscale\undefined%
  \makeatother%
  \begin{picture}(1,1.00059637)%
    \put(0,0){\includegraphics[width=\unitlength,page=1]{PushOff.pdf}}%
  \end{picture}%
\endgroup%
\,\,} U(+1),$$
see Figure~\ref{fig:c3}.
In this case, all rotation numbers are vanishing. From the linking matrix we deduce that $\mu_1=-\mu_3=\mu_2$ is a generator, and in Figure~\ref{fig:c3} we see that the characteristic sublink is the whole link. Thus, we compute
\begin{align*}
\Gamma=-\big(\mu_1+\mu_2+\mu_3\big)=1.
\end{align*}
\begin{figure}[htbp]
    \centering
    \includegraphics[width=1\linewidth]{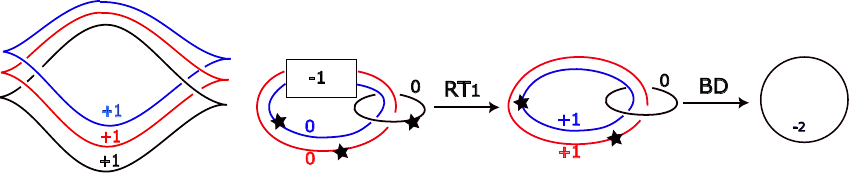}
    \caption{Case (2)}
    \label{fig:c3}  
\end{figure}

 \noindent \textbf{Case (3):} Here the generalized linking matrix is
 \begin{align*}
    Q=\begin{pmatrix}
        -1&-1\\
        -2&-4
    \end{pmatrix}
    \end{align*}
with signature $\sigma(Q)=-2$. The rotation number $r_2$ of $U_2$ can take values $r_2\in\{-2,0, 2\}$, from which we compute the possible values of $\de_3$. 

Using the transformation lemma, we write this contact surgery diagram as 
$$U(+1){\def\svgwidth{1,6ex}\,\,%% Creator: Inkscape 0.91_64bit, www.inkscape.org
%% PDF/EPS/PS + LaTeX output extension by Johan Engelen, 2010
%% Accompanies image file 'PushOff.pdf' (pdf, eps, ps)
%%
%% To include the image in your LaTeX document, write
%%   \input{<filename>.pdf_tex}
%%  instead of
%%   \includegraphics{<filename>.pdf}
%% To scale the image, write
%%   \def\svgwidth{<desired width>}
%%   \input{<filename>.pdf_tex}
%%  instead of
%%   \includegraphics[width=<desired width>]{<filename>.pdf}
%%
%% Images with a different path to the parent latex file can
%% be accessed with the `import' package (which may need to be
%% installed) using
%%   \usepackage{import}
%% in the preamble, and then including the image with
%%   \import{<path to file>}{<filename>.pdf_tex}
%% Alternatively, one can specify
%%   \graphicspath{{<path to file>/}}
%% 
%% For more information, please see info/svg-inkscape on CTAN:
%%   http://tug.ctan.org/tex-archive/info/svg-inkscape
%%
\begingroup%
  \makeatletter%
  \providecommand\color[2][]{%
    \errmessage{(Inkscape) Color is used for the text in Inkscape, but the package 'color.sty' is not loaded}%
    \renewcommand\color[2][]{}%
  }%
  \providecommand\transparent[1]{%
    \errmessage{(Inkscape) Transparency is used (non-zero) for the text in Inkscape, but the package 'transparent.sty' is not loaded}%
    \renewcommand\transparent[1]{}%
  }%
  \providecommand\rotatebox[2]{#2}%
  \ifx\svgwidth\undefined%
    \setlength{\unitlength}{49.54732089bp}%
    \ifx\svgscale\undefined%
      \relax%
    \else%
      \setlength{\unitlength}{\unitlength * \real{\svgscale}}%
    \fi%
  \else%
    \setlength{\unitlength}{\svgwidth}%
  \fi%
  \global\let\svgwidth\undefined%
  \global\let\svgscale\undefined%
  \makeatother%
  \begin{picture}(1,1.00059637)%
    \put(0,0){\includegraphics[width=\unitlength,page=1]{PushOff.pdf}}%
  \end{picture}%
\endgroup%
\,\,} U(+1){\def\svgwidth{1,6ex}\,\,%% Creator: Inkscape 0.91_64bit, www.inkscape.org
%% PDF/EPS/PS + LaTeX output extension by Johan Engelen, 2010
%% Accompanies image file 'PushOff.pdf' (pdf, eps, ps)
%%
%% To include the image in your LaTeX document, write
%%   \input{<filename>.pdf_tex}
%%  instead of
%%   \includegraphics{<filename>.pdf}
%% To scale the image, write
%%   \def\svgwidth{<desired width>}
%%   \input{<filename>.pdf_tex}
%%  instead of
%%   \includegraphics[width=<desired width>]{<filename>.pdf}
%%
%% Images with a different path to the parent latex file can
%% be accessed with the `import' package (which may need to be
%% installed) using
%%   \usepackage{import}
%% in the preamble, and then including the image with
%%   \import{<path to file>}{<filename>.pdf_tex}
%% Alternatively, one can specify
%%   \graphicspath{{<path to file>/}}
%% 
%% For more information, please see info/svg-inkscape on CTAN:
%%   http://tug.ctan.org/tex-archive/info/svg-inkscape
%%
\begingroup%
  \makeatletter%
  \providecommand\color[2][]{%
    \errmessage{(Inkscape) Color is used for the text in Inkscape, but the package 'color.sty' is not loaded}%
    \renewcommand\color[2][]{}%
  }%
  \providecommand\transparent[1]{%
    \errmessage{(Inkscape) Transparency is used (non-zero) for the text in Inkscape, but the package 'transparent.sty' is not loaded}%
    \renewcommand\transparent[1]{}%
  }%
  \providecommand\rotatebox[2]{#2}%
  \ifx\svgwidth\undefined%
    \setlength{\unitlength}{49.54732089bp}%
    \ifx\svgscale\undefined%
      \relax%
    \else%
      \setlength{\unitlength}{\unitlength * \real{\svgscale}}%
    \fi%
  \else%
    \setlength{\unitlength}{\svgwidth}%
  \fi%
  \global\let\svgwidth\undefined%
  \global\let\svgscale\undefined%
  \makeatother%
  \begin{picture}(1,1.00059637)%
    \put(0,0){\includegraphics[width=\unitlength,page=1]{PushOff.pdf}}%
  \end{picture}%
\endgroup%
\,\,} U_2(-1),$$ 
with $t=-1$, see Figure~\ref{fig:c4}.
In this case, $r_1=r_2=0$ and $r_3\in\{0,\pm2\}$, the characteristic sublink is empty, and $\mu_3$ is a generator. Thus we compute the $\Gamma$-invariant as follows:
\begin{align*}
\Gamma=\frac{1}{2}\big(r_3\mu_3\big)=\begin{cases}
    0, & \textrm{ if } r_3=0,  \\
    1, & \textrm{ if } r_3=\pm2. 
\end{cases}
\end{align*}
\begin{figure}[htbp]
    \centering
    \includegraphics[width=1\linewidth]{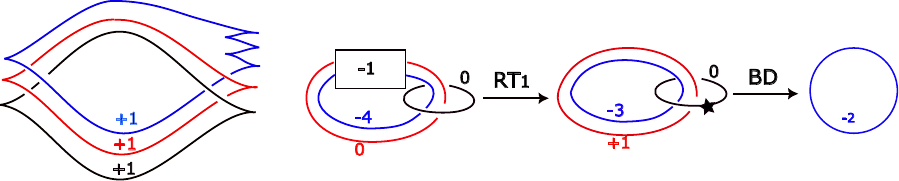}
    \caption{Case (3)}
    \label{fig:c4}
\end{figure}

 \noindent \textbf{Case (4):} In this case, the generalized linking matrix is
 \begin{align*}
    Q=\begin{pmatrix}
        -1&-|n|+3&-1\\
        -2&-2|n|+5&-2\\
        -2&-2|n|+6&-4
    \end{pmatrix}
    \end{align*}
with signature $\sigma(Q)=-3$. The rotation numbers $r_2$ of $U_1$ and $r_3$ of $U_{1,1}$ can take values $r_2=\pm 1$ and $r_3=r_2\pm 1$ from which we compute the possible values of $\de_3$. 

Using the transformation lemma we write this contact surgery diagram as 
$$U(+1){\def\svgwidth{1,6ex}\,\,%% Creator: Inkscape 0.91_64bit, www.inkscape.org
%% PDF/EPS/PS + LaTeX output extension by Johan Engelen, 2010
%% Accompanies image file 'PushOff.pdf' (pdf, eps, ps)
%%
%% To include the image in your LaTeX document, write
%%   \input{<filename>.pdf_tex}
%%  instead of
%%   \includegraphics{<filename>.pdf}
%% To scale the image, write
%%   \def\svgwidth{<desired width>}
%%   \input{<filename>.pdf_tex}
%%  instead of
%%   \includegraphics[width=<desired width>]{<filename>.pdf}
%%
%% Images with a different path to the parent latex file can
%% be accessed with the `import' package (which may need to be
%% installed) using
%%   \usepackage{import}
%% in the preamble, and then including the image with
%%   \import{<path to file>}{<filename>.pdf_tex}
%% Alternatively, one can specify
%%   \graphicspath{{<path to file>/}}
%% 
%% For more information, please see info/svg-inkscape on CTAN:
%%   http://tug.ctan.org/tex-archive/info/svg-inkscape
%%
\begingroup%
  \makeatletter%
  \providecommand\color[2][]{%
    \errmessage{(Inkscape) Color is used for the text in Inkscape, but the package 'color.sty' is not loaded}%
    \renewcommand\color[2][]{}%
  }%
  \providecommand\transparent[1]{%
    \errmessage{(Inkscape) Transparency is used (non-zero) for the text in Inkscape, but the package 'transparent.sty' is not loaded}%
    \renewcommand\transparent[1]{}%
  }%
  \providecommand\rotatebox[2]{#2}%
  \ifx\svgwidth\undefined%
    \setlength{\unitlength}{49.54732089bp}%
    \ifx\svgscale\undefined%
      \relax%
    \else%
      \setlength{\unitlength}{\unitlength * \real{\svgscale}}%
    \fi%
  \else%
    \setlength{\unitlength}{\svgwidth}%
  \fi%
  \global\let\svgwidth\undefined%
  \global\let\svgscale\undefined%
  \makeatother%
  \begin{picture}(1,1.00059637)%
    \put(0,0){\includegraphics[width=\unitlength,page=1]{PushOff.pdf}}%
  \end{picture}%
\endgroup%
\,\,} U(+1){\def\svgwidth{1,6ex}\,\,%% Creator: Inkscape 0.91_64bit, www.inkscape.org
%% PDF/EPS/PS + LaTeX output extension by Johan Engelen, 2010
%% Accompanies image file 'PushOff.pdf' (pdf, eps, ps)
%%
%% To include the image in your LaTeX document, write
%%   \input{<filename>.pdf_tex}
%%  instead of
%%   \includegraphics{<filename>.pdf}
%% To scale the image, write
%%   \def\svgwidth{<desired width>}
%%   \input{<filename>.pdf_tex}
%%  instead of
%%   \includegraphics[width=<desired width>]{<filename>.pdf}
%%
%% Images with a different path to the parent latex file can
%% be accessed with the `import' package (which may need to be
%% installed) using
%%   \usepackage{import}
%% in the preamble, and then including the image with
%%   \import{<path to file>}{<filename>.pdf_tex}
%% Alternatively, one can specify
%%   \graphicspath{{<path to file>/}}
%% 
%% For more information, please see info/svg-inkscape on CTAN:
%%   http://tug.ctan.org/tex-archive/info/svg-inkscape
%%
\begingroup%
  \makeatletter%
  \providecommand\color[2][]{%
    \errmessage{(Inkscape) Color is used for the text in Inkscape, but the package 'color.sty' is not loaded}%
    \renewcommand\color[2][]{}%
  }%
  \providecommand\transparent[1]{%
    \errmessage{(Inkscape) Transparency is used (non-zero) for the text in Inkscape, but the package 'transparent.sty' is not loaded}%
    \renewcommand\transparent[1]{}%
  }%
  \providecommand\rotatebox[2]{#2}%
  \ifx\svgwidth\undefined%
    \setlength{\unitlength}{49.54732089bp}%
    \ifx\svgscale\undefined%
      \relax%
    \else%
      \setlength{\unitlength}{\unitlength * \real{\svgscale}}%
    \fi%
  \else%
    \setlength{\unitlength}{\svgwidth}%
  \fi%
  \global\let\svgwidth\undefined%
  \global\let\svgscale\undefined%
  \makeatother%
  \begin{picture}(1,1.00059637)%
    \put(0,0){\includegraphics[width=\unitlength,page=1]{PushOff.pdf}}%
  \end{picture}%
\endgroup%
\,\,} \underbrace{U_1(-1){\def\svgwidth{1,6ex}\,\,%% Creator: Inkscape 0.91_64bit, www.inkscape.org
%% PDF/EPS/PS + LaTeX output extension by Johan Engelen, 2010
%% Accompanies image file 'PushOff.pdf' (pdf, eps, ps)
%%
%% To include the image in your LaTeX document, write
%%   \input{<filename>.pdf_tex}
%%  instead of
%%   \includegraphics{<filename>.pdf}
%% To scale the image, write
%%   \def\svgwidth{<desired width>}
%%   \input{<filename>.pdf_tex}
%%  instead of
%%   \includegraphics[width=<desired width>]{<filename>.pdf}
%%
%% Images with a different path to the parent latex file can
%% be accessed with the `import' package (which may need to be
%% installed) using
%%   \usepackage{import}
%% in the preamble, and then including the image with
%%   \import{<path to file>}{<filename>.pdf_tex}
%% Alternatively, one can specify
%%   \graphicspath{{<path to file>/}}
%% 
%% For more information, please see info/svg-inkscape on CTAN:
%%   http://tug.ctan.org/tex-archive/info/svg-inkscape
%%
\begingroup%
  \makeatletter%
  \providecommand\color[2][]{%
    \errmessage{(Inkscape) Color is used for the text in Inkscape, but the package 'color.sty' is not loaded}%
    \renewcommand\color[2][]{}%
  }%
  \providecommand\transparent[1]{%
    \errmessage{(Inkscape) Transparency is used (non-zero) for the text in Inkscape, but the package 'transparent.sty' is not loaded}%
    \renewcommand\transparent[1]{}%
  }%
  \providecommand\rotatebox[2]{#2}%
  \ifx\svgwidth\undefined%
    \setlength{\unitlength}{49.54732089bp}%
    \ifx\svgscale\undefined%
      \relax%
    \else%
      \setlength{\unitlength}{\unitlength * \real{\svgscale}}%
    \fi%
  \else%
    \setlength{\unitlength}{\svgwidth}%
  \fi%
  \global\let\svgwidth\undefined%
  \global\let\svgscale\undefined%
  \makeatother%
  \begin{picture}(1,1.00059637)%
    \put(0,0){\includegraphics[width=\unitlength,page=1]{PushOff.pdf}}%
  \end{picture}%
\endgroup%
\,\,} \cdots {\def\svgwidth{1,6ex}\,\,%% Creator: Inkscape 0.91_64bit, www.inkscape.org
%% PDF/EPS/PS + LaTeX output extension by Johan Engelen, 2010
%% Accompanies image file 'PushOff.pdf' (pdf, eps, ps)
%%
%% To include the image in your LaTeX document, write
%%   \input{<filename>.pdf_tex}
%%  instead of
%%   \includegraphics{<filename>.pdf}
%% To scale the image, write
%%   \def\svgwidth{<desired width>}
%%   \input{<filename>.pdf_tex}
%%  instead of
%%   \includegraphics[width=<desired width>]{<filename>.pdf}
%%
%% Images with a different path to the parent latex file can
%% be accessed with the `import' package (which may need to be
%% installed) using
%%   \usepackage{import}
%% in the preamble, and then including the image with
%%   \import{<path to file>}{<filename>.pdf_tex}
%% Alternatively, one can specify
%%   \graphicspath{{<path to file>/}}
%% 
%% For more information, please see info/svg-inkscape on CTAN:
%%   http://tug.ctan.org/tex-archive/info/svg-inkscape
%%
\begingroup%
  \makeatletter%
  \providecommand\color[2][]{%
    \errmessage{(Inkscape) Color is used for the text in Inkscape, but the package 'color.sty' is not loaded}%
    \renewcommand\color[2][]{}%
  }%
  \providecommand\transparent[1]{%
    \errmessage{(Inkscape) Transparency is used (non-zero) for the text in Inkscape, but the package 'transparent.sty' is not loaded}%
    \renewcommand\transparent[1]{}%
  }%
  \providecommand\rotatebox[2]{#2}%
  \ifx\svgwidth\undefined%
    \setlength{\unitlength}{49.54732089bp}%
    \ifx\svgscale\undefined%
      \relax%
    \else%
      \setlength{\unitlength}{\unitlength * \real{\svgscale}}%
    \fi%
  \else%
    \setlength{\unitlength}{\svgwidth}%
  \fi%
  \global\let\svgwidth\undefined%
  \global\let\svgscale\undefined%
  \makeatother%
  \begin{picture}(1,1.00059637)%
    \put(0,0){\includegraphics[width=\unitlength,page=1]{PushOff.pdf}}%
  \end{picture}%
\endgroup%
\,\,} U_1(-1)}_\text{$|n|-3$}{\def\svgwidth{1,6ex}\,\,%% Creator: Inkscape 0.91_64bit, www.inkscape.org
%% PDF/EPS/PS + LaTeX output extension by Johan Engelen, 2010
%% Accompanies image file 'PushOff.pdf' (pdf, eps, ps)
%%
%% To include the image in your LaTeX document, write
%%   \input{<filename>.pdf_tex}
%%  instead of
%%   \includegraphics{<filename>.pdf}
%% To scale the image, write
%%   \def\svgwidth{<desired width>}
%%   \input{<filename>.pdf_tex}
%%  instead of
%%   \includegraphics[width=<desired width>]{<filename>.pdf}
%%
%% Images with a different path to the parent latex file can
%% be accessed with the `import' package (which may need to be
%% installed) using
%%   \usepackage{import}
%% in the preamble, and then including the image with
%%   \import{<path to file>}{<filename>.pdf_tex}
%% Alternatively, one can specify
%%   \graphicspath{{<path to file>/}}
%% 
%% For more information, please see info/svg-inkscape on CTAN:
%%   http://tug.ctan.org/tex-archive/info/svg-inkscape
%%
\begingroup%
  \makeatletter%
  \providecommand\color[2][]{%
    \errmessage{(Inkscape) Color is used for the text in Inkscape, but the package 'color.sty' is not loaded}%
    \renewcommand\color[2][]{}%
  }%
  \providecommand\transparent[1]{%
    \errmessage{(Inkscape) Transparency is used (non-zero) for the text in Inkscape, but the package 'transparent.sty' is not loaded}%
    \renewcommand\transparent[1]{}%
  }%
  \providecommand\rotatebox[2]{#2}%
  \ifx\svgwidth\undefined%
    \setlength{\unitlength}{49.54732089bp}%
    \ifx\svgscale\undefined%
      \relax%
    \else%
      \setlength{\unitlength}{\unitlength * \real{\svgscale}}%
    \fi%
  \else%
    \setlength{\unitlength}{\svgwidth}%
  \fi%
  \global\let\svgwidth\undefined%
  \global\let\svgscale\undefined%
  \makeatother%
  \begin{picture}(1,1.00059637)%
    \put(0,0){\includegraphics[width=\unitlength,page=1]{PushOff.pdf}}%
  \end{picture}%
\endgroup%
\,\,} U_{1,1}(-1),$$
see Figure~\ref{fig:c5}.
From the linking matrix, we deduce that $\mu_{-n}$ is a generator of the first homology and that
$$\mu_1=\mu_2=-\mu_{-n},\: \mu_3=\mu_4,\ldots,\mu_{-n-1}=2\mu_{-n}=0.$$
The characteristic sublink depends on the parity of $n$. Nevertheless, in both cases the $\Gamma$ invariant computes as
\begin{align*}
    \Gamma=\frac{r_{-n}}{2}\mu_{-n}=\begin{cases}
        0 \text{ if } r_{-n}=0, \\
1 \text{ if } r_{-n}=\pm 2.
    \end{cases}
\end{align*}
\begin{figure}[htbp]
    \centering
    \includegraphics[width=1\linewidth]{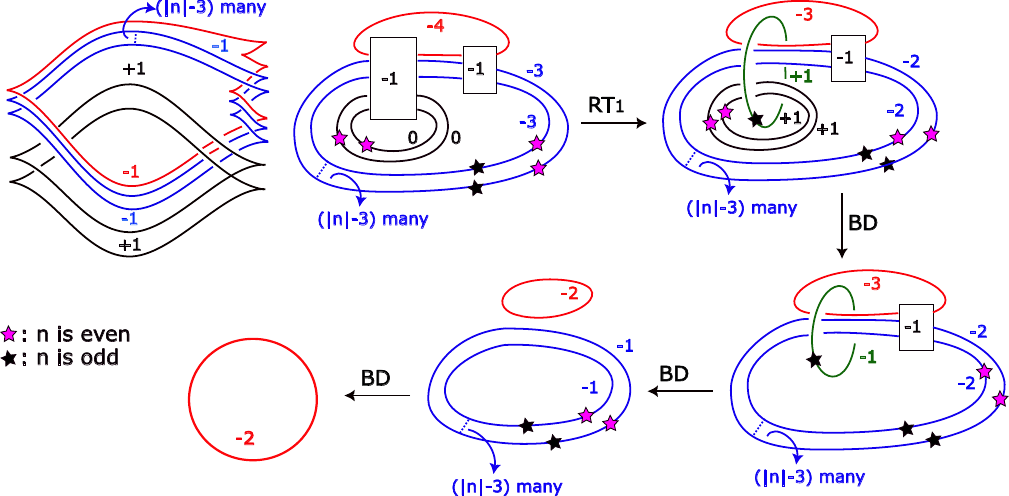}
    \caption{Case (4)}
    \label{fig:c5}
\end{figure}

 \noindent \textbf{Case (5):}  The generalized linking matrix is
 \begin{align*}
    Q=\begin{pmatrix}
        1+t&-t^2-t&2t\\
           t &-t^2&2t-2\\
            t&-t^2+1&2t-2n-3
    \end{pmatrix}
    \end{align*}
with signature $\sigma(Q)=-1$. The possible values of the rotation vectors are 
$$\mathbf{r}=(r,r\pm1,r+1+ x)^T$$
where
\begin{equation*}
  x =
    \begin{cases}
      2i, & \text{for } i=-\frac{n}{2},-\frac{n}{2}+1\ldots, \frac{n}{2}  \text{ if  $n$ is even,}\\
       2i+1, & \text{for } i=-\left\lfloor \frac{n}{2}\right\rfloor,-\left\lfloor \frac{n}{2}\right\rfloor+1,\ldots, \left\lfloor \frac{n}{2}\right\rfloor \text{ if $n$ is odd.}
    \end{cases}
\end{equation*}
Then we compute
\begin{align*}
    d_3= \frac{1}{8}\left\{(2n+1)(t\pm r)^2+(4n-2)(t\pm r)\mp 4x(t\pm r+1)+2n-1\right\}.
\end{align*}
Since the sum of rotation number and Thurston--Bennequin invariant of a Legendrian knot is always an odd integer~\cite{Geiges_book}, we can write $t\pm r=2k+1$ for some $k\leq -1$. Using this substitution, we get the claimed values for $\de_3$.

Using the transformation lemma we rewrite this contact surgery diagram as 
$$U(+1){\def\svgwidth{1,6ex}\,\,%% Creator: Inkscape 0.91_64bit, www.inkscape.org
%% PDF/EPS/PS + LaTeX output extension by Johan Engelen, 2010
%% Accompanies image file 'PushOff.pdf' (pdf, eps, ps)
%%
%% To include the image in your LaTeX document, write
%%   \input{<filename>.pdf_tex}
%%  instead of
%%   \includegraphics{<filename>.pdf}
%% To scale the image, write
%%   \def\svgwidth{<desired width>}
%%   \input{<filename>.pdf_tex}
%%  instead of
%%   \includegraphics[width=<desired width>]{<filename>.pdf}
%%
%% Images with a different path to the parent latex file can
%% be accessed with the `import' package (which may need to be
%% installed) using
%%   \usepackage{import}
%% in the preamble, and then including the image with
%%   \import{<path to file>}{<filename>.pdf_tex}
%% Alternatively, one can specify
%%   \graphicspath{{<path to file>/}}
%% 
%% For more information, please see info/svg-inkscape on CTAN:
%%   http://tug.ctan.org/tex-archive/info/svg-inkscape
%%
\begingroup%
  \makeatletter%
  \providecommand\color[2][]{%
    \errmessage{(Inkscape) Color is used for the text in Inkscape, but the package 'color.sty' is not loaded}%
    \renewcommand\color[2][]{}%
  }%
  \providecommand\transparent[1]{%
    \errmessage{(Inkscape) Transparency is used (non-zero) for the text in Inkscape, but the package 'transparent.sty' is not loaded}%
    \renewcommand\transparent[1]{}%
  }%
  \providecommand\rotatebox[2]{#2}%
  \ifx\svgwidth\undefined%
    \setlength{\unitlength}{49.54732089bp}%
    \ifx\svgscale\undefined%
      \relax%
    \else%
      \setlength{\unitlength}{\unitlength * \real{\svgscale}}%
    \fi%
  \else%
    \setlength{\unitlength}{\svgwidth}%
  \fi%
  \global\let\svgwidth\undefined%
  \global\let\svgscale\undefined%
  \makeatother%
  \begin{picture}(1,1.00059637)%
    \put(0,0){\includegraphics[width=\unitlength,page=1]{PushOff.pdf}}%
  \end{picture}%
\endgroup%
\,\,} \underbrace{U_1(-1){\def\svgwidth{1,6ex}\,\,%% Creator: Inkscape 0.91_64bit, www.inkscape.org
%% PDF/EPS/PS + LaTeX output extension by Johan Engelen, 2010
%% Accompanies image file 'PushOff.pdf' (pdf, eps, ps)
%%
%% To include the image in your LaTeX document, write
%%   \input{<filename>.pdf_tex}
%%  instead of
%%   \includegraphics{<filename>.pdf}
%% To scale the image, write
%%   \def\svgwidth{<desired width>}
%%   \input{<filename>.pdf_tex}
%%  instead of
%%   \includegraphics[width=<desired width>]{<filename>.pdf}
%%
%% Images with a different path to the parent latex file can
%% be accessed with the `import' package (which may need to be
%% installed) using
%%   \usepackage{import}
%% in the preamble, and then including the image with
%%   \import{<path to file>}{<filename>.pdf_tex}
%% Alternatively, one can specify
%%   \graphicspath{{<path to file>/}}
%% 
%% For more information, please see info/svg-inkscape on CTAN:
%%   http://tug.ctan.org/tex-archive/info/svg-inkscape
%%
\begingroup%
  \makeatletter%
  \providecommand\color[2][]{%
    \errmessage{(Inkscape) Color is used for the text in Inkscape, but the package 'color.sty' is not loaded}%
    \renewcommand\color[2][]{}%
  }%
  \providecommand\transparent[1]{%
    \errmessage{(Inkscape) Transparency is used (non-zero) for the text in Inkscape, but the package 'transparent.sty' is not loaded}%
    \renewcommand\transparent[1]{}%
  }%
  \providecommand\rotatebox[2]{#2}%
  \ifx\svgwidth\undefined%
    \setlength{\unitlength}{49.54732089bp}%
    \ifx\svgscale\undefined%
      \relax%
    \else%
      \setlength{\unitlength}{\unitlength * \real{\svgscale}}%
    \fi%
  \else%
    \setlength{\unitlength}{\svgwidth}%
  \fi%
  \global\let\svgwidth\undefined%
  \global\let\svgscale\undefined%
  \makeatother%
  \begin{picture}(1,1.00059637)%
    \put(0,0){\includegraphics[width=\unitlength,page=1]{PushOff.pdf}}%
  \end{picture}%
\endgroup%
\,\,} \cdots {\def\svgwidth{1,6ex}\,\,%% Creator: Inkscape 0.91_64bit, www.inkscape.org
%% PDF/EPS/PS + LaTeX output extension by Johan Engelen, 2010
%% Accompanies image file 'PushOff.pdf' (pdf, eps, ps)
%%
%% To include the image in your LaTeX document, write
%%   \input{<filename>.pdf_tex}
%%  instead of
%%   \includegraphics{<filename>.pdf}
%% To scale the image, write
%%   \def\svgwidth{<desired width>}
%%   \input{<filename>.pdf_tex}
%%  instead of
%%   \includegraphics[width=<desired width>]{<filename>.pdf}
%%
%% Images with a different path to the parent latex file can
%% be accessed with the `import' package (which may need to be
%% installed) using
%%   \usepackage{import}
%% in the preamble, and then including the image with
%%   \import{<path to file>}{<filename>.pdf_tex}
%% Alternatively, one can specify
%%   \graphicspath{{<path to file>/}}
%% 
%% For more information, please see info/svg-inkscape on CTAN:
%%   http://tug.ctan.org/tex-archive/info/svg-inkscape
%%
\begingroup%
  \makeatletter%
  \providecommand\color[2][]{%
    \errmessage{(Inkscape) Color is used for the text in Inkscape, but the package 'color.sty' is not loaded}%
    \renewcommand\color[2][]{}%
  }%
  \providecommand\transparent[1]{%
    \errmessage{(Inkscape) Transparency is used (non-zero) for the text in Inkscape, but the package 'transparent.sty' is not loaded}%
    \renewcommand\transparent[1]{}%
  }%
  \providecommand\rotatebox[2]{#2}%
  \ifx\svgwidth\undefined%
    \setlength{\unitlength}{49.54732089bp}%
    \ifx\svgscale\undefined%
      \relax%
    \else%
      \setlength{\unitlength}{\unitlength * \real{\svgscale}}%
    \fi%
  \else%
    \setlength{\unitlength}{\svgwidth}%
  \fi%
  \global\let\svgwidth\undefined%
  \global\let\svgscale\undefined%
  \makeatother%
  \begin{picture}(1,1.00059637)%
    \put(0,0){\includegraphics[width=\unitlength,page=1]{PushOff.pdf}}%
  \end{picture}%
\endgroup%
\,\,} U_1(-1)}_\text{$-t-1$}{\def\svgwidth{1,6ex}\,\,%% Creator: Inkscape 0.91_64bit, www.inkscape.org
%% PDF/EPS/PS + LaTeX output extension by Johan Engelen, 2010
%% Accompanies image file 'PushOff.pdf' (pdf, eps, ps)
%%
%% To include the image in your LaTeX document, write
%%   \input{<filename>.pdf_tex}
%%  instead of
%%   \includegraphics{<filename>.pdf}
%% To scale the image, write
%%   \def\svgwidth{<desired width>}
%%   \input{<filename>.pdf_tex}
%%  instead of
%%   \includegraphics[width=<desired width>]{<filename>.pdf}
%%
%% Images with a different path to the parent latex file can
%% be accessed with the `import' package (which may need to be
%% installed) using
%%   \usepackage{import}
%% in the preamble, and then including the image with
%%   \import{<path to file>}{<filename>.pdf_tex}
%% Alternatively, one can specify
%%   \graphicspath{{<path to file>/}}
%% 
%% For more information, please see info/svg-inkscape on CTAN:
%%   http://tug.ctan.org/tex-archive/info/svg-inkscape
%%
\begingroup%
  \makeatletter%
  \providecommand\color[2][]{%
    \errmessage{(Inkscape) Color is used for the text in Inkscape, but the package 'color.sty' is not loaded}%
    \renewcommand\color[2][]{}%
  }%
  \providecommand\transparent[1]{%
    \errmessage{(Inkscape) Transparency is used (non-zero) for the text in Inkscape, but the package 'transparent.sty' is not loaded}%
    \renewcommand\transparent[1]{}%
  }%
  \providecommand\rotatebox[2]{#2}%
  \ifx\svgwidth\undefined%
    \setlength{\unitlength}{49.54732089bp}%
    \ifx\svgscale\undefined%
      \relax%
    \else%
      \setlength{\unitlength}{\unitlength * \real{\svgscale}}%
    \fi%
  \else%
    \setlength{\unitlength}{\svgwidth}%
  \fi%
  \global\let\svgwidth\undefined%
  \global\let\svgscale\undefined%
  \makeatother%
  \begin{picture}(1,1.00059637)%
    \put(0,0){\includegraphics[width=\unitlength,page=1]{PushOff.pdf}}%
  \end{picture}%
\endgroup%
\,\,} U_{1,n}(-1){\def\svgwidth{1,6ex}\,\,%% Creator: Inkscape 0.91_64bit, www.inkscape.org
%% PDF/EPS/PS + LaTeX output extension by Johan Engelen, 2010
%% Accompanies image file 'PushOff.pdf' (pdf, eps, ps)
%%
%% To include the image in your LaTeX document, write
%%   \input{<filename>.pdf_tex}
%%  instead of
%%   \includegraphics{<filename>.pdf}
%% To scale the image, write
%%   \def\svgwidth{<desired width>}
%%   \input{<filename>.pdf_tex}
%%  instead of
%%   \includegraphics[width=<desired width>]{<filename>.pdf}
%%
%% Images with a different path to the parent latex file can
%% be accessed with the `import' package (which may need to be
%% installed) using
%%   \usepackage{import}
%% in the preamble, and then including the image with
%%   \import{<path to file>}{<filename>.pdf_tex}
%% Alternatively, one can specify
%%   \graphicspath{{<path to file>/}}
%% 
%% For more information, please see info/svg-inkscape on CTAN:
%%   http://tug.ctan.org/tex-archive/info/svg-inkscape
%%
\begingroup%
  \makeatletter%
  \providecommand\color[2][]{%
    \errmessage{(Inkscape) Color is used for the text in Inkscape, but the package 'color.sty' is not loaded}%
    \renewcommand\color[2][]{}%
  }%
  \providecommand\transparent[1]{%
    \errmessage{(Inkscape) Transparency is used (non-zero) for the text in Inkscape, but the package 'transparent.sty' is not loaded}%
    \renewcommand\transparent[1]{}%
  }%
  \providecommand\rotatebox[2]{#2}%
  \ifx\svgwidth\undefined%
    \setlength{\unitlength}{49.54732089bp}%
    \ifx\svgscale\undefined%
      \relax%
    \else%
      \setlength{\unitlength}{\unitlength * \real{\svgscale}}%
    \fi%
  \else%
    \setlength{\unitlength}{\svgwidth}%
  \fi%
  \global\let\svgwidth\undefined%
  \global\let\svgscale\undefined%
  \makeatother%
  \begin{picture}(1,1.00059637)%
    \put(0,0){\includegraphics[width=\unitlength,page=1]{PushOff.pdf}}%
  \end{picture}%
\endgroup%
\,\,} U_{1,n}(-1),$$ 
see Figure~\ref{fig:c6}.
In this case, the linking matrix shows that $\mu_{-t+1}$ is a generator and 
$$ \mu_2=\ldots=\mu_{-t+1},\mu_{-t+1}=-\mu_{-t+2}, \mu_1=(t-2)\mu_{-t+1}.$$
The characteristic sublink, depending on the parities of $n$ and $t$, is shown in Figure~\ref{fig:c6}. Thus we obtain from Lemma~\ref{lem:Gamma} that
\begin{align*}
    \Gamma= \begin{cases}
        \frac{r+t+1}{2}\mu_{-t+1} & \textrm{ if $t$ is odd,}\\
        \frac{r_2+t}{2}\mu_{-t+1} & \textrm{ if $t$ is even.}
    \end{cases}
\end{align*}
By setting $r_2=r\pm1$, and substituting $t\pm r=2k+1$, we get $\Gamma=k+1\, \pmod{2}$. Now, to distinguish between parities of $k$, we write $k=2m+1$ if $k$ is odd and thus $\Gamma=0$ and $k=2m$ if $k$ is even and thus $\Gamma=1$. Doing the same in the formula for the $\de_3$-invariants we obtain the claimed pairs of the invariants. 

Note that in the case that $k$ is odd, $m=-1$ is possible. However, in Table~\ref{tab:tab2} we have only listed the values for $m<-1$. This is justified, because whenever in that case $m=-1$, it follows that, independent of the other parameters, we get $\de_3=1/4$. Indeed, we show below in Lemma~\ref{lem:tight} that for $m=-1$ and $k$ odd, we always get the tight contact structure $\xist$ which we already obtained in Case $(0)$. Thus we do not list the case $m=-1$ again in Case $(5)$.
\begin{figure}[htbp]
    \centering
    \includegraphics[width=0.99\linewidth]{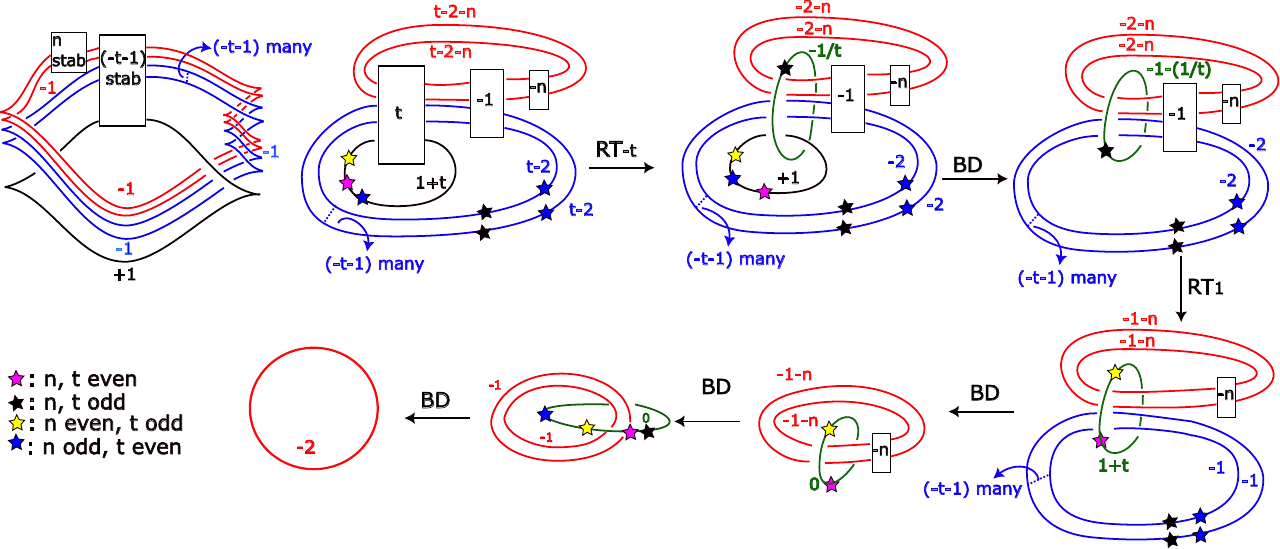}
    \caption{Case (5)}
    \label{fig:c6}
\end{figure}

 \noindent \textbf{Case (6):}  The generalized linking matrix is
 \begin{align*}
    Q=\begin{pmatrix}
        -1&-2\\
        -2 &-6
    \end{pmatrix}
    \end{align*}
with signature $\sigma(Q)=-2$. The rotation numbers of $U$ and $U_3$ can take values $r=\pm1$ and $r+x$, where $x=\pm 1, \pm 3$, from which we compute the $\de_3$-invariants.

Using the transformation lemma we write this contact surgery diagram as 
$$U(+1){\def\svgwidth{1,6ex}\,\,%% Creator: Inkscape 0.91_64bit, www.inkscape.org
%% PDF/EPS/PS + LaTeX output extension by Johan Engelen, 2010
%% Accompanies image file 'PushOff.pdf' (pdf, eps, ps)
%%
%% To include the image in your LaTeX document, write
%%   \input{<filename>.pdf_tex}
%%  instead of
%%   \includegraphics{<filename>.pdf}
%% To scale the image, write
%%   \def\svgwidth{<desired width>}
%%   \input{<filename>.pdf_tex}
%%  instead of
%%   \includegraphics[width=<desired width>]{<filename>.pdf}
%%
%% Images with a different path to the parent latex file can
%% be accessed with the `import' package (which may need to be
%% installed) using
%%   \usepackage{import}
%% in the preamble, and then including the image with
%%   \import{<path to file>}{<filename>.pdf_tex}
%% Alternatively, one can specify
%%   \graphicspath{{<path to file>/}}
%% 
%% For more information, please see info/svg-inkscape on CTAN:
%%   http://tug.ctan.org/tex-archive/info/svg-inkscape
%%
\begingroup%
  \makeatletter%
  \providecommand\color[2][]{%
    \errmessage{(Inkscape) Color is used for the text in Inkscape, but the package 'color.sty' is not loaded}%
    \renewcommand\color[2][]{}%
  }%
  \providecommand\transparent[1]{%
    \errmessage{(Inkscape) Transparency is used (non-zero) for the text in Inkscape, but the package 'transparent.sty' is not loaded}%
    \renewcommand\transparent[1]{}%
  }%
  \providecommand\rotatebox[2]{#2}%
  \ifx\svgwidth\undefined%
    \setlength{\unitlength}{49.54732089bp}%
    \ifx\svgscale\undefined%
      \relax%
    \else%
      \setlength{\unitlength}{\unitlength * \real{\svgscale}}%
    \fi%
  \else%
    \setlength{\unitlength}{\svgwidth}%
  \fi%
  \global\let\svgwidth\undefined%
  \global\let\svgscale\undefined%
  \makeatother%
  \begin{picture}(1,1.00059637)%
    \put(0,0){\includegraphics[width=\unitlength,page=1]{PushOff.pdf}}%
  \end{picture}%
\endgroup%
\,\,} U_3(-1),$$ 
see Figure~\ref{fig:c7}.
Here we compute similarly to the previous cases that 
\begin{align*}
    \Gamma=\begin{cases}
        0 \text{ if } r_2=\pm 2, \\
        1\text{ if } r_2=0,\pm 4.
    \end{cases} 
\end{align*}
\begin{figure}[htbp]
    \centering
    \includegraphics[width=0.7\linewidth]{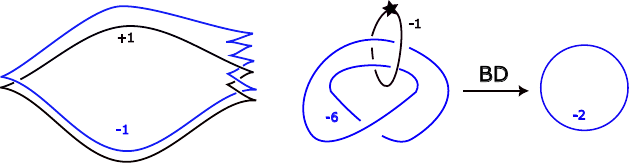}
    \caption{Case (6)}
    \label{fig:c7}
\end{figure}
    
 \noindent \textbf{Case (7):}  The generalized linking matrix is
 \begin{align*}
    Q=\begin{pmatrix}
        1+t&-t^2-2t&t\\
        t &-t^2-t+1& t-1\\
        t& -t^2-t+2&t-4
    \end{pmatrix}
    \end{align*}
with signature $\sigma(Q)=-3$. The possible values of the rotation vectors are 
$$\mathbf{r}=(r,r\pm1,r+1+ x)^T$$
where $x$ can take values $-2$, $0$, or $2$. And similar to Case $(5)$, we can compute the values of the $\de_3$-invariants.

Using the transformation lemma we write this contact surgery diagram as 
$$U(+1){\def\svgwidth{1,6ex}\,\,%% Creator: Inkscape 0.91_64bit, www.inkscape.org
%% PDF/EPS/PS + LaTeX output extension by Johan Engelen, 2010
%% Accompanies image file 'PushOff.pdf' (pdf, eps, ps)
%%
%% To include the image in your LaTeX document, write
%%   \input{<filename>.pdf_tex}
%%  instead of
%%   \includegraphics{<filename>.pdf}
%% To scale the image, write
%%   \def\svgwidth{<desired width>}
%%   \input{<filename>.pdf_tex}
%%  instead of
%%   \includegraphics[width=<desired width>]{<filename>.pdf}
%%
%% Images with a different path to the parent latex file can
%% be accessed with the `import' package (which may need to be
%% installed) using
%%   \usepackage{import}
%% in the preamble, and then including the image with
%%   \import{<path to file>}{<filename>.pdf_tex}
%% Alternatively, one can specify
%%   \graphicspath{{<path to file>/}}
%% 
%% For more information, please see info/svg-inkscape on CTAN:
%%   http://tug.ctan.org/tex-archive/info/svg-inkscape
%%
\begingroup%
  \makeatletter%
  \providecommand\color[2][]{%
    \errmessage{(Inkscape) Color is used for the text in Inkscape, but the package 'color.sty' is not loaded}%
    \renewcommand\color[2][]{}%
  }%
  \providecommand\transparent[1]{%
    \errmessage{(Inkscape) Transparency is used (non-zero) for the text in Inkscape, but the package 'transparent.sty' is not loaded}%
    \renewcommand\transparent[1]{}%
  }%
  \providecommand\rotatebox[2]{#2}%
  \ifx\svgwidth\undefined%
    \setlength{\unitlength}{49.54732089bp}%
    \ifx\svgscale\undefined%
      \relax%
    \else%
      \setlength{\unitlength}{\unitlength * \real{\svgscale}}%
    \fi%
  \else%
    \setlength{\unitlength}{\svgwidth}%
  \fi%
  \global\let\svgwidth\undefined%
  \global\let\svgscale\undefined%
  \makeatother%
  \begin{picture}(1,1.00059637)%
    \put(0,0){\includegraphics[width=\unitlength,page=1]{PushOff.pdf}}%
  \end{picture}%
\endgroup%
\,\,} \underbrace{U_1(-1){\def\svgwidth{1,6ex}\,\,%% Creator: Inkscape 0.91_64bit, www.inkscape.org
%% PDF/EPS/PS + LaTeX output extension by Johan Engelen, 2010
%% Accompanies image file 'PushOff.pdf' (pdf, eps, ps)
%%
%% To include the image in your LaTeX document, write
%%   \input{<filename>.pdf_tex}
%%  instead of
%%   \includegraphics{<filename>.pdf}
%% To scale the image, write
%%   \def\svgwidth{<desired width>}
%%   \input{<filename>.pdf_tex}
%%  instead of
%%   \includegraphics[width=<desired width>]{<filename>.pdf}
%%
%% Images with a different path to the parent latex file can
%% be accessed with the `import' package (which may need to be
%% installed) using
%%   \usepackage{import}
%% in the preamble, and then including the image with
%%   \import{<path to file>}{<filename>.pdf_tex}
%% Alternatively, one can specify
%%   \graphicspath{{<path to file>/}}
%% 
%% For more information, please see info/svg-inkscape on CTAN:
%%   http://tug.ctan.org/tex-archive/info/svg-inkscape
%%
\begingroup%
  \makeatletter%
  \providecommand\color[2][]{%
    \errmessage{(Inkscape) Color is used for the text in Inkscape, but the package 'color.sty' is not loaded}%
    \renewcommand\color[2][]{}%
  }%
  \providecommand\transparent[1]{%
    \errmessage{(Inkscape) Transparency is used (non-zero) for the text in Inkscape, but the package 'transparent.sty' is not loaded}%
    \renewcommand\transparent[1]{}%
  }%
  \providecommand\rotatebox[2]{#2}%
  \ifx\svgwidth\undefined%
    \setlength{\unitlength}{49.54732089bp}%
    \ifx\svgscale\undefined%
      \relax%
    \else%
      \setlength{\unitlength}{\unitlength * \real{\svgscale}}%
    \fi%
  \else%
    \setlength{\unitlength}{\svgwidth}%
  \fi%
  \global\let\svgwidth\undefined%
  \global\let\svgscale\undefined%
  \makeatother%
  \begin{picture}(1,1.00059637)%
    \put(0,0){\includegraphics[width=\unitlength,page=1]{PushOff.pdf}}%
  \end{picture}%
\endgroup%
\,\,} \cdots {\def\svgwidth{1,6ex}\,\,%% Creator: Inkscape 0.91_64bit, www.inkscape.org
%% PDF/EPS/PS + LaTeX output extension by Johan Engelen, 2010
%% Accompanies image file 'PushOff.pdf' (pdf, eps, ps)
%%
%% To include the image in your LaTeX document, write
%%   \input{<filename>.pdf_tex}
%%  instead of
%%   \includegraphics{<filename>.pdf}
%% To scale the image, write
%%   \def\svgwidth{<desired width>}
%%   \input{<filename>.pdf_tex}
%%  instead of
%%   \includegraphics[width=<desired width>]{<filename>.pdf}
%%
%% Images with a different path to the parent latex file can
%% be accessed with the `import' package (which may need to be
%% installed) using
%%   \usepackage{import}
%% in the preamble, and then including the image with
%%   \import{<path to file>}{<filename>.pdf_tex}
%% Alternatively, one can specify
%%   \graphicspath{{<path to file>/}}
%% 
%% For more information, please see info/svg-inkscape on CTAN:
%%   http://tug.ctan.org/tex-archive/info/svg-inkscape
%%
\begingroup%
  \makeatletter%
  \providecommand\color[2][]{%
    \errmessage{(Inkscape) Color is used for the text in Inkscape, but the package 'color.sty' is not loaded}%
    \renewcommand\color[2][]{}%
  }%
  \providecommand\transparent[1]{%
    \errmessage{(Inkscape) Transparency is used (non-zero) for the text in Inkscape, but the package 'transparent.sty' is not loaded}%
    \renewcommand\transparent[1]{}%
  }%
  \providecommand\rotatebox[2]{#2}%
  \ifx\svgwidth\undefined%
    \setlength{\unitlength}{49.54732089bp}%
    \ifx\svgscale\undefined%
      \relax%
    \else%
      \setlength{\unitlength}{\unitlength * \real{\svgscale}}%
    \fi%
  \else%
    \setlength{\unitlength}{\svgwidth}%
  \fi%
  \global\let\svgwidth\undefined%
  \global\let\svgscale\undefined%
  \makeatother%
  \begin{picture}(1,1.00059637)%
    \put(0,0){\includegraphics[width=\unitlength,page=1]{PushOff.pdf}}%
  \end{picture}%
\endgroup%
\,\,} U_1(-1)}_\text{$-t-2$}{\def\svgwidth{1,6ex}\,\,%% Creator: Inkscape 0.91_64bit, www.inkscape.org
%% PDF/EPS/PS + LaTeX output extension by Johan Engelen, 2010
%% Accompanies image file 'PushOff.pdf' (pdf, eps, ps)
%%
%% To include the image in your LaTeX document, write
%%   \input{<filename>.pdf_tex}
%%  instead of
%%   \includegraphics{<filename>.pdf}
%% To scale the image, write
%%   \def\svgwidth{<desired width>}
%%   \input{<filename>.pdf_tex}
%%  instead of
%%   \includegraphics[width=<desired width>]{<filename>.pdf}
%%
%% Images with a different path to the parent latex file can
%% be accessed with the `import' package (which may need to be
%% installed) using
%%   \usepackage{import}
%% in the preamble, and then including the image with
%%   \import{<path to file>}{<filename>.pdf_tex}
%% Alternatively, one can specify
%%   \graphicspath{{<path to file>/}}
%% 
%% For more information, please see info/svg-inkscape on CTAN:
%%   http://tug.ctan.org/tex-archive/info/svg-inkscape
%%
\begingroup%
  \makeatletter%
  \providecommand\color[2][]{%
    \errmessage{(Inkscape) Color is used for the text in Inkscape, but the package 'color.sty' is not loaded}%
    \renewcommand\color[2][]{}%
  }%
  \providecommand\transparent[1]{%
    \errmessage{(Inkscape) Transparency is used (non-zero) for the text in Inkscape, but the package 'transparent.sty' is not loaded}%
    \renewcommand\transparent[1]{}%
  }%
  \providecommand\rotatebox[2]{#2}%
  \ifx\svgwidth\undefined%
    \setlength{\unitlength}{49.54732089bp}%
    \ifx\svgscale\undefined%
      \relax%
    \else%
      \setlength{\unitlength}{\unitlength * \real{\svgscale}}%
    \fi%
  \else%
    \setlength{\unitlength}{\svgwidth}%
  \fi%
  \global\let\svgwidth\undefined%
  \global\let\svgscale\undefined%
  \makeatother%
  \begin{picture}(1,1.00059637)%
    \put(0,0){\includegraphics[width=\unitlength,page=1]{PushOff.pdf}}%
  \end{picture}%
\endgroup%
\,\,} U_{1,2}(-1),$$ 
see Figure~\ref{fig:c8}.
In this case, $\mu_{-t}$ is a generator and we have the relations
$$2\mu_{-t}=0, \mu_1=3t\mu_{-t},\mu_2=\ldots=\mu_{-t}$$
from which we compute
\begin{align*}
    \Gamma
    =\frac{(r+1+t)}{2}t\mu_{-t}-\frac{(r_2+t)}{2}t\mu_{-t}+\frac{(t+r_{-t})}{2}\mu_{-t}.
\end{align*}
If we set $2k+1=t\pm r$, $r_2=r\pm1$, and $r_{-t}=r\pm 1+x$, for $x=0,\pm 2$, this yields
\begin{align*}
    \Gamma=k+1+\frac{x}{2}\, \pmod{2}.
\end{align*}
By considering the different parities of $k$ as in Case $(5)$ we obtain the claimed pairs of invariants.
\begin{figure}[htbp]
    \centering
    \includegraphics[width=0.9\linewidth]{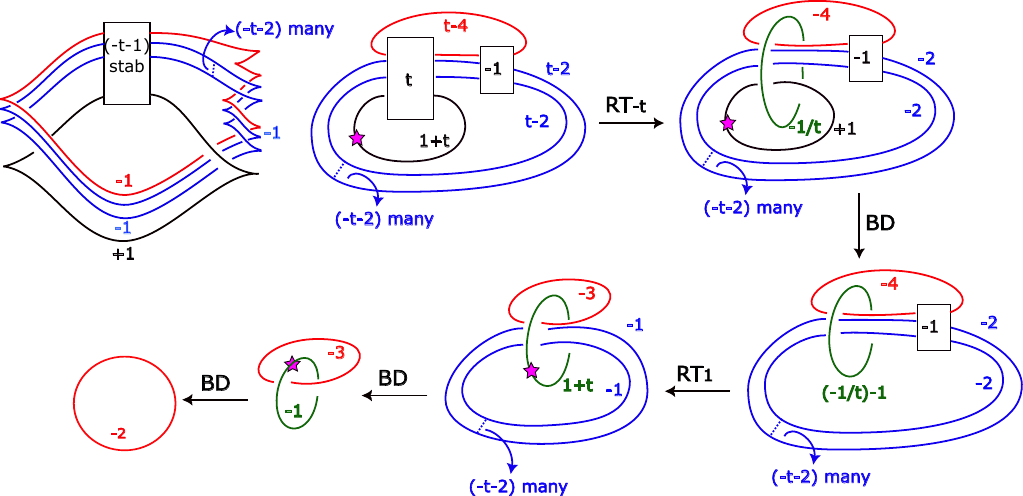}
    \caption{Case (7)}
    \label{fig:c8}
\end{figure}

 \noindent \textbf{Case (8):}  The generalized linking matrix is
 \begin{align*}
    Q=\begin{pmatrix}
        -1&2n+4&-2\\
        -2 &4n+7& -4\\
        -2&4n+8&-6
    \end{pmatrix}
    \end{align*}
with signature $\sigma(Q)=-3$. The possible values of the rotation vectors are 
$$\mathbf{r}=(r,r+x,r+x+y)^T$$
where $x$ can take values $-2$, $0$, or $2$ and $y=\pm1$ from which we compute the $\de_3$-invariants.

Using the transformation lemma we write this contact surgery diagram as 
$$U(+1){\def\svgwidth{1,6ex}\,\,%% Creator: Inkscape 0.91_64bit, www.inkscape.org
%% PDF/EPS/PS + LaTeX output extension by Johan Engelen, 2010
%% Accompanies image file 'PushOff.pdf' (pdf, eps, ps)
%%
%% To include the image in your LaTeX document, write
%%   \input{<filename>.pdf_tex}
%%  instead of
%%   \includegraphics{<filename>.pdf}
%% To scale the image, write
%%   \def\svgwidth{<desired width>}
%%   \input{<filename>.pdf_tex}
%%  instead of
%%   \includegraphics[width=<desired width>]{<filename>.pdf}
%%
%% Images with a different path to the parent latex file can
%% be accessed with the `import' package (which may need to be
%% installed) using
%%   \usepackage{import}
%% in the preamble, and then including the image with
%%   \import{<path to file>}{<filename>.pdf_tex}
%% Alternatively, one can specify
%%   \graphicspath{{<path to file>/}}
%% 
%% For more information, please see info/svg-inkscape on CTAN:
%%   http://tug.ctan.org/tex-archive/info/svg-inkscape
%%
\begingroup%
  \makeatletter%
  \providecommand\color[2][]{%
    \errmessage{(Inkscape) Color is used for the text in Inkscape, but the package 'color.sty' is not loaded}%
    \renewcommand\color[2][]{}%
  }%
  \providecommand\transparent[1]{%
    \errmessage{(Inkscape) Transparency is used (non-zero) for the text in Inkscape, but the package 'transparent.sty' is not loaded}%
    \renewcommand\transparent[1]{}%
  }%
  \providecommand\rotatebox[2]{#2}%
  \ifx\svgwidth\undefined%
    \setlength{\unitlength}{49.54732089bp}%
    \ifx\svgscale\undefined%
      \relax%
    \else%
      \setlength{\unitlength}{\unitlength * \real{\svgscale}}%
    \fi%
  \else%
    \setlength{\unitlength}{\svgwidth}%
  \fi%
  \global\let\svgwidth\undefined%
  \global\let\svgscale\undefined%
  \makeatother%
  \begin{picture}(1,1.00059637)%
    \put(0,0){\includegraphics[width=\unitlength,page=1]{PushOff.pdf}}%
  \end{picture}%
\endgroup%
\,\,} \underbrace{U_2(-1){\def\svgwidth{1,6ex}\,\,%% Creator: Inkscape 0.91_64bit, www.inkscape.org
%% PDF/EPS/PS + LaTeX output extension by Johan Engelen, 2010
%% Accompanies image file 'PushOff.pdf' (pdf, eps, ps)
%%
%% To include the image in your LaTeX document, write
%%   \input{<filename>.pdf_tex}
%%  instead of
%%   \includegraphics{<filename>.pdf}
%% To scale the image, write
%%   \def\svgwidth{<desired width>}
%%   \input{<filename>.pdf_tex}
%%  instead of
%%   \includegraphics[width=<desired width>]{<filename>.pdf}
%%
%% Images with a different path to the parent latex file can
%% be accessed with the `import' package (which may need to be
%% installed) using
%%   \usepackage{import}
%% in the preamble, and then including the image with
%%   \import{<path to file>}{<filename>.pdf_tex}
%% Alternatively, one can specify
%%   \graphicspath{{<path to file>/}}
%% 
%% For more information, please see info/svg-inkscape on CTAN:
%%   http://tug.ctan.org/tex-archive/info/svg-inkscape
%%
\begingroup%
  \makeatletter%
  \providecommand\color[2][]{%
    \errmessage{(Inkscape) Color is used for the text in Inkscape, but the package 'color.sty' is not loaded}%
    \renewcommand\color[2][]{}%
  }%
  \providecommand\transparent[1]{%
    \errmessage{(Inkscape) Transparency is used (non-zero) for the text in Inkscape, but the package 'transparent.sty' is not loaded}%
    \renewcommand\transparent[1]{}%
  }%
  \providecommand\rotatebox[2]{#2}%
  \ifx\svgwidth\undefined%
    \setlength{\unitlength}{49.54732089bp}%
    \ifx\svgscale\undefined%
      \relax%
    \else%
      \setlength{\unitlength}{\unitlength * \real{\svgscale}}%
    \fi%
  \else%
    \setlength{\unitlength}{\svgwidth}%
  \fi%
  \global\let\svgwidth\undefined%
  \global\let\svgscale\undefined%
  \makeatother%
  \begin{picture}(1,1.00059637)%
    \put(0,0){\includegraphics[width=\unitlength,page=1]{PushOff.pdf}}%
  \end{picture}%
\endgroup%
\,\,} \cdots {\def\svgwidth{1,6ex}\,\,%% Creator: Inkscape 0.91_64bit, www.inkscape.org
%% PDF/EPS/PS + LaTeX output extension by Johan Engelen, 2010
%% Accompanies image file 'PushOff.pdf' (pdf, eps, ps)
%%
%% To include the image in your LaTeX document, write
%%   \input{<filename>.pdf_tex}
%%  instead of
%%   \includegraphics{<filename>.pdf}
%% To scale the image, write
%%   \def\svgwidth{<desired width>}
%%   \input{<filename>.pdf_tex}
%%  instead of
%%   \includegraphics[width=<desired width>]{<filename>.pdf}
%%
%% Images with a different path to the parent latex file can
%% be accessed with the `import' package (which may need to be
%% installed) using
%%   \usepackage{import}
%% in the preamble, and then including the image with
%%   \import{<path to file>}{<filename>.pdf_tex}
%% Alternatively, one can specify
%%   \graphicspath{{<path to file>/}}
%% 
%% For more information, please see info/svg-inkscape on CTAN:
%%   http://tug.ctan.org/tex-archive/info/svg-inkscape
%%
\begingroup%
  \makeatletter%
  \providecommand\color[2][]{%
    \errmessage{(Inkscape) Color is used for the text in Inkscape, but the package 'color.sty' is not loaded}%
    \renewcommand\color[2][]{}%
  }%
  \providecommand\transparent[1]{%
    \errmessage{(Inkscape) Transparency is used (non-zero) for the text in Inkscape, but the package 'transparent.sty' is not loaded}%
    \renewcommand\transparent[1]{}%
  }%
  \providecommand\rotatebox[2]{#2}%
  \ifx\svgwidth\undefined%
    \setlength{\unitlength}{49.54732089bp}%
    \ifx\svgscale\undefined%
      \relax%
    \else%
      \setlength{\unitlength}{\unitlength * \real{\svgscale}}%
    \fi%
  \else%
    \setlength{\unitlength}{\svgwidth}%
  \fi%
  \global\let\svgwidth\undefined%
  \global\let\svgscale\undefined%
  \makeatother%
  \begin{picture}(1,1.00059637)%
    \put(0,0){\includegraphics[width=\unitlength,page=1]{PushOff.pdf}}%
  \end{picture}%
\endgroup%
\,\,} U_2(-1)}_\text{$-n-2$}{\def\svgwidth{1,6ex}\,\,%% Creator: Inkscape 0.91_64bit, www.inkscape.org
%% PDF/EPS/PS + LaTeX output extension by Johan Engelen, 2010
%% Accompanies image file 'PushOff.pdf' (pdf, eps, ps)
%%
%% To include the image in your LaTeX document, write
%%   \input{<filename>.pdf_tex}
%%  instead of
%%   \includegraphics{<filename>.pdf}
%% To scale the image, write
%%   \def\svgwidth{<desired width>}
%%   \input{<filename>.pdf_tex}
%%  instead of
%%   \includegraphics[width=<desired width>]{<filename>.pdf}
%%
%% Images with a different path to the parent latex file can
%% be accessed with the `import' package (which may need to be
%% installed) using
%%   \usepackage{import}
%% in the preamble, and then including the image with
%%   \import{<path to file>}{<filename>.pdf_tex}
%% Alternatively, one can specify
%%   \graphicspath{{<path to file>/}}
%% 
%% For more information, please see info/svg-inkscape on CTAN:
%%   http://tug.ctan.org/tex-archive/info/svg-inkscape
%%
\begingroup%
  \makeatletter%
  \providecommand\color[2][]{%
    \errmessage{(Inkscape) Color is used for the text in Inkscape, but the package 'color.sty' is not loaded}%
    \renewcommand\color[2][]{}%
  }%
  \providecommand\transparent[1]{%
    \errmessage{(Inkscape) Transparency is used (non-zero) for the text in Inkscape, but the package 'transparent.sty' is not loaded}%
    \renewcommand\transparent[1]{}%
  }%
  \providecommand\rotatebox[2]{#2}%
  \ifx\svgwidth\undefined%
    \setlength{\unitlength}{49.54732089bp}%
    \ifx\svgscale\undefined%
      \relax%
    \else%
      \setlength{\unitlength}{\unitlength * \real{\svgscale}}%
    \fi%
  \else%
    \setlength{\unitlength}{\svgwidth}%
  \fi%
  \global\let\svgwidth\undefined%
  \global\let\svgscale\undefined%
  \makeatother%
  \begin{picture}(1,1.00059637)%
    \put(0,0){\includegraphics[width=\unitlength,page=1]{PushOff.pdf}}%
  \end{picture}%
\endgroup%
\,\,} U_{2,1}(-1),$$ 
see Figure~\ref{fig:c9}.
From the linking matrix, we deduce that $mu_{-n}$ is a generator and 
$$2\mu_{-n}=0=\mu_1=\ldots=\mu_{-n-1},$$
from which we compute
\begin{align*}
    \Gamma
       =\frac{(r_{-n}+6)}{2}\mu_{-n}=\begin{cases}
           1 \text{ if } r_{-n}=0, \pm 4,\\
            0 \text{ if } r_{-n}=\pm 2.
       \end{cases}
\end{align*}
By considering the different values of $r_{-n}$ we obtain the claimed pairs of invariants.
\begin{figure}[htbp]
    \centering
    \includegraphics[width=0.9\linewidth]{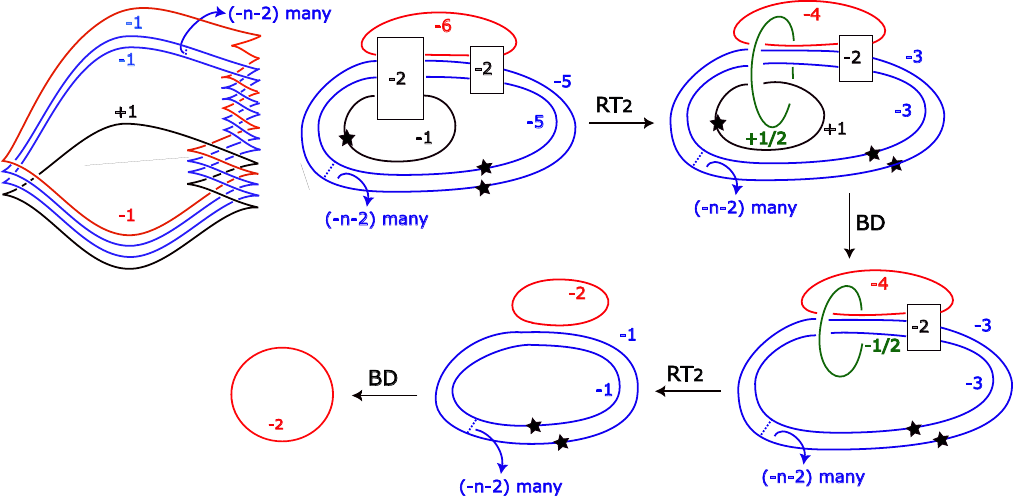}
    \caption{Case (8)}
    \label{fig:c9}
\end{figure}
 \noindent \textbf{Case (9):}  The generalized linking matrix is
 \begin{align*}
    Q=\begin{pmatrix}
       1+t&-t^2-2t&-nt-2t&t\\
        t&-t^2-t+1&-nt-2t+n+2&t-1\\
        t&-t^2-t+2&2n-2t-nt+3&t-2\\
        t&-t^2-t+2&-nt-2t+2n+4&t-4
    \end{pmatrix}
    \end{align*}
with signature $\sigma(Q)=-4$. The possible values of the rotation vectors are 
$$\mathbf{r}=(r,r+x,r+x+y,r+x+y+z)^T$$
where $x$, $y$, and $z$ each can take values $\pm1$. By substituting $t\pm r=2k+1$ as in Case $(5)$, we compute the claimed values for the $\de_3$-invariants.

Using the transformation lemma we write this contact surgery diagram as 
$$U(+1){\def\svgwidth{1,6ex}\,\,\,\,} \underbrace{U_1(-1){\def\svgwidth{1,6ex}\,\,\,\,} \cdots {\def\svgwidth{1,6ex}\,\,\,\,} U_1(-1)}_\text{$-t-2$}{\def\svgwidth{1,6ex}\,\,\,\,} \underbrace{U_{1,1} {\def\svgwidth{1,6ex}\,\,\,\,} \cdots {\def\svgwidth{1,6ex}\,\,\,\,} U_{1,1}(-1)}_{-n-2} {\def\svgwidth{1,6ex}\,\,\,\,} U_{1,1,1}(-1),$$ 
see Figure~\ref{fig:c10}.
From the linking matrix, we read off that $\mu_{-t-n-2}$ is a generator and  
\begin{align*}
0=\mu_{-t}=\ldots=\mu_{-t-n-3},\: \mu_2=\ldots=\mu_{-t-1}=\mu_{-t-n-2},\: \mu_1=t\mu_{-t-n-2}.
\end{align*}
Then from Lemma~\ref{lem:Gamma} it follows that
\begin{align*}
    \Gamma
     =\begin{cases}
         \left(\frac{r+1-t}{2}t-\frac{r_2-t}{2}t+t+\frac{(r_{-t-n-2}-t)}{2}\right)\mu_{-t-n-2}, & \textrm{ if $n$ is even,}\\
         \left(-\frac{r_2-r-3t+3+n}{2}t-\frac{t(n+t)}{2}+\frac{r_{-t-n-2}-3t}{2}\right)\mu_{-t-n-2}, & \textrm{ if $n$ and $t$ are odd,}\\
         \left(-\frac{t}{2}n+\frac{r_{-t-n-2}}{2}\right)\mu_{-t-n-2}, & \textrm{ if $n$ is odd and $t$ even.}\\
     \end{cases}
\end{align*}
Next, we write $r_2=r+x$, $r_{-t}=r+x+y$, and $r_{-t-n-2}=r+x+y+z$, for $x,y,z \in \{1,-1\}$, and substitute $t+xr=2k+1$. This yields
\begin{align*}
    \Gamma=\begin{cases}
        k\:\pmod{2} \:&\text{ if }  y=z, \\
       (k+1)\: \pmod{2} \:&\text{ if } y=-z.\\
        \end{cases}
\end{align*}
By considering the different parities of $k$, as in Case $(5)$, we obtain the claimed pairs of invariants.
\begin{figure}[htbp]
    \centering
    \includegraphics[width=1\linewidth]{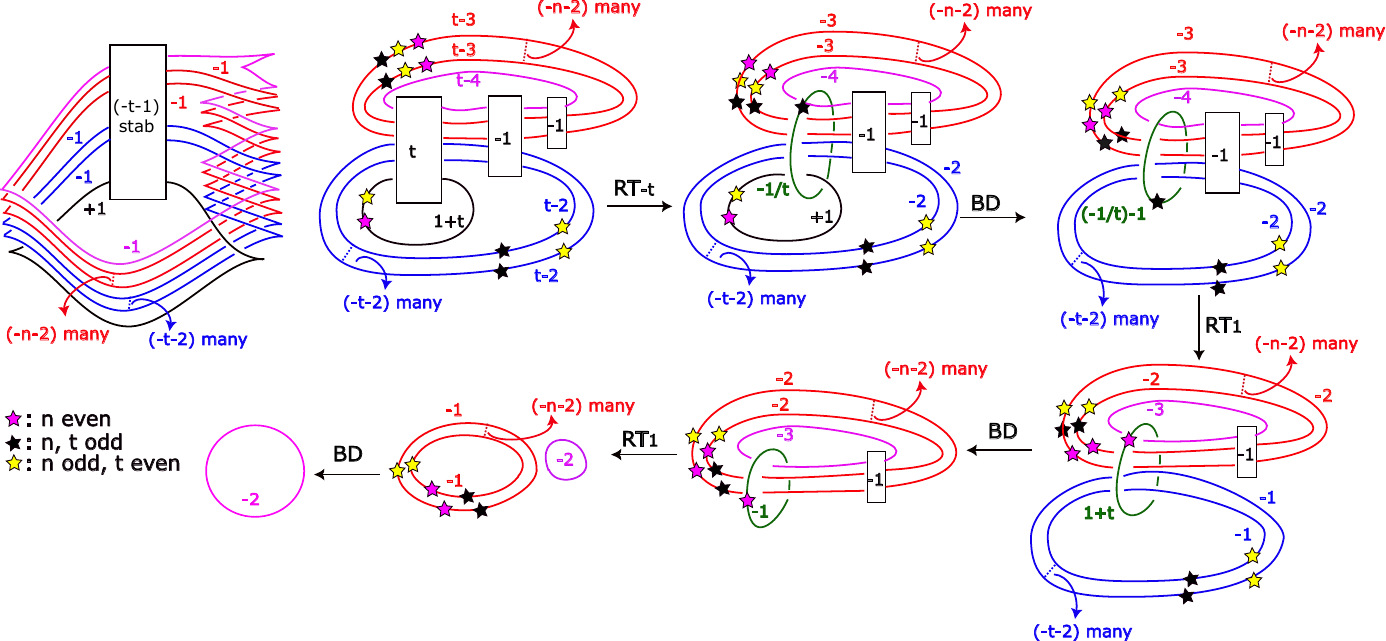}
    \caption{Case (9)}
    \label{fig:c10}
\end{figure}

 \noindent \textbf{Case (10):}  This case is contact $(+1)$-surgery on a single Legendrian knot and thus the $\de_3$-invariant is straightforward to compute.

Using the transformation lemma we write this contact surgery diagram as 
$U(+1)$ with $t=-3$, 
see Figure~\ref{fig:c11}.
This is contact $(+1)$-surgery on a single Legendrian knot and thus we compute straightforward that
\begin{align*}
    \Gamma=\begin{cases}
        0 \text{ if $r=0$, } \\
        1 \text{ if $r=\pm 2$.}
    \end{cases} 
\end{align*}
\begin{figure}[htbp]
    \centering
    \includegraphics[width=.3\linewidth]{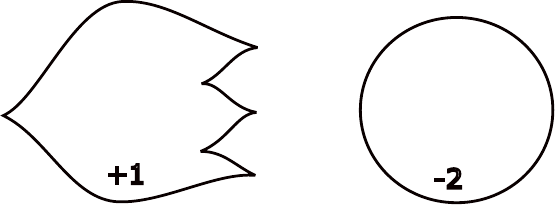}
    \caption{Case (10)}
    \label{fig:c11}
\end{figure}

 \noindent \textbf{Case (11):}  The generalized linking matrix is
 \begin{align*}
    Q=\begin{pmatrix}
       1 +t      & -t^2-3t  \\
       t  & -t^2-2t+2
    \end{pmatrix}
    \end{align*}
with signature $\sigma(Q)=-2$. The rotation vector is 
$\mathbf{r}=(r,r\pm 1)^T$. By substituting $t\pm r=2k+1$ as in Case $(5)$, we compute the claimed values for the $\de_3$-invariants.

Using the transformation lemma we write this contact surgery diagram as 
$$U(+1){\def\svgwidth{1,6ex}\,\,\,\,} \underbrace{U_1(-1){\def\svgwidth{1,6ex}\,\,\,\,} \cdots {\def\svgwidth{1,6ex}\,\,\,\,} U_1(-1)}_\text{$-t-3$},$$ 
see Figure~\ref{fig:c12}.
We get that $\mu_{-t-2}$ is a generator with relations
$$ \mu_1=t\mu_{-t-2},\: 2\mu_{-t-2}=0,\: \mu_2=\ldots=\mu_{-t-2}$$
and thus we compute
\begin{align*}
    \Gamma=\begin{cases}
        \frac{t+r+3}{2} \mu_{-t-2} &\textrm{ if $t$ is odd,}\\
        \frac{(t+r_2+2)}{2}\mu_{-t-2} &\textrm{ if $t$ is even.}
    \end{cases}
\end{align*}
By writing $r_2=r\pm1$ and $t\pm r=2k+1$, we get $\Gamma=k\,\pmod{2}$. Then by considering the different parities of $k$ as in Case $(5)$ we obtain the claimed pairs of invariants.
\end{proof}
\begin{figure}[htbp]
    \centering
    \includegraphics[width=0.9\linewidth]{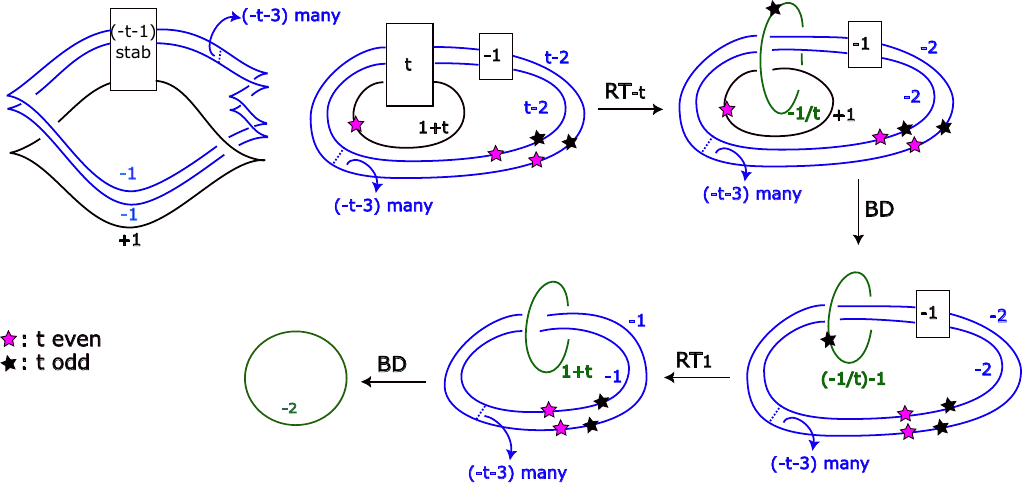}
    \caption{Case (11)}
    \label{fig:c12}
\end{figure}

\section{Tight and overtwisted contact structures}

\begin{lem}\label{lem:tight}
    Among the contact surgery descriptions from Table~\ref{tab:tab1} we get the tight contact structure $\xist$ on $\RP$ exactly in Case $(0)$ and $(1)$, and if $m=-1$ and all stabilizations of the first two knots are of the same sign also in Case $(5)$. All other contact surgery descriptions from Table~\ref{tab:tab1} yield overtwisted contact structures.
\end{lem}

\begin{proof}
    In Case $(0)$, we perform a contact $(-1)$-surgery on a single Legendrian knot in $(\mathbb{S}^3,\xist)$. Since contact surgery with a negative surgery coefficient preserves tightness~\cite{Wand} and since $\xist$ is the unique tight contact structure on $\RP$~\cite{Etnyre_lens,Honda_lens,Giroux_lens} it follows that Case $(0)$ provides a contact surgery diagram of $(\RP,\xist)$.\footnote{Alternatively, we can also argue as follows. By the transformation lemma, a negative contact surgery can be replaced by a sequence of Legendrian surgeries, which correspond to the attachment of a Weinstein $2$-handle to $D^4$~\cite{Eliashberg,Weinstein}. Thus the resulting contact structure on $\R P^3$ is fillable and thus tight by~\cite{Eliashberg_tight,Gromov}.} We also observe from Table~\ref{tab:tab2} that $\xist$ has $(\Gamma,\de_3)=(0,\frac{1}{4})$. 

    In Case $(1)$, we can use the transformation lemma as in the proof of Lemma~\ref{lem:d3} to obtain an equivalent contact $(\pm1)$-surgery diagram. In this surgery diagram, we can perform a contact handle slide of one of the $(-1)$-framed knots over the other $(-1)$-framed knot. This will yield a contact surgery diagram consisting of a $(-1)$-framed Legendrian unknot of arbitrary classical invariants together with two meridians, one framed with $(+1)$ and the other with $(-1)$. Since a $(-1)$-framed knot together with a $(+1)$-framed meridian cancel, we are left with a contact surgery diagram along a single Legendrian unknot with $t=-1$ and contact surgery coefficient $(-1)$, which represents by Case $(0)$ the tight contact structure $\xist$ on $\RP$.

    Next, we use Theorem 3.1 from~\cite{CK_contact_surgery_numbers} to deduce that any contact $r_c$-surgery along a Legendrian unknot with Thurston--Bennequin invariant $t$ is overtwisted if $0<r_c<-t$. In our setting, the contact surgery coefficient is $r_c=\frac{2}{2n+1}-t$ and thus contact $r_c$-surgery is overtwisted whenever $n<0$. It follows that the only other case that might yield a tight contact structure is Case $(5)$. To analyze this case, we compare the homotopical invariants in Case $(5)$ with the homotopical invariants of the unique tight contact structure $\xist$. If $\Gamma=1$, the contact structure is overtwisted. In the case where $\Gamma=0$ and $m<-1$, we can use $\pm2x(m+1)\geq2n(m+1)$ to estimate the $\de_3$-invariant as
    \begin{equation*}
        \de_3\geq 2m^2(2n+1)+6n+10mn+2m+\frac{1}{4}>\frac{1}{4}.
    \end{equation*}
    Thus these contact structures are all overtwisted. For $m=-1$, the $\de_3$-invariant is always $\frac{1}{4}$ and indeed in that case, the stabilizations of $U$ and $U_1$ are all of the same sign, and therefore we can use the lantern destabilization~\cite{EKS_contact_surgery_numbers} to transform that surgery diagram to the diagram from Case $(1)$ which represents $\xist$.
\end{proof}

\section{Proof of the main result and its corollaries}

Now, the main result follows by combining the previous lemmas.
    
\begin{proof}[Proof of Theorem~\ref{thm:main}]
By Lemma~\ref{lem:tight}, we see that $\cs_{\pm1}(\RP,\xist)=1$, which proves $(2)$. Since $\xist$ is the only tight contact structure on $\RP$~\cite{Etnyre_lens,Honda_lens,Giroux_lens}, we consider only overtwisted contact structures in the following.

For the other statements, we take a Legendrian knot $K$ in $(\mathbb{S}^3,\xist)$ such that contact $r_c$-surgery on $K$ yields an overtwisted contact structure on $\RP$. By Lemma~\ref{lem:contact_diagrams_RP3}, $K$ appears as one of the surgery diagrams from Table~\ref{tab:tab1}. By Lemma~\ref{lem:tight} we can ignore the cases where we get the tight contact structure, i.e.\ Cases $(0)$ and $(1)$ (and the cases $m=-1$ in the subcase of Case $(5)$, which we have omitted from Table~\ref{tab:tab2}, cf.~Proof of Lemma~\ref{lem:tight}). Thus Lemma~\ref{lem:d3} gives the classification of overtwisted contact structures on $\RP$ with $\cs=1$, which proves $(5)$.

The classification of overtwisted contact structures on $\RP$ with $\cs_{\pm1}=1$ and $\cs_{1/\Z}=1$ follows similarly. From Table~\ref{tab:tab1} we see that the only contact surgery diagrams along a single Legendrian knot with contact surgery coefficient a reciprocal integer yielding an overtwisted contact structure on $\RP$ are the ones in Cases $(2)$ and $(10)$. From Table~\ref{tab:tab2} we read off their homotopical invariants as claimed in the theorem. This shows $(3)$.

For $(4)$, i.e.\ the integer contact surgery numbers, we observe that the contact surgery coefficient $r_c=\frac{2}{2n+1}-t$ is an integer, if and only if $n=0$ or $n=-1$. If we exclude the contact surgery diagrams yielding tight contact structures, we see that this happens only in Cases $(5)$, $(10)$, and $(11)$. In Cases $(10)$ and $(11)$ we have the possible values $\left(0, 1+\frac{1}{4}\right),\,\left(1,\frac{3}{4}\right)$,
\begin{align*}
\left(0,-2m^2-4m-1+\frac{1}{4}\right),\, \left(1,-2m^2-6m-4+\frac{3}{4}\right), \textrm{ for } m\leq -1
\end{align*}
and plugging in $n=0$ in Case $(5)$ yields
\begin{align*}
\left(0,2m^2-2m+\frac{1}{4}\right),\, \left(1,2m^2-1+\frac{3}{4}\right), \textrm{ for } m\leq -1.
\end{align*}
To show $(1)$, we observe that by $(3)$ there exist two contact structures $\xi_0$, $\xi_1$ with $\Gamma(\xi_i)=i$ such that $\cs_{\pm1}(\xi_i)=1$, for $i=0,1$. Thus we get any overtwisted contact structure on $\RP$ by performing connected sums of $(\RP,\xi_i)$ with the overtwisted contact structures on $\mathbb S^3$~\cite{Ding_Geiges_Stipsicz}. But the overtwisted contact structures on $\mathbb{S}^3$ all have contact surgery number $\cs_{\pm1}\leq2$ by~\cite{EKS_contact_surgery_numbers}. It follows that any contact structure on $\RP$ has $\cs_{\pm1}\leq3$. 
\end{proof}
    
Next, we prove the corollaries of our main theorem. 

\begin{proof}[Proof of Corollary \ref{cor:unique_diag1} and \ref{cor:unique_diag2}]
        Let $K$ be a Legendrian knot in $(\mathbb{S}^3,\xist)$ such that for some $k\in\Z-\{0\}$ contact $(1/k)$-surgery on $K$ yields a contact structure on $\RP$. Then Lemma~\ref{lem:contact_diagrams_RP3} implies that $K$ appears as some surgery diagram in Table~\ref{tab:tab1}. In that table, we check that the only contact surgery diagrams along a single knot are the ones in Cases $(0)$, $(2)$, and $(10)$. By Lemma~\ref{lem:tight}, Case $(0)$ yields the standard tight contact structure $\xist$, while Case $(2)$ and Case $(10)$ yield overtwisted contact structures, and from Table~\ref{tab:tab2} we read off their homotopical invariants.
\end{proof}

Next, we prove Corollary \ref{cor:Gamma} saying that the $\Gamma$-invariant of a tangential $2$-plane field on $\RP$ is determined by its $\de_3$-invariant.

\begin{proof}[Proof of Corollary \ref{cor:Gamma}]
    Let $\xi$ be a tangential $2$-plane field on a $3$-manifold $M$ then any other $2$-plane field on $M$ with the same $spin^c$-structure (and hence same $\Gamma$-invariant) as $\xi$ can be obtained by performing connected sums with the overtwisted contact structures on $\mathbb{S}^3$~\cite{Ding_Geiges_Stipsicz}.

    On $M=\RP$, the $\Gamma$-invariant is an element of $H_1(\RP)=\Z_2$ and thus can only take two possible values. From Table~\ref{tab:tab2} we observe that there exist contact structures $\xi_0$ and $\xi_1$ on $\RP$ with $\Gamma(\xi_i)=i$ for $i=0,1$, such that $\de_3(\xi_0)\in\Z+\frac{1}{4}$ and $\de_3(\xi_1)\in\Z+\frac{3}{4}$. By the above, it follows that we get all tangential $2$-plane fields on $\RP$ by performing connected sums of $(\RP,\xi_i)$ with the overtwisted contact structures on $\mathbb{S}^3$. In our normalization of the $\de_3$-invariant the contact structures on $\mathbb{S}^3$ take exactly the integers as values and the $\de_3$-invariant is additive under connected sum. Thus we directly deduce $(1)$, $(2)$, and $(3)$. Therefore we also deduce statement $(4)$ by applying~\cite{Eliashberg_OT,Gompf_Stein}. 
\end{proof}

\begin{remark}
For deducing Corollary~\ref{cor:Gamma} we only need the computation of the $\Gamma$-invariant in Cases $(0)$ and $(2)$. Then we can deduce from Corollary~\ref{cor:Gamma} the values of the $\Gamma$-invariants just from the values of the $\de_3$-invariants. So in principle, the computations of the $\Gamma$-invariants in the proof of Theorem~\ref{thm:main} is not needed to deduce the main result. However, we included these computations here, since the results from Theorem~\ref{thm:main} can then be verified by checking that they are compatible with Corollary~\ref{cor:Gamma}. 
\end{remark}

It remains to show Corollary~\ref{cor:cs2}, which states that there exist infinitely many contact structures on $\RP$ with $\cs=2$. We will provide a more concrete version of Corollary~\ref{cor:cs2} below. For that, we will introduce the following notation. By Corollary~\ref{cor:Gamma} for an integer $d\in\Z$, we can write $\xi_{(0,d)}$ for the unique overtwisted contact structure on $\RP$ with $\Gamma(\xi_{(0,d)})=0$ and $\de_3(\xi_{(0,d)})=d+\frac{1}{4}$. Similarly, we write $\xi_{(1,d)}$ for the unique overtwisted contact structure on $\RP$ with $\Gamma(\xi_{(1,d)})=1$ and $\de_3(\xi_{(1,d)})=d+\frac{3}{4}$.   

\begin{cor}\hfill
\begin{enumerate}
    \item If $d\in\Z$ is even and negative, then $\cs(\xi_{(0,d)})>1$.
    \item If $d\in\Z$ is odd and negative, then $\cs(\xi_{(1,d)})>1$.
    \item For all $m\leq-3$, it follows that 
    \begin{align*}
         \cs(\xi_{(0,-2m^2-4m)})=2 \,\textrm{ and }\, \cs(\xi_{(1,-2m^2-6m-3)})=2.
    \end{align*}
\end{enumerate}
\end{cor}

\begin{proof}
    We start by proving $(1)$ and $(2)$. By Corollary~\ref{cor:Gamma}, there exists for every overtwisted contact structure $\xi$ on $\RP$ a unique pair $(i,d)\in\Z_2\times \Z$ such that $\xi$ is contactomorphic to $\xi_{(i,d)}$. If $\cs(\xi_{(i,d)})=1$ then its pair of $\Gamma$- and $\de_3$-invariant appears in Table~\ref{tab:tab2}. By analyzing that table we will prove the statements.

    First, we prove $(1)$. If $\Gamma(\xi_{(i,d)})=i=0$, then we see from Table~\ref{tab:tab1} that $d$ is odd in all cases except Case $(5)$, in which it is always even. But in the proof of Lemma~\ref{lem:tight} we have estimated $d>0$ in Case $(5)$. Thus it follows that if $d\in\Z$ is even and negative, then it cannot be obtained by a single rational contact surgery from $(\mathbb S^3,\xist)$. 

    For $(2)$, we proceed analogously. We observe that if $i=1$, then $d$ is either $1$, even, or takes the odd values in Case $(5)$. In Case $(5)$, we estimate $d$ as
    \begin{align*}
        d&\geq 2m^2(2n+1)+n(4m+1)-1+(2m+1)n\\
        &=(2m^2+1)(2n+1)+6mn-2>0,
    \end{align*}
    which implies Statement $(2)$.

    To prove Statement $(3)$, we observe that by $(1)$ and $(2)$ these families of contact structures have $\cs>1$. To write them as surgery on a $2$-component link we consider the overtwisted contact structure on $\mathbb S^3$ with $\de_3$-invariant $1$, which can be obtained by a contact $(+1)$-surgery along a Legendrian unknot $U$ with $\tb=-2$. Then we take the split union of $U$ and the contact surgery diagrams from Case $(11)$. This corresponds to taking the connected sum of the contact structures on $\RP$ from Case $(11)$ and the overtwisted contact structure on $\mathbb S^3$ with $\de_3=1$. Since in our normalization, the $\de_3$-invariant behaves additive, we get the claimed contact structures by surgery along a $2$-component Legendrian link from $(\mathbb{S}^3,\xist)$.
\end{proof}

  \let\MRhref\undefined
  \bibliographystyle{hamsalpha}
  \bibliography{ref.bib}

\providecommand{\bysame}{\leavevmode\hbox to3em{\hrulefill}\thinspace}
\providecommand{\MR}{\relax\ifhmode\unskip\space\fi MR }
% \MRhref is called by the amsart/book/proc definition of \MR.
\providecommand{\MRhref}[2]{%
  \href{http://www.ams.org/mathscinet-getitem?mr=#1}{#2}
}
\providecommand{\href}[2]{#2}
\begin{thebibliography}{KMOS07}

\bibitem[CEK24]{casals2021stein}
R.~Casals, J.~B. Etnyre, and M.~Kegel, \emph{Stein traces and characterizing slopes}, Math. Ann. \textbf{389} (2024), 1053--1098. \MR{4745731}

\bibitem[CK24]{CK_contact_surgery_numbers}
R.~Chatterjee and M.~Kegel, \emph{Contact surgery numbers of {$\Sigma(2,3,11)$ and $L(4m+3,4)$}}, 2024, \href{http://arxiv.org/abs/2404.18177}{arXiv:2404.18177}.

\bibitem[DG01]{Ding_Geiges_torus_bundles}
F.~Ding and H.\vspace{0cm} Geiges, \emph{Symplectic fillability of tight contact structures on torus bundles}, Algebr. Geom. Topol. \textbf{1} (2001), 153--172. \MR{1823497}

\bibitem[DG04]{Ding_Geiges_Surgery}
F.~Ding and H.~Geiges, \emph{A {L}egendrian surgery presentation of contact 3-manifolds}, Math. Proc. Cambridge Philos. Soc. \textbf{136} (2004), 583--598. \MR{2055048}

\bibitem[DGS04]{Ding_Geiges_Stipsicz}
F.~Ding, H.~Geiges, and A.~I. Stipsicz, \emph{Surgery diagrams for contact 3-manifolds}, Turkish J. Math. \textbf{28} (2004), 41--74. \MR{2056760}

\bibitem[DK16]{Durst_Kegel_rot_surgery}
S.~Durst and M.~Kegel, \emph{Computing rotation and self-linking numbers in contact surgery diagrams}, Acta Math. Hungar. \textbf{150} (2016), 524--540. \MR{3568107}

\bibitem[EF98]{Eliashberg_Fraser}
Y.~Eliashberg and M.~Fraser, \emph{Classification of topologically trivial {L}egendrian knots}, Geometry, topology, and dynamics ({M}ontreal, {PQ}, 1995), CRM Proc. Lecture Notes, vol.~15, Amer. Math. Soc., Providence, RI, 1998, pp.~17--51. \MR{1619122}

\bibitem[EKO23]{EKS_contact_surgery_numbers}
J.~B. Etnyre, M.~Kegel, and S.~Onaran, \emph{Contact surgery numbers}, J. Symplectic Geom. \textbf{21} (2023), 1255--1333. \MR{4767855}

\bibitem[Eli88]{Eliashberg_tight}
Y.~Eliashberg, \emph{Three lectures on symplectic topology in {C}ala {G}onone. {B}asic notions, problems and some methods}, vol.~58, 1988, Conference on Differential Geometry and Topology (Sardinia, 1988), pp.~27--49. \MR{1122856}

\bibitem[Eli89]{Eliashberg_OT}
\bysame, \emph{Classification of overtwisted contact structures on {$3$}-manifolds}, Invent. Math. \textbf{98} (1989), 623--637. \MR{1022310}

\bibitem[Eli90]{Eliashberg}
\bysame, \emph{Topological characterization of {S}tein manifolds of dimension {$>2$}}, Internat. J. Math. \textbf{1} (1990), 29--46. \MR{1044658}

\bibitem[Etn00]{Etnyre_lens}
J.~B. Etnyre, \emph{Tight contact structures on lens spaces}, Commun. Contemp. Math. \textbf{2} (2000), 559--577. \MR{1806947}

\bibitem[Gei08]{Geiges_book}
H.~Geiges, \emph{An introduction to contact topology}, Cambridge Studies in Advanced Mathematics, vol. 109, Cambridge University Press, Cambridge, 2008. \MR{2397738}

\bibitem[Gir00]{Giroux_lens}
E.~Giroux, \emph{Structures de contact en dimension trois et bifurcations des feuilletages de surfaces}, Invent. Math. \textbf{141} (2000), 615--689. \MR{1779622}

\bibitem[Gom98]{Gompf_Stein}
R.~E. Gompf, \emph{Handlebody construction of {S}tein surfaces}, Ann. of Math. (2) \textbf{148} (1998), 619--693. \MR{1668563}

\bibitem[Gro85]{Gromov}
M.~Gromov, \emph{Pseudo holomorphic curves in symplectic manifolds}, Invent. Math. \textbf{82} (1985), 307--347. \MR{809718}

\bibitem[GS99]{GS99}
R.~E. Gompf and A.~I. Stipsicz, \emph{{$4$}-manifolds and {K}irby calculus}, Graduate Studies in Mathematics, vol.~20, American Mathematical Society, Providence, RI, 1999. \MR{1707327}

\bibitem[Hon00]{Honda_lens}
K.~Honda, \emph{On the classification of tight contact structures {I}}, Geom. Topol. \textbf{4} (2000), 309--368. \MR{1786111}

\bibitem[Keg17]{phdthesis}
M.~Kegel, \emph{Legendrian knots in surgery diagrams and the knot complement problem}, Ph.D. thesis, Universität zu K\"oln, 2017.

\bibitem[Keg18]{Legendrian_knot_complement}
M.~Kegel, \emph{The {L}egendrian knot complement problem}, J. Knot Theory Ramifications \textbf{27} (2018), 1850067, 36. \MR{3896311}

\bibitem[KMOS07]{KMOS}
P.~Kronheimer, T.~Mrowka, P.~Ozsv\'ath, and Z.~Szab\'o, \emph{Monopoles and lens space surgeries}, Ann. of Math. (2) \textbf{165} (2007), 457--546. \MR{2299739}

\bibitem[KO23]{surgery_graph}
M.~Kegel and S.~Onaran, \emph{Contact surgery graphs}, Bull. Aust. Math. Soc. \textbf{107} (2023), 146--157. \MR{4531699}

\bibitem[OS04]{Ozbagci_Stipsicz_book}
B.~Ozbagci and A.~I. Stipsicz, \emph{Surgery on contact 3-manifolds and {S}tein surfaces}, Bolyai Society Mathematical Studies, vol.~13, Springer-Verlag, Berlin; J\'{a}nos Bolyai Mathematical Society, Budapest, 2004. \MR{2114165}

\bibitem[Wan15]{Wand}
A.~Wand, \emph{Tightness is preserved by {L}egendrian surgery}, Ann. of Math. (2) \textbf{182} (2015), 723--738. \MR{3418529}

\bibitem[Wei91]{Weinstein}
A.~Weinstein, \emph{Contact surgery and symplectic handlebodies}, Hokkaido Math. J. \textbf{20} (1991), 241--251. \MR{1114405}

\end{thebibliography}

\end{document}